\documentclass[a4paper]{amsart}
\usepackage{amssymb}
\usepackage[polutonikogreek, english]{babel}
\usepackage{graphicx}
\usepackage{pstricks}
\usepackage{enumerate}
\theoremstyle{plain}
\newtheorem{theorem}{Theorem}
\newtheorem{corollary}[theorem]{Corollary}
\newtheorem{lemma}[theorem]{Lemma}
\newtheorem{axiom}{Axiom}
\newtheorem{remark}{Remark}
\newtheorem{definition}{Definition}
\author{Wim Veldman}
\address{Institute for Mathematics, Astrophysics and Particle Physics, Faculty of Science, Radboud University,
Postbus 9010, 6500 GL Nijmegen, the Netherlands}
\email{W.Veldman@science.ru.nl}

\begin{document}
\title{On some of Brouwer's axioms}
\maketitle

\section{Introduction} \subsection{ Bishop's disagreement with Brouwer}\label{SS:begin} E. Bishop, who founded {\it constructive analysis},  has an ambivalent attitude towards L.E.J. Brouwer, who, a generation earlier, began {\it intuitionistic mathematics}. \footnote{This paper has been written as a contribution for the {\it Handbook of Constructive Mathematics}, soon to appear with Springer Verlag. Unfortunately, it was completed so late that the editors were unable to consider it for inclusion in this handbook.} 

On the one hand, Bishop recognizes that Brouwer was the first to raise his voice against the disturbing fact that many mathematical theorems lack constructive content.  He also judges  that Brouwer  made a good beginning with the necessary reconstruction of parts of mathematics. He agrees with him that disjunction and the existential quantifier should be interpreted constructively and that, as a consequence, the principle of the excluded third $X\;\vee\;\neg X$ should be rejected.

On the other hand, he thinks Brouwer went the wrong way by introducing  `{\it semi-mystical elements}' into mathematics in order to prove the theorem that every (effectively total) function from $[0,1]$ to $\mathcal{R}$ is uniformly continuous. 

This theorem may be split up into two statements:
\begin{enumerate}[\upshape 1.]\item Every function from $[0,1]$ to $\mathcal{R}$ is pointwise continuous. \item Every pointwise continuous function from $[0,1]$ to $\mathcal{R}$ is uniformly continuous. \end{enumerate} 

The first statement is a consequence of {\it Brouwer's Continuity Principle} and the second one follows from his {\it Fan Theorem}.

Bishop rejects Brouwer's argument  for the first conclusion. He says that a set like `{\it the set  of all functions from $[0,1]$ to $\mathcal{R}$}' seems to have little practical interest, \cite[p. 67]{bishopbridges85}, meaning probably, that one does not need general statements about such functions like the statement 1. 

In addition, he decides not to use   the notion of a {\it pointwise} continuous function, see \cite[page 66]{bishopbridges85}.
He {\it defines} a function from $\mathcal{R}$ to $\mathcal{R}$ to be {\it continuous} if and only if  the function is {\it uniformly} continuous on every closed interval $[a,b]$, see \cite[Chapter 2, Definition 4.5]{bishopbridges85}.\footnote{This strategy may work for a locally compact space like $\mathcal{R}$, but it does not help if one starts thinking on, for instance, continuous functions from Baire space $\mathcal{N}$ to $\omega$.} 

The statement 2 then becomes a tautology.

Brouwer's Fan Theorem also implies that, given a closed interval $[a,b]$ and a sequence $f_0, f_1, \ldots$ of functions from $[a,b]$ to $\mathcal{R}$ that converges pointwise to a a function $f$ from $[a,b]$ to $\mathcal{R}$ will converge {\it uniformly} to $f$. Bishop avoids the notion of pointwise convergence, see  \cite[Chapter 2, Definition 4.7]{bishopbridges85}. 

Brouwer derived his Fan Theorem from a much stronger statement: the Bar Theorem. Bishop does not discuss this stronger statement.

\subsection{Going back to Brouwer's basic assumptions} 

Brouwer's arguments for the statements 1 and 2 necessarily are of a philosophical rather  than a mathematical nature. One should not put them aside  as {\it non-mathematical} and, {\it therefore,  not worth a mathematician's attention}. 

Besides, these arguments might have  consequences that go further than the statements 1 and 2 in Subsection \ref{SS:begin} and are possibly important for the development of constructive mathematics.

 Brouwer is trying  to redefine the  {\it game of mathematics} and make it a better game than it has been up to now.   The starting points of the game have to be agreed upon, and  are a topic of ongoing discussion. They should be called {\it axioms}, although Brouwer avoids this expression.  Brouwer, somewhat misleadingly, presents his axioms as being beyond doubt and speaking for themselves, and he does not distinguish the arguments supporting them from his more mathematical arguments.  

\subsection{The three varieties} Constructive mathematics is often dexcribed as having three varieties: $BISH$, i.e. Bishop style constructive mathematics, $INT$, i.e. intuitionistic mathematics and $RUSS$, i.e. Russian style constructive mathematics, where the notion of a computable function sets the scene, see \cite{bridgesrichman87}. This division is problematic.
 
 It is  difficult to make sense of the slogan: \begin{quote} {\it In essence, BISH is simply mathematics with intuitionistic logic} \end{quote}
 
 explained and defended in the preface to \cite{bridgesvita06}.

The slogan suggests that, for a classical mathematician who decides to join the constructive enterprise, mathematical objects remain the same in spite of the fact that he is changing  the language he uses for describing them. This suggestion is wrong. It sounds as if the objects are and remain there, like animals in a zoo, while we, the visitors, start  babbling about them in a foreign tongue.

More fundamentally, it does not seem to make sense to say: \begin{quote} `{\it We may prove statement $X$ classically as well as intuitionistically}' \end{quote}
 
 as the statement $X$ does not {\it mean} the same intuitionistically as it does classically. Classical and intuitionistic mathematicians do not speak the same language. 
 
 Of course, if we {\it formalize} mathematics we may prove combinatorial facts of the form:\begin{quote} Formula $\varphi$ is provable in the classical as well as in the intuitionistic system. \end{quote}
 
 But the {\it meaning} of the formula would change with the user of the formalisms. 
 
 It probably would be better to say\begin{quote} {\it $BISH$ is part of $INT$. Only, an intuitionistic  result is reckoned to belong to $BISH$ if it can be proven without making use of either the Continuity Principle, the Fan Theorem or the Bar Theorem.} \end{quote}
 
The third variety,  $RUSS$, arises from $BISH$ by adding the assumption that every real is given by an algorithm in the Church-Turing sense. The study of computable functions, however, is part of  intuitionistic mathematics, in fact a part of intuitionistic number theory,  and not an alternative for intuitionistic mathematics, see \cite{veldman2011b}. Unfortunately, the theory of computable functions, up to now,  is mostly done from a classical point of view. 
  
  So, we propose the following picture:\begin{center}
  
  $RUSS\subseteq BISH\subseteq INT$ \end{center}
  
  where each of the three $RUSS, BISH$ and $INT$ is considered as a body of proven intuitionistically meaningful results.\footnote{The picture is slightly inaccurate as the intuitionistic mathematian does not want to use Markov's Principle, unlike some memebers of the Russian school.}
  
  $CLASS$, the collection of results obtained classically, does not occur in the picture. The constructive mathematician has no immediate understanding of results in $CLASS$. The inclusion $BISH\subseteq CLASS$ makes no sense. 
 
\subsection{The need for axioms} The constructive rebuilding of mathematics forces one to rethink radically the meaning of mathematical statements. One should keep in mind that the meaning of a mathematical statement is ultimately given by  its proof. The proof should be seen as an {\it explication} or {\it unfolding} of the meaning of the statement.

In the proof we may refer to constructions we did earlier but, sometimes,  we come to invoke `axioms'. An axiom might be seen as a stipulation on the meaning of some of the expressions we are using in our language. As such, axioms come very close to `definitions'. 

Axioms arise from situations we want to consider as {\it canonical}, as setting an example.

For instance, once we have seen Euclid's proof that there exist infinitely many primes, we see how  we want to prove such a thing, and we {\it define}: `$X\subseteq \omega$ is {\it infinite} if one is able to indicate an algorithm  providing, given any finite list $(n_0, n_1, \ldots, n_{k-1})$ of natural numbers, an element of $X$ not occurring in the list.'

Or, observing our own use of the disjunction, we may decide to lay down: `a proof of $X \vee Y$ should consist either in a proof of $X$ or in a proof of $Y$.'

Setting up our common mathematical discourse, we have to discuss carefully the question   which principles  deserve the status of a {\it canonical starting point} for our arguments, i.e., the status of an \textit{axiom}.

An axiom  is not a truth solid as a rock that is beyond doubt and discussion. On the contrary, it is a proposal that invites and shapes discussion. It might be compared to an {\it hypothesis} or a {\it thought experiment}. Like G\"odel suggested in the context of axiomatic  set theory, an axiom may prove its value when we decide to follow its lead and try to see what we find when using it. 

It is an illusion however that one can do without axioms, and it is wrong to condemn arguments defending  axioms as philosophical and unmathematical and therefore not relevant.

Brouwer's Continuity Principle, the Fan Theorem and the Bar Theorem, to be discussed later in this paper, may be seen as  agreements on the meaning of certain statements of the form $\forall x \exists y[xRy]$.

\subsection{The contents of the paper} In this paper, we  introduce three main axioms of intuitionistic mathematics: the Continuity Principle, the Fan Theorem and the Thesis on Bars in $\mathcal{N}$.  We briefly discuss their plausibility and then show some of their applications in intuitionistic mathematics.

Apart from this introductory Section, the paper contains eight Sections.

In Section \ref{S:cp}, we explain {\it Brouwer's Continuity Principle}. We show its famous consequence: the pointwise continuity of real functions. We also introduce axioms of countable choice.

In Section \ref{S:bht}, we give a more sophisticated application of Brouwer's Continuity Principle: the proof of the {\it Borel Hierarchy Theorem},  see \cite[Chapter 9]{veldman1981} and \cite[Section 7]{veldman2008a}.   The Continuity Principle is also important at other points in the development of  intuitionistic {\it descriptive set theory}.   It is crucial for proving the fine structure of the hierarchy, see \cite{veldman2009b},  and a strong formulation of the  Principle leads to the collapse of the projective hierarchy, see \cite[Chapter 14]{veldman1981} and \cite[Section 7]{veldman2019}.

In Section \ref{S:ft}, we explain the {\it Fan Theorem}. We give its first and most famous application: functions with domain $[0,1]$ that are pointwise continuous are also uniformly continuous.

In Section \ref{S:measure} we sketch the intuitionistic development of the theory of measure and integration, as begun by Brouwer and developed further by some students of Heyting, see \cite[Chapter VI]{heyting56}. Bishop, see \cite[Chapter 6]{bishopbridges85} decided not to follow  Brouwer's lead in this field, probably out of fear of the Fan Theorem.

In Section \ref{S:barth}, we explain {\it Brouwer's Thesis on bars in $\mathcal{N}$}. Brouwer introduced it for proving the Fan Theorem but it has stronger consequences than that. As an example, we prove the equivalence of two definitions of the class of decidable and well-founded subsets of the set of the rationals.  

In Section \ref{S:aft}, we explain the {\it Almost-Fan Theorem}. Like the Fan Theorem itself, the Almost-Fan Theorem follows from the Bar Theorem. For a classical spectator, the Almost-Fan Theorem is difficult to distinguish from the Fan Theorem itself. The Almost-Fan Theorem implies the Fan Theorem but it is a stronger statement and  we will see it has other important consequences too. 

In Section \ref{S:notat}, we explain our notations. 

\section{Axioms of  Continuity and Choice}\label{S:cp}

\subsection{The Continuity Principle}\begin{axiom}[Brouwer's Continuity Principle]\label{ax:bcp} For every relation $R\subseteq \omega^\omega\times\omega$, \\if $\forall \alpha \exists n[\alpha Rn]$, then $\forall \alpha\exists m\exists n \forall \beta[\overline\alpha m \sqsubset \beta\rightarrow \beta Rn]$. \end{axiom}

The Principle came on the scene in 1918. In \cite[p. 13]{brouwer18}, Brouwer explains  that there can not exist an injective function from the set $\mathcal{N}$ of all infinite sequences of natural numbers  to the set $\omega$ of the natural numbers. He  says: if we should have such a function, say $f$, and an element $\alpha$ of $\mathcal{N}$ is given, then $f(\alpha)$, the value of $f$ at $\alpha$, would have to be decided upon at a point of time at which  only finitely many values of $\alpha$, say, $ \alpha(0), \alpha(1), \ldots, \alpha(m-1),$ {\it have become known}. So all infinite sequences $\beta$ that take the same values as $\alpha$ on the arguments $0,1, \ldots, m-1$ would be allotted the same value by $f$ as $\alpha$, and $f$ would be non-injective. 

The principle Brouwer is using  here, for the first time, is  what we now call his Continuity Principle. In our  formulation of the principle,   the starting point seems to be  a little bit more general: $\forall \alpha\exists n[\alpha R n]$, but we interpret this as: we  have a {\it method} to assign to {\it any} given $\alpha$ a suitable $n$, i.e.   there exists a {\it function} $f$ from $\mathcal{N}$ to $\omega$ such that $\forall \alpha[\alpha R f(\alpha)]$. 

Brouwer is clearly imagining that the sequence \begin{quote} $\alpha=\alpha(0), \alpha(1), \ldots$ \end{quote} is given to us only step by step.  We learn its values one by one and have no information on the development of the sequence as a whole.  

 The first and foremost example of   an infinite sequence is the  sequence \begin{quote} $\alpha(0)=0, \alpha(1)=1, \alpha(2)=2,\ldots$ \end{quote} of the natural numbers themselves. Even this infinite sequence is growing {\it step by step} and never fully realized. Its construction  is a job that we started to carry out and always will work on without ever finishing it. As one sometimes says, it is a {\it project} rather than an {\it object}.

Another infinite sequence that grows step by step is the sequence \begin{quote} $\alpha(0)= 1, \alpha(1)= 4, \alpha(2) =1, \ldots$ \end{quote} of the decimals of $\pi$. 

These two examples do not give us a complete picture. 
Although we construct the natural numbers one by one and  also calculate the decimals of $\pi$ one by one, we  in both cases have a key for finding  all  values that precludes surprises. Brouwer calls such algorithmic infinite sequences {\it lawlike} sequences.

We want to make room for other infinite sequences too.
We admit the possibility that the values of $\alpha$ are disclosed to us, or chosen by us, one by one, and that we do not know any finite algorithm that determines the values of $\alpha$.

Every  individual member of  the {\it set} $\mathcal{N}$ of the infinite sequences of natural numbers may be imagined to be  always under construction and never completed.

Brouwer thus saw that `{\it sets}' like $\mathcal{N}$ deserve careful treatment.

Cantor's idea that a set is the result of collecting certain already existing  objects, to be called its elements, into a new whole, is wrong. In this picture, the elements of the set are `earlier there' than the set itself. The intuitionistic mathematician proposes to view a set like $\mathcal{N}$ as a realm of possibilities. A set is like a musical instrument, on which many tunes will be played in the future.

One should keep in mind that it does not make sense, intuitionistically, that something is the case without our knowing so. The meaning of  a statement $$\forall \alpha\exists n[\alpha Rn]$$ must be that I see that I am able to {\it effectively} find a suitable $n$ to {\it any possible} $\alpha$.  $\alpha$ may be given to me, or created by myself making free choices,  value by value without,  at any point of time, any further information on the whole of its course, and, what is very important, even if $\alpha$ is not given in this way but, somehow, at one stroke, {\it $\alpha$ might have been given to me in this way}.  

In this way, we defend the claim that, if $\forall \alpha\exists n[\alpha Rn]$,  then, given any $\alpha$, one must be able to come up with a suitable $n$ for $\alpha$ knowing only finitely many  values of $\alpha$.

 Note that    an infinite sequence that I am creating freely, value by value,   may turn out, {\it in the end}, to be a `simple' one that admits of a finite description. The decimal expansion of $\pi$, for instance,  may be the result of an infinite sequence of free choices.
 
 \subsection{A first application}
\begin{theorem}\label{T:realcont} $\;$\\For every relation $R\subseteq \mathcal{R}\times \omega$, if $\forall x \exists n[x Rn]$, then $\forall x\exists m  \exists n \forall y[|x-y|<\frac{1}{2^m}\rightarrow yRn]$. \end{theorem}

\begin{proof} Assume $\forall x\in \mathcal{R}\exists n[x Rn]$.  We prove the promised conclusion for the case $x=0$, i.e. $\exists m \exists n[\forall y[|y|<\frac{1}{2^m}\rightarrow yRn]$. The general case is proven similarly.

\smallskip
We first define a function $\varphi$ from $\mathcal{N}$ to $\mathcal{N}$.

\smallskip
Let $q_0:=0, q_1, q_2, \ldots$ be an enumeration of the rationals.

Let $\alpha$ be given. \\We have to define the infinite sequence $\varphi|\alpha= (\varphi|\alpha)(0), (\varphi|\alpha)(1), \ldots$ and we do so as follows.

We define $(\varphi|\alpha)(0)=0$. \\ For each $n$,  {\it if}  $|q_{\alpha(n+1)}-q_{(\varphi|\alpha)(n)}|\le \frac{1}{2^n}$, we define $(\varphi|\alpha)(n+1)= \alpha(n+1)$, and,  \\{\it if not}, we define $(\varphi|\alpha)(n+1)=(\varphi|\alpha)(n)$. 

Note that, for each $\alpha$, the sequence $q_{(\varphi|\alpha)(0)}, q_{(\varphi|\alpha)(1)}, \ldots$ converges.

Also note that, if, for all $n$, $|q_{\alpha(n+1)}-q_{\alpha(n)}|\le \frac{1}{2^n}$, then $\varphi|\alpha=\alpha$.

We define, for each $\alpha$, $x_\alpha:=\lim_{n\rightarrow\infty} q_{(\varphi|\alpha)(n)}$.

Note: $\forall \alpha \exists n[x_\alpha Rn]$.\\ Apply the Continuity Principle and find $m,n$ such that  $\forall\alpha[\underline{\overline 0}m \sqsubset \alpha \rightarrow x_\alpha Rn]$. 
\\Note that,  for each $x$ in $[-\frac{1}{2^m}, \frac{1}{2^m}]$ there exists $\alpha$ such that $\overline{\underline 0}m\sqsubset \alpha$ and $x=x_\alpha$. 

 Conclude: $\forall x \in [-\frac{1}{2^m}, \frac{1}{2^m}][ x R n]$. \end{proof}

\begin{corollary}\label{C:contreal} Every function from $\mathcal{R}$ to $\mathcal{R}$ is pointwise continuous. \end{corollary}

\begin{proof} Let a function $f$ from $\mathcal{R}$ to $\mathcal{R}$ be given. 
\\Let $q_0, q_1, \ldots$ be an enumeration of the rationals.
Let $p$ be given.  \\Note $\forall x \in \mathcal{R}\exists n[|f(x) -  q_n| < \frac{1}{2^{p+1}}]$.  \\Let $x$ be given. Applying Theorem \ref{T:realcont}, find $m$, $n$ such that \\$\forall y \in \mathcal{R}[|x-y|<\frac{1}{2^m} \rightarrow |f(x) - q_n|< \frac{1}{2^{p+1}}]$. \\Conclude: $\forall y \in \mathcal{R}[|f(x)-f(y)|< \frac{1}{2^m}\rightarrow|f(x)-f(y)|<\frac{1}{2^p}]$. \\We thus see: $\forall p\exists m\forall y \in \mathcal{R}[|f(x)-f(y)|< \frac{1}{2^m}\rightarrow|f(x)-f(y)|<\frac{1}{2^p}]$,\\
 i.e. $f$ is continuous at $x$. \end{proof}
We thus see that Corollary \ref{C:contreal} follows from the Continuity Principle alone. Brouwer seems to have thought the Fan Theorem is needed for this result, see \cite{brouwer27}, \cite{veldman1982} and \cite{veldman2001}.\footnote{As was observed by Serge Bozon, there is a mistake in the proof of \cite[Theorem 2.6]{veldman2001}.  Here is a correction. After the first two sentences of the proof, continue as follows: `Find $p$ such that $|\alpha-q_{\alpha(n-1)}|+1/2^p < 1/2^n$. Then, for every canonical real number $\beta$, if $|\alpha-\beta|<\frac{1}{2^p}$, there exists a canonical real $\gamma$ such that $\beta=_\mathbb{R}\gamma$ and, for all $i<n$, $\alpha(i)=\gamma(i)$, and, therefore, $f(\beta)=f(\gamma)$ and $|f(\alpha-f(\beta)|<\frac{1}{2^m}$.'}
\subsection{Spreads} The `set'  $\omega^\omega$ is an example of a kind of sets that are called {\it spreads}.

A spread is  given by a {\it spread-law} $\beta$. The elements of the spread will be certain infinite sequences of natural numbers that, in general, are unfinished and are created step-by-step. The spread-law is there to regulate the process of defining elements of the spread. It informs me, whenever I have completed a finite initial part \begin{quote} $\alpha(0), \alpha(1), \ldots, \alpha(n-1)$ \end{quote}  of  an element $\alpha$ of the spread, which numbers I  may choose for the next value: $\alpha(n)$.  

The spread-law $\beta$ itself is an element of $\omega^\omega$. For every $s$, the informal meaning of `$\beta(s)=0$' is: `the finite sequence coded by $s$ is admitted by $\beta$'. $\beta$ has to satisfy the following condition:
\begin{quote} $\forall s[\beta(s)=0\leftrightarrow \exists n[\beta(s\ast\langle n \rangle)=0]]$ \end{quote}

The `set' consisting of all $\alpha$ such that $\forall n[\beta(\overline \alpha n)=0]$ will be called $\mathcal{F}_\beta$. $\mathcal{F}_\beta$ is  the {\it spread} determined by the spread-law $\beta$. 

The condition just imposed on the spread-law $\beta$ guarantees that, when I am creating an element $\alpha$ of $\mathcal{F}_\beta$, then, at every stage $n$, having chosen \\$\alpha(0), \alpha(1), \ldots, \alpha(n-1)$, I am able to decide, for each $m$, if I may define $\alpha(n):=m$,  and: there will be at least one such $m$, i.e. I never will get `stuck'. 
\begin{theorem}[Brouwer's Continuity Principle, extended to spreads]\label{T:bcpext}$\;$\\ Let $\beta$ be a spread-law.
For every relation $R\subseteq \mathcal{F}_\beta\times\omega$, \\if $\forall \alpha \in \mathcal{F}_\beta\exists n[\alpha Rn]$, then $\forall \alpha\in \mathcal{F}_\beta\exists m\exists n \forall \gamma\in \mathcal{F}_\beta[\overline\alpha m \sqsubset \gamma\rightarrow \gamma Rn]$. \end{theorem}

\begin{proof} One might defend this theorem like Axiom \ref{ax:bcp} itself, as a more general formulation of the same principle.

One may also derive the Theorem from the Axiom, as follows.

Let a spread-law $\beta$ be given.  

\smallskip We define $\rho:\omega^\omega\rightarrow \omega^\omega$ such that, \\for each $\alpha$, for each $n$, {\it if $\beta\bigl(\overline \alpha(n+1)\bigr)=0$}, then $(\rho|\alpha)(n)=\alpha (n)$, and, \\ {\it if not}, then $(\rho|\alpha)(n)=\mu k[\beta\bigl(\overline{\rho|\alpha}n\ast\langle k \rangle\bigr)=0]$.

Then $\forall \alpha[\rho|\alpha\in \mathcal{F}_\beta]$ and $\forall \alpha \in \mathcal{F}_\beta[\rho|\alpha=\alpha]$.

The function $\rho$ is called a {\it retraction} of $\omega^\omega$ onto $\mathcal{F}_\beta$.

\smallskip Now assume $\forall \alpha \in \mathcal{F}_\beta\exists n[\alpha R n]$.

Then $\forall \alpha \exists n[(\rho|\alpha)Rn]$.

 By Axiom \ref{ax:bcp}, for any given $\alpha$ in $\mathcal{F}_\beta$, one may find $m,n$ such that \\$\forall \gamma[\overline \alpha m\sqsubset \gamma \rightarrow (\rho|\gamma)Rn]$ and, therefore, $\forall \gamma \in \mathcal{F}_\beta[\overline\alpha m \sqsubset \gamma \rightarrow \gamma R n]$.

Conclude: $\forall \alpha\in \mathcal{F}_\beta\exists m\exists n \forall \gamma\in \mathcal{F}_\beta[\overline\alpha m \sqsubset \gamma\rightarrow \gamma Rn]$.
\end{proof}
\subsection{Axioms of Countable Choice}

\begin{axiom}[First Axiom of Countable Choice]\label{ax:firstchoice} For every relation $R\subseteq \omega\times\omega$, \\if $\forall m \exists n [mRn]$, then $\exists \alpha\forall m[m R\alpha(m)]$. \end{axiom}
This axiom seems to be a good proposal as we decided to allow the possibility that an infinite sequence $\alpha$ is created step by step. 

Assume $\forall m\exists n[mRn]$. We first find $n$, such that $0Rn$ and define $\alpha(0)=n$, we then find $n$ such that $1Rn$ and define $\alpha(1)=n$, and so on.

Note that the axiom is not so plausible if we should require that an infinite sequence $\alpha$ is given by means of an algorithm. 

The classical mathematician would suggest to {\it define} $\alpha$: for each $m$, $\alpha(m)$ should be the least $n$ such that $mRn$. This suggestion does not work in a constructive context, except for the case that one may decide, for all $m, n$, if $mRn$ or not $mRn$.

\begin{axiom}[Second Axiom of Countable Choice]\label{ax:secondchoice} For every relation $R\subseteq \omega\times\omega^\omega$, \\if $\forall m \exists \alpha [mRn]$, then $\exists \alpha\forall m[m R\alpha^m]$. \end{axiom}

This axiom also seems to be a good proposal as, in general,  we think it possible to start a project for an infinite sequence, do some work on it, then start a second project for an infinite sequence, and do some work on it, then return to the first project and do some further work on it, and so on. 

Assume $\forall n\exists \alpha[nR\alpha]$. Start a project for an infinite sequence suitable to $0$, calling it $\alpha^0$ and find $\alpha^0(0)$. Start a project for an infinite sequence suitable to $1$, calling it $\alpha^1$, and find $\alpha^1(0)$ and $\alpha^0(1)$. Start a project for an infinite sequence suitable to $2$, calling it $\alpha^2$ and calculate $\alpha^2(0)$, $\alpha^1(1)$ and $\alpha^0(2)$. And so on.
\subsection{Sharpened versions of the Continuity Principle}
\begin{axiom}[First Axiom of Continuous Choice]\label{ax:firstacc} For every relation $R\subseteq \omega^\omega\times \omega$,\\ if $\forall \alpha \exists n[\alpha R n]$, then there exists $\varphi:\omega^\omega\rightarrow \omega$ such that $\forall \alpha[\alpha R \varphi(\alpha)]$.  \end{axiom}

Assume $\forall \alpha\exists n[\alpha R n]$. We find the promised $\varphi$ recursively, as follows. Given any $s$, we first ask if $\exists t\sqsubset s[\varphi (t)\neq 0]$. If so, we define $\varphi(s)=0$. If not, we imagine the finite sequence (coded by) $s$ as the beginning of an infinite sequence $\alpha$ that is created step by step, and we ask ourselves: `does $s$ contain sufficient information for finding $n$ such that $\alpha R n$?' If so, we choose such $n$ and define $\varphi(s)=n+1$, and, if not, we define $\varphi(s) =0$.

\begin{axiom}[Second Axiom of Continuous Choice]\label{ax:secondacc} For every relation $R\subseteq \omega^\omega\times \omega^\omega$,\\ if $\forall \alpha \exists \beta[\alpha R \beta]$, then there exists $\varphi:\omega^\omega\rightarrow \omega^\omega$ such that $\forall \alpha[\alpha R (\varphi|\alpha)]$.  \end{axiom}

Assume $\forall \alpha\exists \beta[\alpha R \beta]$. We find the promised $\varphi$ recursively, as follows. Given any $s$, we find the least $p$ such that  if $\forall i<p\exists t\sqsubset s[\varphi^i (t)\neq 0]$. We imagine the finite sequence (coded by) $s$ as the beginning of an infinite sequence $\alpha$ that is created step by step. Clearly, we convinced ourselves that $s$ contains sufficient sufficient information for finding the first $p$ values of a sequence $\beta$ satifying $\alpha R \beta$.  We now ask ourselves: `does $s$ contain sufficient information for finding the next value of the sequence $\beta$ such that $\alpha R\beta$, i.e. for finding $n$ such that $n =\beta(p)$?'  If so, we choose such $n$ and define $\varphi^p(s)=n+1$, and, if not, we define $\varphi^p(s) =0$.

\section{The Borel Hierarchy Theorem}\label{S:bht}

Although the early descriptive set theorist had their doubts about some of Cantor's assumptions, they never questioned the use of classical logic.
The symmetry of classical logic is heavily used in the classical proof of the Borel Hierarchy Theorem. It is not so easy to formulate and prove a satisfying similar result in a constructive context.   Brouwer's Continuity Principle comes to the rescue, and the resulting theorem may be considered a significant and surprising application of the principle.

\subsection{Introducing stumps}{\it Generalized inductive definitions} like the following one are acceptable intuitionistically. 
\begin{definition}\label{D:stp} 
$\mathbf{Stp}$, a collection  of  subsets of $\omega$, called \emph{stumps}, is defined as follows. 

\begin{enumerate}[\upshape (i)]
\item $\emptyset \in \mathbf{Stp}$, and
\item for every infinite sequence $S_0, S_1, \ldots$ of elements of $\mathbf{Stp}$,  the set \\$S:= \{\langle \;\rangle\}\cup\bigcup\limits_{n \in \omega}\langle
 n \rangle \ast S_n$ is again an element of $\mathbf{Stp}$, and
\item nothing more: every element of $\mathbf{Stp}$ is obtained by starting from $\emptyset$ and applying the operation mentioned in (ii) repeatedly. \end{enumerate}

For every non-empty stump $S$, for every $n$, $S\upharpoonright\langle n \rangle = \{s\mid \langle n \rangle \ast s \in S\}$ is called the $n$-th \emph{immediate substump} of $S$. 
\end{definition}
\begin{definition} For every stump $S$  we define $S':=\{s\mid s\notin S \;\wedge\;\forall t[t\sqsubset s\rightarrow t \in S]\}$. The set $S'$ is called the \emph{border of the stump $S$}. \end{definition}

The border $S'$ of $S$ consists of those (code numbers of) finite sequences of natural numbers that are {\it just outside} $S$.

We shall call $\mathbf{Stp}$ a \textit{class} or a {\it set}, although it is a totality of a different kind than $\omega^\omega$ or $\omega$. Again, we have a general idea how its members are created but only a very few of them have been realized until now.

Stumps take the r\^ole fulfilled by {\it countable ordinals} in classical analysis.

As we accept Definition \ref{D:stp}, we also feel entitled to use the following axiom.

\begin{axiom}[Induction on $\mathbf{Stp}$]\label{ax:stumpinduction} Let $P\subseteq \mathbf{Stp}$ be given. If  \begin{enumerate}[\upshape (i)] \item $\emptyset \in P$ and, \item for every nonempty stump $S$, if, for all $n$, $S\upharpoonright \langle n \rangle  \in P$, then $S\in P$,\end{enumerate} then $\mathbf{Stp}=P$. \end{axiom}

The following definition introduces a subclass of the class of stumps,  useful for the treatment of Borel sets.

\begin{definition} $\mathbf{Hrs}$,  a collection  of  subsets of $\omega$, called \emph{hereditarily repetitive nonzero stumps}, is defined as follows. 

\begin{enumerate}[\upshape (i)]
\item $\{\langle\;\rangle\} \in \mathbf{Hrs}$, and
\item for every infinite sequence $S_0, S_1, \ldots$ of elements of $\mathbf{Hrs}$,  the set \\$S:= \{\langle \;\rangle\}\cup\bigcup\limits_{m, n }
\langle (m, n)\rangle \ast S_n$ is again an element of $\mathbf{Hrs}$, and
\item nothing more: every element of $\mathbf{Hrs}$ is obtained by starting from $\{\langle\;\rangle\}$ and applying the operation mentioned in (ii) repeatedly.
\end{enumerate}

\smallskip $1^\ast:=\{\langle\;\rangle\}$ is called the \emph{basic element of} $\mathbf{Hrs}$. Note that, for every $S$ in $\mathbf{Hrs}$, $S=1^\ast$ if and only if $\langle 0 \rangle \notin S$, so one may decide if $S=1^\ast$ or not. 

\smallskip For every $S$ in $\mathbf{Hrs}$, for every $s$, $s$ is called an \emph{endpoint of $S$} if and only if $s\in S$ and $s\ast\langle 0 \rangle \notin S$. Note that $\langle\;\rangle$ is an endpoint of $1^\ast$ and that, for every  $S\ne 1^\ast$ in $\mathbf{Hrs}$, for every $s$, $s$ is and endpoint of $S$ if and only if there exist $n,t$ such that $s=\langle n \rangle\ast t$ and $t$ is an endpoint of $S\upharpoonright\langle n \rangle$.

\smallskip Note that, for every $S$ in $\mathbf{Hrs}$, for each $s$, $s$ is an endpoint of $S$ if and only if, for each $n$, $s\ast\langle n \rangle$ is en element of the border $S'$ of $S$.

\smallskip  Note: $(1^\ast)' = \{\langle n \rangle \mid n \in \omega\}$. 

\end{definition}

\subsection{(Positively) Borel sets}

The class of the (positively) Borel subsets of $\omega^\omega$ is the least class of subsets of $\omega^\omega$ containing the open subsets of $\omega^\omega$ and the closed subsets of $\omega^\omega$ that is  closed under the operations of countable union and countable intersection. We define this class using hereditarily repetitive nonzero stumps as indices.

The constructive mathematician avoids  negation\footnote{A negative statement $\neg P$ reports the {\it failure} of obtaining a construction validating $P$.} and the operation of taking the complement $X\setminus Y$ of a given set $Y\subseteq X$ as much as possible. It is not true, intuitionistically, that the complement $\omega^\omega\setminus Y$ of a of a given positively Borel set $Y\subseteq \omega^\omega$  is again positively Borel.

\begin{definition}[Borel sets and Borel classes] For every  $S$ in $\mathbf{Hrs}$, for every $\beta$, \\we define subsets $\mathcal{G}^S_\beta$ and $\mathcal{F}^S_\beta$ of $\omega^\omega$,  by induction, as follows.

\begin{enumerate}[\upshape (i)] \item $\mathcal{G}^{1^\ast}_\beta =\{\alpha\mid \exists n[\beta (\overline \alpha  n) \neq 0]\}$ and $\mathcal{F}^{1^\ast}_\beta =\{\alpha\mid \forall n[\beta (\overline\alpha n) = 0]\}$.

\item For every $S\neq 1^\ast$,
$\mathcal{G}^S_\beta = \bigcup_n \mathcal{F}^{S\upharpoonright\langle n \rangle }_{\beta\upharpoonright \langle n \rangle}$ and  $\mathcal{F}^S_\beta = \bigcap_n \mathcal{G}^{S\upharpoonright\langle n \rangle }_{\beta\upharpoonright \langle n \rangle}$.
\end{enumerate}
\smallskip
For  every $S$ in $\mathbf{Hrs}$, we define classes $\mathbf{\Sigma}^0_S$ and $\mathbf{\Pi}^0_S$ of subsets of $\omega^\omega$ as follows.

For every $\mathcal{X}\subseteq \omega^\omega$, 

$\mathcal{X}$ is  $\mathbf{\Sigma}^0_S$ if and only if $\exists \beta[\mathcal{X}=\mathcal{G}^S_\beta]$ and   $\mathcal{X}$ is  $\mathbf{\Pi}^0_S$ if and only if $\exists \beta[\mathcal{X}=\mathcal{F}^S_\beta]$.

\smallskip $\mathcal{X}\subseteq \omega^\omega$ is \emph{ open} if and only if $\mathcal{X}$ is $\mathbf{\Sigma}^0_{1^\ast}$ and \emph{ closed}   if and only if $\mathcal{X}$ is $\mathbf{\Pi}^0_{1^\ast}$.

\smallskip $\mathcal{X}\subseteq \omega^\omega$ is $\mathfrak{Borel}$ if and only if, for some $S$ in $\mathbf{Hrs}$, $\mathcal{X}$ is $\mathbf{\Sigma}^0_S$.
 \end{definition}
 
 In each Borel class, we single out a special element that will turn out to be an element of the class of maximal complexity.
\begin{definition}[The leading sets of the hierarchy]\label{D:E_SA_S}

 For every $S$ in $\mathbf{Hrs}$ we define subsets $\mathcal{E}_S$ and $\mathcal{A}_S$ of $\omega^\omega$ as follows.

\begin{enumerate}[\upshape (i)] \item $\mathcal{E}_{1^\ast}:=\mathcal{G}^{1^\ast}_{Id}:=\{\alpha\mid\exists n[\alpha(\langle n \rangle)\neq 0]\}$ and $\mathcal{A}_{1^\ast}:=\mathcal{F}^{1^\ast}_{Id}:=\{\alpha\mid\forall  n[\alpha(\langle n \rangle)= 0]\}$. \item for each hereditarily repetitive nonzero stump $S$, if $S\neq 1^\ast$, \\then $\mathcal{E}_S:= \mathcal{G}^S_{Id}=\{\alpha\mid \exists n [\alpha\upharpoonright\langle n \rangle \in \mathcal{A}_{S\upharpoonright \langle n \rangle}]\}$ and \\
 $\mathcal{A}_S:=\mathcal{F}^S_{Id}= \{\alpha\mid \forall n[\alpha\upharpoonright\langle n \rangle \in \mathcal{E}_{S\upharpoonright \langle n \rangle}]\}$. \end{enumerate}

 \end{definition}
 
 \begin{remark} One may prove, by induction on $\mathbf{Hrs}$: \\for each $S$ in $\mathbf{Hrs}$, for all $\alpha, \beta$, if $\alpha \in \mathcal{A}_S$ and $\beta\in \mathcal{E}_S$, then $\alpha\:\#\;\beta$. \end{remark}
\subsection{Games} It is very useful to think of the leading sets of the Borel hierarchy in a {\it game theoretic} way, as follows. 
 
\begin{definition}[Introducing games and strategies] Let a hereditarily repetitive nonzero stump $S$ be given, and let also $\alpha$ in $\omega^\omega$ be given. We introduce $\mathbb{G}_S(\alpha)$, \emph{the game in $S$ for $\alpha$}. 
 
 A play in $\mathbb{G}_S(\alpha)$ goes as follows.  Players $I,II$ start constructing an infinite sequence $\gamma$ in $\omega^\omega$.\begin{quote} $I$ chooses  $\gamma(0)$, $II$ chooses $\gamma(1)$, $I$ chooses $\gamma(2)$, $\ldots$ \end{quote}
 The play ends as soon as a position $\overline \gamma n= \langle \gamma(0), \gamma(1), \ldots, \gamma(n-1)\rangle$ in the \emph{border} $S'$ of $S$ is reached. If $n$ is even, then Player $I$ wins the play if and only if $\alpha(\overline \gamma n)=0$ and, if $n$ is odd, then Player $I$ wins the play if and only if $\alpha(\overline \gamma n)\neq 0$. Player $II$ wins the play if and only if Player $I$ does not win the play.
 
 \smallskip For all $\sigma, \tau$, for all $s$, we define: $s\in_I\sigma$, \emph{`$s$ is played by Player $I$ according to the strategy $\sigma$'}, if and only if $\forall n[2n<length(s)\rightarrow s(2n)=\sigma\bigl(\overline s(2n)\bigr)]$, and
 $s\in_{II}\tau$, \emph{`$s$ is played by Player $II$ according to the strategy $\tau$'}, if and only if $\forall n[2n+1<length(s)\rightarrow s(2n+1)=\tau\bigl(\overline s(2n+1)\bigr)]$.

  $\;$\\For each  $S$ in $\mathbf{Hrs}$, for each $\alpha$, for all $\sigma,\tau$, we define:\\ $\mathcal{W}^I_S(\sigma, \alpha)$, \emph{$\sigma$ is a winning strategy for Player $I$ in $\mathbb{G}_S(\alpha)$},  if and only if \\$\forall i\forall s\in S'[s\in_I\sigma\rightarrow\bigl( (s\in \omega^{2i} \rightarrow \alpha(s)= 0)\;\wedge\; (s\in \omega^{2i+1} \rightarrow \alpha(s)\neq 0)\bigr)]$, and:\\ $\mathcal{W}^{II}_S(\tau, \alpha)$, \emph{$\tau$ is a winning strategy for Player $II$ in $\mathbb{G}_S(\alpha)$},  if and only if  \\$\forall i\forall s\in S'[s\in_{II}\tau\rightarrow\bigl( (s\in \omega^{2i} \rightarrow \alpha(s)\neq 0)\;\wedge\; (s\in \omega^{2i+1} \rightarrow \alpha(s)= 0)\bigr)]$. \end{definition}
 
 \begin{theorem}\label{T:strategies} For every $S$ in $\mathbf{Hrs}$,  \begin{enumerate}[\upshape (i)] \item for every $\alpha$,  for every $\sigma$, if $\mathcal{W}^I(\sigma, \alpha)$, then $\alpha \in \mathcal{E}_S$ and \item for every $\alpha$, for every $\tau$,  if $\mathcal{W}^{II}_S(\tau, \alpha)$, then $\alpha \in \mathcal{A}_S$. \item  for every $\alpha$,  if $ \alpha \in \mathcal{E}_S$, then, for some $\sigma$, $\mathcal{W}^I_{S}(\sigma, \alpha)$,  and \item for every $\alpha$,  if $ \alpha \in \mathcal{A}_S$, then, for some $\tau$,  $\mathcal{W}^{II}_S(\tau, \alpha)$.\end{enumerate} \end{theorem}
 \begin{proof} (i) and (ii).  The proof is by induction on $\mathbf{Hrs}$. 
 
 The case $S=1^\ast$ is easy and left to the reader.

 Now let $S$ be given such that $S\neq 1^\ast$ and, for each $n$, for each $\alpha$, \\ if $\exists \sigma[\mathcal{W}_{S\upharpoonright\langle n \rangle}^I(\sigma, \alpha)]$,, then $\alpha \in \mathcal{E}_{S\upharpoonright\langle n \rangle}$, and,  if $\exists \tau[\mathcal{W}_{S\upharpoonright\langle n \rangle}^{II}(\tau, \alpha)]$, then $\alpha \in \mathcal{A}_{S\upharpoonright\langle n \rangle}$.
 
 \smallskip Let $\sigma, \alpha$ be given such that $\mathcal{W}^I_S(\sigma, \alpha)$.  Define $n_0:=\sigma(\langle \;\rangle)$. 
  \\ Define  $\tau:=\sigma\upharpoonright \langle n_0\rangle$ and note:  $\mathcal{W}^{II}_{S\upharpoonright\langle n_0\rangle}(\tau, \alpha\upharpoonright\langle n_o\rangle)$. \\ Conclude: $\alpha\upharpoonright\langle n_0\rangle \in \mathcal{A}_{S\upharpoonright\langle n_0\rangle}$ and $\alpha \in \mathcal{E}_S$.
  
  \smallskip  Let $\tau, \alpha$ be given such that $\mathcal{W}_S^{II}(\tau, \alpha)$. 
  \\Then, for each $n$, $\mathcal{W}_{S\upharpoonright\langle n \rangle}^I(\tau\upharpoonright \langle n \rangle, \alpha\upharpoonright\langle n \rangle)$.  \\Conclude: for each $n$, $\alpha\upharpoonright\langle n \rangle \in \mathcal{E}_{S\upharpoonright\langle n \rangle}$ and: $\alpha \in \mathcal{A}_S$. 
  
  \smallskip (iii) and (iv). The proof is by induction on $\mathbf{Hrs}$. 
 
 The case $S=1^\ast$ is easy and left to the reader.

 Now let $S$ be given such that $S\neq 1^\ast$ and, for each $n$, for each $\alpha$, \\ if $ \alpha \in \mathcal{E}_{S\upharpoonright\langle n \rangle} $, then, for some $\sigma$,  $\alpha \in \mathcal{W}^I_{S\upharpoonright\langle n \rangle}(\sigma, \alpha)$, and,  \\if $ \alpha \in \mathcal{A}_{S\upharpoonright\langle n \rangle} $, then for some $\tau$,  $\alpha \in \mathcal{W}^{II}_{S\upharpoonright\langle n \rangle}(\tau, \alpha)$, 
 
  Let $\alpha$ be given such that $\alpha \in \mathcal{E}_S$. Find $n$ such that $\alpha\upharpoonright\langle n \rangle \in \mathcal{A}_{S\upharpoonright\langle n \rangle}$. \\Find $\tau$ such that $\mathcal{W}^{II}_{S\upharpoonright\langle n \rangle} (\tau, \alpha\upharpoonright\langle n \rangle)$. \\Define $\sigma$ such that $\sigma(\langle\;\rangle)=n$ and $\sigma\upharpoonright \langle n \rangle =\tau$ and note: $\mathcal{W}^I_S(\sigma, \alpha)$. 
  
  \smallskip Let $\alpha$ be given such that $\alpha \in \mathcal{A}_S$. \\Then $\forall n[\alpha\upharpoonright\langle n \rangle \in \mathcal{E}_{S\upharpoonright\langle n \rangle}]$ and $\forall n\exists \sigma[\mathcal{W}^I_{S\upharpoonright\langle n \rangle}(\sigma, \alpha\upharpoonright\langle n \rangle)]$. 
   \\Using the Second Axiom of Countable Choice, Axiom \ref{ax:secondchoice}, find $\tau$ such that \\$\forall n[\mathcal{W}^I_{S\upharpoonright\langle n\rangle}(\tau\upharpoonright\langle n \rangle, \alpha\upharpoonright\langle n \rangle)]$ and note: $\mathcal{W}^{II}_S(\tau, \alpha)$.  
  \end{proof}
  \begin{definition}\label{D:correctionbystrategies}For every $S$ in $\mathbf{Hrs}$, for every $\tau$, for every $\alpha$, we let $\tau\Join_S^{II}\alpha$, \\\emph{`$\alpha$-as-corrected by the strategy $\tau$ for Player $II$ in $\mathbb{G}_S(\alpha)$'}, be the element of $\omega^\omega $ satisfying:
  for all $s$ in the border $S'$ of $S$, \\if $length(s)$ is even, then $\tau\Join_S^{II}\alpha(s) =\max\bigl(1, \alpha(s)\bigr)$, and, \\if $length(s)$ is odd, then $\tau\Join_S^{II}\alpha(s) =0,$ and, \\for all $s$, if $s$ is not in the border $S'$ of $S$, then $\tau\Join_S^{II}\alpha(s)=\alpha(s)$.\end{definition}
  
  \begin{remark}Note: for every $S$ in $\mathbf{Hrs}$, for every $\alpha$, \\$\alpha \in \mathcal{A}_S$ if and only if, for some $\tau$, $\alpha= \tau\Join^{II}_S\alpha$. \end{remark}
  
  The following Lemma is crucial. With this tool we will prove the Hierarchy Theorem.
  \begin{lemma}[A consequence of Brouwer's Continuity Principle]\label{L:continuityborel}$\;$\\For every $S$ in $\mathbf{Hrs}$, for every relation $R\subseteq \omega^\omega\times\omega$, if $\forall \alpha \in \mathcal{A}_S\exists n[\alpha R n]$, then $\forall \alpha \in \mathcal{A}_S\exists p\exists q\forall \beta \in \mathcal{A}_S[\bigl(\beta(0)=\alpha(0) \;\wedge\;\forall n<p[\beta\upharpoonright\langle n\rangle=\alpha\upharpoonright \langle n\rangle]\bigr)\rightarrow \beta Rq]$. \end{lemma}
  
  \begin{proof}Assume: $\forall \alpha\in \mathcal{A}_S\exists n[\alpha Rn]$. 
  Conclude: $\forall \alpha \forall \tau\exists n[(\tau\Join^{II}_S\alpha)Rn]$.
  
 Let $\alpha$ in $\mathcal{A}_S$ be given. Using Theorem \ref{T:strategies}(iv), find $\tau$ such that $\alpha = \tau\Join_S^{II}\alpha$.
 
  Using Brouwer's Continuity Principle, Axiom \ref{ax:bcp}, find $m,q$ such that \\$\forall \beta\forall \sigma[(\overline\alpha m\sqsubset \beta \;\wedge\; \overline \tau m\sqsubset \sigma)\rightarrow (\sigma\Join^{II}_S \beta)Rq]$.
  
  Find $p$ such that, for all $s$, if $s\le m$, then $s(0)< p$.
  
  Note that, for all $\beta$ in $\mathcal{A}_S$, if $\beta(0)=\alpha(0)$ and $\forall n<p[\beta\upharpoonright\langle n \rangle =\alpha\upharpoonright\langle n \rangle]$, then $\overline \alpha m \sqsubset \beta$ and, for some $\sigma$, $\overline \tau m\sqsubset \sigma$ and $\beta=\sigma\Join^{II}_S\beta$, and, therefore, $\beta Rq$. \end{proof}
  
\subsection{(Wadge-)reducibility}  
 The following notion of reducibility plays a key r\^ole. In classical descriptive set theory this notion is called {\it Wadge-reducibility}.  Its analog in computability theory is  {\it many-one-reducibility}.  
 \begin{definition}[Reducibility] For all $\mathcal{X}, \mathcal{Y}\subseteq \omega^\omega$, for all $\varphi:\omega^\omega\rightarrow \omega^\omega$ we define: \\\emph{$\varphi$ reduces $\mathcal{X}$ to $\mathcal{Y}$} if and only if, for every $\alpha$, $\alpha \in \mathcal{X} \leftrightarrow \varphi|\alpha \in \mathcal{Y}$. 
 
 \smallskip For all $\mathcal{X}, \mathcal{Y}\subseteq \omega^\omega$, we define: $\mathcal{X}\preceq \mathcal{Y}$, $\mathcal{X}$ \emph{reduces to} $\mathcal{Y}$ if and only if \\there exists $\varphi:\omega^\omega\rightarrow\omega^\omega$ reducing $\mathcal{X}$ to $\mathcal{Y}$. \end{definition}

\begin{theorem}[$\mathcal{E}_S$, $\mathcal{A}_S$ are \emph{complete} elements of $\mathbf{\Sigma}^0_S$, $\mathbf{\Pi}^0_S$, respectively]$\;$

For each  $S$ in $\mathbf{Hrs}$, for every $\mathcal{X}\subseteq \omega^\omega$,
$\mathcal{X}\subseteq \omega^\omega$ is $\mathbf{\Sigma}^0_S$ if and only if $\mathcal{X}\preceq \mathcal{E}_S$, and  $\mathcal{X}\subseteq \omega^\omega$ is $\mathbf{\Pi}^0_S$ if and only if $\mathcal{X}\preceq \mathcal{A}_S$. \end{theorem} 

\begin{proof} The proof that, for each $S$ in $\mathbf{Hrs}$, $\mathcal{E}_S$ is $\mathbf{\Sigma}^0_S$ and $\mathcal{A}_S$ is $\mathbf{\Pi}^0_S$ is left to the reader.
Also the proof that, for each $S$ in $\mathbf{Hrs}$, for  all $\mathcal{X}, \mathcal{Y}\subseteq \omega^\omega$, if $\mathcal{X}\preceq \mathcal{Y}$, then, if $\mathcal{Y}$ is $\mathbf{\Sigma}^0_S$, then  $\mathcal{X}$ is $\mathbf{\Sigma}^0_S$ and, if $\mathcal{Y}$ is $\mathbf{\Pi}^0_S$, then  $\mathcal{X}$ is $\mathbf{\Pi}^0_S$, is left to the reader.

\smallskip

 We now prove that, for each $S$ in $\mathbf{Hrs}$, for each $\beta$, there exists $\varphi:\omega^\omega\rightarrow \omega^\omega$ reducing both $\mathcal{G}^S_\beta$ to $\mathcal{E}_S$ and $\mathcal{F}^S_\beta$ to $\mathcal{A}_S$ and we do so by induction on $\mathbf{Hrs}$. 

\smallskip
We first consider the case that $S=\{\langle\;\rangle\}=1^\ast$ is the basic element of $\mathbf{Hrs}$.  \\Let $\beta$ be given. Define $\varphi:\omega^\omega\rightarrow\omega^\omega$ such that, for each $\alpha$, for each $n$, \\$(\varphi|\alpha)(\langle n\rangle)=\beta(\overline\alpha n)$, and note that $\varphi$ reduces $\mathcal{G}^{1^\ast}_\beta$ to $\mathcal{E}_{1^\ast}$ and $\mathcal{F}^{1^\ast}_\beta$ to $\mathcal{A}_{1^\ast}$.

\smallskip Now let $S$ be a non-basic element of $\mathbf{Hrs}$. Let $\beta$ be given. Using the induction hypothesis and the Second Axiom of Countable Choice\footnote{The use of this axiom at this place may be avoided.}, find $\varphi$ such that, for each $n$, $\varphi^n:\omega^\omega\rightarrow \omega^\omega$ reduces $\mathcal{G}^{S\upharpoonright\langle n \rangle}_{\beta\upharpoonright\langle n \rangle}$ to $\mathcal{E}_{S\upharpoonright \langle n \rangle}$ and $\mathcal{F}^{S\upharpoonright\langle n \rangle}_{\beta\upharpoonright\langle n \rangle}$ to $\mathcal{A}_{S\upharpoonright \langle n \rangle}$. \\Define $\psi:\omega^\omega\rightarrow\omega^\omega$ such that, for each $\alpha$, for each $n$, $(\psi|\alpha)\upharpoonright \langle n \rangle = \varphi^n|(\alpha\upharpoonright\langle n\rangle )$ and note that $\psi$ reduces $\mathcal{G}^S_\beta$ to $\mathcal{E}_S$ and $\mathcal{F}^S_\beta$ to $\mathcal{A}_S$. 
\end{proof}

 \begin{definition}\label{D:canonicalelements}For each  $S$ in $\mathbf{Hrs}$, we define  \emph{canonical elements} $\varepsilon^{\ast}_S$, $\alpha^{\ast}_S$,  of  the sets $\mathcal{E}_S$, $\mathcal{A}_S$, respectively, as follows.

 $\forall s[(s\notin S' \rightarrow \varepsilon^\ast_S(s)=0] \;\wedge\\ \forall i \forall s\in S'[ \bigl(s\in \omega^{2i} \rightarrow \varepsilon^\ast_S(s)=0\bigr)\;\wedge\; \bigl(s\in \omega^{2i+1} \rightarrow \varepsilon^\ast_S(s)=1\bigr)]$, and
 
 $\forall s[(s\notin S' \rightarrow \alpha^\ast_S(s)=0] \;\wedge\\ \forall i \forall s\in S'[ \bigl(s\in \omega^{2i} \rightarrow \alpha^\ast_S(s)=1\bigr)\;\wedge\; \bigl(s\in \omega^{2i+1} \rightarrow \alpha^\ast_S(s)=0\bigr)]$.\end{definition}
 
 \begin{remark} Note that, for every $S$ in $\mathbf{Hrs}$, \\for every $\sigma$, $\mathcal{W}^I_S(\sigma, \varepsilon^\ast_S)]$, and, for every $\tau$, $\mathcal{W}^{II}_S(\tau, \alpha_S^\ast)]$. \end{remark}
 
 \subsection{The Hierarchy Theorem}
 
 We would like to prove the statement that, for each $S$ in $\mathbf{Hrs}$, $\mathbf{\Pi}^0_S$ is not a subclass of $\mathbf{\Sigma}^0_S$ and, conversely,  $\mathbf{\Sigma}^0_S$ is not a subclass of $\mathbf{\Pi}^0_S$. We are going to prove a stronger and more positive  statement. We shall make use of the following technical notion.
\begin{definition}[Freedom in spreads]\label{D:value-dictating} $\;$ \\ Let $\beta$ be a spread-law and let $\mathcal{F}=\mathcal{F}_\beta$ be the corresponding spread. 

Let $t$ be given. 

We define: $t$ is \emph{free in $\mathcal{F}$} if and only if, for all $s$ in $\omega^t$, if $\beta(s)=0$, then, for all $m$, $\beta(s\ast\langle m \rangle)=0$.  

We  define: $t$ is \emph{completely free in $\mathcal{F}$} if and only if, for all $u$ such that $t\sqsubseteq u$, $u$ is free in $\mathcal{F}$.

 \end{definition}
\begin{remark} If $t$ is free in $\mathcal{F}=\mathcal{F}_\beta$, then, when building, step-by-step, an element $\alpha$ of $\mathcal{F}$, and having defined $\alpha(0), \ldots, \alpha(t-1)$, we may choose any number $m$ as a value of $\alpha(t)$. The spread-law $\beta$ does not impose any restriction on our freedom of choice at $t$.\end{remark}

 We shall prove that, for each $S$ in $\mathbf{Hrs}$, $\mathbf{\Pi}^0_S$ is not a subclass of $\mathbf{\Sigma}^0_S$ by showing that every $\varphi:\omega^\omega\rightarrow\omega^\omega$ mapping $\mathcal{A}_S$ into $\mathcal{E}_S$ {\it positively fails to reduce $\mathcal{A}_S$ to $\mathcal{E}_S$}, and, similarly, we prove that, for each $S$ in $\mathbf{Hrs}$, $\mathbf{\Sigma}^0_S$ is not a subclass of $\mathbf{\Pi}^0_S$ by showing that every $\varphi:\omega^\omega\rightarrow\omega^\omega$ mapping $\mathcal{E}_S$ into $\mathcal{A}_S$ {\it positively fails to reduce $\mathcal{E}_S$ to $\mathcal{A}_S$}, i.e.:

 \begin{theorem}[Borel Hierarchy Theorem]\label{T:bht}
 $\;$ \begin{enumerate}[\upshape (i)] \item For every  $S$ in $\mathbf{Hrs}$, for every $\varphi:\omega^\omega\rightarrow\omega^\omega$, \\if $\forall\alpha[\alpha \in \mathcal{A}_S\rightarrow \varphi|\alpha \in \mathcal{E}_S]$, then $\exists \alpha[\alpha \in \mathcal{E}_S \;\wedge\; \varphi|\alpha \in \mathcal{E}_S]$.
 
 \item For every  $S$ in $\mathbf{Hrs}$, for every $\varphi:\omega^\omega\rightarrow\omega^\omega$, \\if $\forall\alpha[\alpha \in \mathcal{E}_S\rightarrow \varphi|\alpha \in \mathcal{A}_S]$, then $\exists \alpha[\alpha \in \mathcal{A}_S \;\wedge\; \varphi|\alpha \in \mathcal{A}_S]$.  \end{enumerate} \end{theorem}
 
 \begin{proof} (i) Let $S$ in $\mathbf{Hrs}$ be given.  
 
 Let $\varphi:\omega^\omega\rightarrow\omega^\omega$ be given such that   $\forall\alpha[\alpha \in \mathcal{A}_S\rightarrow \varphi|\alpha \in \mathcal{E}_S]$.

 We are going to construct $\alpha, \sigma, \tau$ such that $\mathcal{W}^{I}_S(\sigma, \alpha)]$ and $\mathcal{W}^{I}_S(\tau, \varphi|\alpha)$, and, therefore, both $\alpha$ and $\varphi|\alpha$ are in $\mathcal{E}_S$.  
 
 We do so by building  an infinite sequence $\mathcal{F}_0, \mathcal{F}_1, \mathcal{F}_2, \ldots$ of  spreads such that
 
 \begin{enumerate}[\upshape a.] \item $\mathcal{F}_0 = \omega^\omega$ and for each $t$, $\mathcal{F}_{t+1}\subseteq\mathcal{F}_t$,  and, for each $s$, if $length(s)<length(t)$, then $s$ is admitted by the spread-law defining $\mathcal{F}_t$ if and only if $s$ is admitted by the spread-law defining $\mathcal{F}_{t+1}$.  \item for all $i,t, u$, \\if  $t,u\in \omega^{2i}$ and $t \in S$ and $t\in_I \sigma$ and $u\in_I\tau$ and $t_{II}=u_{II}$,  then \begin{enumerate}[\upshape b1.] \item $t$ is completely free in $\mathcal{F}_t$, and \item $u\in S$ and  $S\upharpoonright t=S\upharpoonright u$,  and \item $\forall \alpha \in \mathcal{F}_t[\alpha\upharpoonright t \in \mathcal{A}_{S\upharpoonright t} \rightarrow (\varphi|\alpha)\upharpoonright u \in \mathcal{E}_{S\upharpoonright u}]$, and, \item if $t$ is an endpoint of $S$,  then $u$ is an endpoint of $S$ and \\$\forall \alpha \in \mathcal{F}_{t+1}[ \alpha(t\ast\langle\sigma(t)\rangle)=(\varphi|\alpha)(u\ast\langle \delta(u)\rangle) =1  ]$,  and, \item if $t\ast\langle \sigma(t)\rangle$ is an endpoint of $S$,   then $u\ast\langle\tau(u)\rangle$ is an endpoint of $S$ and $\forall \alpha \in \mathcal{F}_{t+1}\forall n[ \alpha(t\ast\langle \sigma(s), n\rangle)=(\varphi|\alpha)(u\ast\langle \delta(u), n\rangle = 0]$. \end{enumerate}\end{enumerate}
 
 Note that, as $\mathcal{F}_0=\omega^\omega$, conditions b1, b2 and b3 are satisfied for $i=0$.

Now let $t$ be given. We want to define $\mathcal{F}_{t+1}$ and distinguish two cases.

\smallskip

{\it Case (i)}  $t\notin S$, or $\exists i[t\in \omega^{2i+1}]$, or $t\notin_I\sigma$.  We define $\mathcal{F}_{t+1}=\mathcal{F}_t$.

 \smallskip {\it Case (ii)}.  $t\in S$, and $\exists i[t\in \omega^{2i}]$, and $t\in_I\sigma$.
 
  We now determine $u$ such that   $length(u)=length(t)$  and $u\in_I\tau$ and $t_{II}=u_{II}$.  We then know  $S\upharpoonright t=S\upharpoonright u$ and $t$ is  completely free in $\mathcal{F}_t$   and \\$\forall \alpha \in \mathcal{F}_t[\alpha\upharpoonright t \in \mathcal{A}_{S\upharpoonright t} \rightarrow (\varphi|\alpha)\upharpoonright u \in \mathcal{E}_{S\upharpoonright u}]$. 
 
 \smallskip
 
 We again distinguish two cases.
 
 \smallskip
 {\it Case (ii)(I)}. $S\upharpoonright t= S\upharpoonright u =1^\ast=\{\langle\;\rangle\}$, i.e. $t,u$ are endpoints of $S$. \\ Note: $\forall \alpha \in \mathcal{F}_t[\alpha\upharpoonright t \in \mathcal{A}_{1^\ast}\rightarrow (\varphi|\alpha)\upharpoonright u \in \mathcal{E}_{1^\ast}]$. \\Find $\eta$ in $\mathcal{F}_t$ such that $\eta\upharpoonright t =\alpha^\ast_{S\upharpoonright t}=\alpha_{1^\ast}^\ast\in \mathcal{A}_{1^\ast}$, i.e. $\forall n[\eta(t\ast\langle n \rangle)=0]$. \\Note: $(\varphi|\eta)\upharpoonright u \in \mathcal{E}_{1^\ast}$, and find   $q$ such that $(\varphi|\eta)(u\ast\langle q\rangle)\neq 0$.
 
 Define $\tau(u):=q$. \\ Find $p$ such that $\overline{\varphi|\eta}(u\ast\langle \tau(u) \rangle +1)\sqsubseteq \varphi|(\overline \eta p)$.
 \\Find $r$  such that $t\ast\langle r\rangle \ge p$. Note that  $t\ast\langle r\rangle$ is  free in $\mathcal{F}_t$.
 
 Define $\sigma(t):=r$.
  
 Find $c$ such that $\overline \eta \tau(u)\sqsubset c$ and $c(t\ast\langle \sigma(t) \rangle) =1$ and $c$ is admitted by $\mathcal{F}_t$.   
 \\Define   $\mathcal{F}_{t+1}= \mathcal{F}_t \cap c$.
 
 \smallskip
 Note: for every $\alpha$ in $\mathcal{F}_{t+1}$, $\alpha(t\ast\langle \sigma(t)\rangle) =1$ and $(\varphi|\alpha)(u\ast\langle \tau(u)\rangle)=1$.
 
 Also note: for all $v\perp t$, if $v$ is  completely free in $\mathcal{F}_t$, then $v$ is completely free in $\mathcal{F}_{t+1}$.
 
 \medskip
 
 {\it Case (ii)(II)}. $S\upharpoonright t= S\upharpoonright u \neq 1^\ast$, i.e. $t,u$ are no endpoints    of $S$.
 
  We have  to define $\mathcal{F}_{t+1}$ but also $\sigma(t)$ and $\tau(u)$.

    Find $\eta$ in $\mathcal{F}_t$ such that $\eta\upharpoonright t =\alpha^\ast_{S\upharpoonright t}\in \mathcal{A}_{S\upharpoonright t}$.

   We define $\psi:\omega^\omega\rightarrow\omega^\omega$ such that, for each $\beta$,  \begin{enumerate}[\upshape (1)] \item  $\forall u[\neg (t\sqsubseteq u)\rightarrow (\psi|\beta)(u)=\eta(u)]$, and  \item $(\varphi|\beta)\upharpoonright t=\beta$.
   
    \end{enumerate}
   
Note that, for all $\beta$, if $\beta \in \mathcal{A}_{S\upharpoonright t}$, then $(\psi|\beta)\upharpoonright t \in \mathcal{A}_{S \upharpoonright t}$ and $\bigl(\varphi|(\psi|\beta)\bigr)\upharpoonright u \in \mathcal{E}_{S\upharpoonright u}$.

In particular, this is true for $\beta= \alpha^\ast_{S\upharpoonright t}$.

Note: $\psi|(\alpha^\ast_{S\upharpoonright t})=\;\eta$.

Using Lemma \ref{L:continuityborel}, find $p, q$ such that, for each $\beta$, if  $\beta \in\mathcal{A}_{S\upharpoonright t}$ and $\beta(0)=\alpha^\ast_{S\upharpoonright t}(0)$ and $\forall n <p[\beta\upharpoonright \langle n \rangle = \alpha^\ast_{S\upharpoonright t}\upharpoonright \langle n \rangle]$, then $\bigl(\varphi|(\psi|\beta)\bigr)\upharpoonright u\ast\langle q\rangle \in \mathcal{A}_{S\upharpoonright u\ast\langle q \rangle}$. 

Define  $\tau(u)=q$.

Find $r>p$ such that $S\upharpoonright t\ast\langle r \rangle = S\upharpoonright u\ast\langle q \rangle = S\upharpoonright u \ast\langle \tau(u)\rangle$ and define $\sigma(t)=r$.

Define $\mathcal{F}= \{\alpha \in \mathcal{F}_t \mid \forall u [(t\sqsubset u \;\wedge\;u\perp t\ast\langle \gamma(t)\rangle)\rightarrow \alpha(u)=\eta(u)]\}$. 

Note that $t\ast\langle\sigma(t)\rangle$ is completely free in $\mathcal{F}$. 

Also note that, for each $v$, if $v \perp t$ and $v$ is  completely free in $\mathcal{F}_t$, then $v$ is  completely free in $\mathcal{F}$. 

Note: $\forall \alpha \in \mathcal{F}[\alpha\upharpoonright t\ast\langle\gamma(t)\rangle \in \mathcal{E}_{S\upharpoonright t\ast\langle \gamma(t)\rangle }\rightarrow(\varphi|\alpha)\upharpoonright u\ast\langle\delta(u)\rangle \in \mathcal{A}_{S\upharpoonright u\ast\langle \delta(u)\rangle }]$.

\smallskip
We distinguish two subcases of Case {\it (ii)(II)}. 

\smallskip
{\it Case (ii)(II)a}. $t\ast \langle \sigma(t)\rangle$ and $u\ast\langle\tau(u)\rangle$ are endpoints of $S$.

Define $\mathcal{F}_{t+1}= \{\alpha \in \mathcal{F} \mid \forall n[\alpha(t\ast\langle \sigma(t), n \rangle)=0]\}$. 

We claim: $\forall \alpha \in \mathcal{F}_{t+1} \forall n[(\varphi|\alpha)(u\ast\langle \tau (u), n\rangle)=0]$.  

We prove this claim as follows. \\Assume we find $\alpha$ in $\mathcal{F}_{t+1}$ and $n$ such that $(\varphi|\alpha)(u\ast\langle \tau (u), n\rangle)\neq 0$.

Find $m$ such that $\overline{\varphi|\alpha}(u\ast\langle\tau(u),n\rangle +1)\sqsubseteq \varphi|\overline \alpha m$.

Find $\alpha^\diamond$ in $\mathcal{F}_{t+1}$ such that $\overline \alpha m \sqsubset \alpha^\diamond$ and, for some $p$, $\alpha^\diamond(t\ast\langle \sigma(t), p\rangle)\neq 0$.

Note: both $\alpha^\diamond\upharpoonright t\ast\langle \sigma(t)\rangle$ and $(\varphi|\alpha^\diamond)\upharpoonright u\ast\langle \tau(u)\rangle$ are in $\mathcal{E}_{1^\ast}$. 

Contradiction. 

We have to conclude: $\forall \alpha \in \mathcal{F}_{t+1} \forall n[(\varphi|\alpha)(u\ast\langle \tau (u), n\rangle)=0]$.

\smallskip {\it Case (ii)(II)b}. $t\ast\langle\sigma(t)\rangle$ and $u\ast\langle \tau(u)\rangle$ are no endpoints of $S$. 

Define $\mathcal{F}_{t+1} =\mathcal{F}$.

Note: $\forall \varepsilon \in \mathcal{F}_{t+1}[\varepsilon\upharpoonright t\ast\langle\sigma(t)\rangle \in \mathcal{E}_{S\upharpoonright t\ast\langle \sigma(t)\rangle }\rightarrow(\varphi|\varepsilon)\upharpoonright u\ast\langle\tau(u)\rangle \in \mathcal{A}_{S\upharpoonright u\ast\langle \delta(u)\rangle }]$.

 Conclude: \\$\forall \alpha \in \mathcal{F}_{t+1}\forall n[\alpha\upharpoonright t\ast\langle\sigma(t),n\rangle \in \mathcal{A}_{S\upharpoonright t\ast\langle \gamma(t),n\rangle }\rightarrow(\varphi|\alpha)\upharpoonright u\ast\langle\delta(u), n\rangle \in \mathcal{E}_{S\upharpoonright u\ast\langle \delta(u),n\rangle }]$.
 
 Also note: for each $n$, $t\ast\langle\sigma(t), n\rangle$ is completely free in $\mathcal{F}_{t+1}$.
 
 \smallskip  This concludes the definition of the sequence $\mathcal{F}_0, \mathcal{F}_1, \ldots$.
 
 \bigskip Using the fact that  condition a is satisfied, find $\alpha$ such that for each $t$, $\alpha\in \mathcal{F}_t$.

 \medskip Conclude from the fact that conditions b4 and b5 are satisfied: 
 
 for all $t$ in the border $S'$ of $S$ such that $t\in_I\sigma$, \\if  $length(t)$ is even, then $\alpha(t)=0$ and, if $length(t)$ is odd, then $\alpha(t)=1$, and also:  
 
 for all $u$ in the border $S'$ of $S$ such that $u\in_I\tau$, \\if  $length(u)$ is even, then $(\varphi|\alpha)(u)=0$ and, if $length(u)$ is odd, then $(\varphi|\alpha)(u)=1$.
 
 Conclude $\mathcal{W}^{I}_S(\sigma, \alpha)$ and $\mathcal{W}^{I}_S(\tau, \varphi|\alpha)$ and: both $\alpha$ and  $\varphi|\alpha$ are in $\mathcal{E}_S$. 
 
 \bigskip (ii) Let $S$ in $\mathbf{Hrs}$ and $\varphi:\omega^\omega\rightarrow \omega^\omega$ be given such that $\forall \alpha[\alpha \in \mathcal{E}_S \rightarrow \varphi|\alpha \in \mathcal{A}_S]$.
 
 Find $T$ in $\mathbf{Hrs}$ such that $\forall n[T\upharpoonright\langle n \rangle = S]$.
\footnote{$T$ might be called the {\it successor } of $S$.} 
 
 Define $\psi:\omega^\omega\rightarrow \omega^\omega$ such that $\forall \alpha\forall n[(\psi|\alpha)\upharpoonright \langle n \rangle = \varphi|(\alpha\upharpoonright\langle n  \rangle )]$.
 
 Note: $\forall \alpha[\alpha\in\mathcal{A}_T\rightarrow \forall n[(\psi|\alpha)\upharpoonright\langle n \rangle\in\mathcal{A}_S]$. 
 
 Find $\rho:\omega^\omega\rightarrow\omega^\omega$ such that $\forall \beta[\forall n[\beta\upharpoonright\langle n \rangle \in \mathcal{A}_S] \leftrightarrow \rho|\beta \in \mathcal{A}_S]$.\footnote{Proving the existence of $\rho$ is equivalent to proving that the class $\mathbf{\Pi}^0_S$ is closed under the operation of countable intersection.  We leave it to the reader to find this proof.}
 
 Define $\zeta:\omega^\omega\rightarrow \omega^\omega$ such that, for all $\alpha$, for all $n$, $(\zeta|\alpha)\upharpoonright\langle n \rangle = \rho |(\psi|\alpha)$. 
 
 Note: for every $\alpha$, if $\alpha \in \mathcal{A}_T$, then $\zeta|\alpha \in \mathcal{E}_T$. 
 
 Using (i), find $\alpha$ in $\mathcal{E}_T$ such that $\zeta|\alpha \in \mathcal{E}_T$. 
 
 Find $n$ such that $\alpha\upharpoonright \langle n \rangle \in \mathcal{A}_{T\upharpoonright \langle n \rangle}=\mathcal{A}_S$.
 
 Find $m$ such that $(\zeta|\alpha)\upharpoonright\langle m \rangle \in \mathcal{A}_{T\upharpoonright\langle m \rangle}=\mathcal{A}_S$.
 
 Note: $(\zeta|\alpha)\upharpoonright \langle m \rangle = \rho|(\psi|\alpha) \in \mathcal{A}_S$ and conclude: \\$(\psi|\alpha)\upharpoonright\langle n \rangle =\varphi|(\alpha\upharpoonright\langle n \rangle) \in \mathcal{A}_S$.
 
 Defining $\beta:=\alpha\upharpoonright \langle n \rangle$, we thus see: both $\beta$ and $\varphi|\beta$ are in $\mathcal{A}_S$. 
 \end{proof}

 \section{The Fan Theorem}\label{S:ft}
 
 \subsection{Finitary spreads}\begin{definition}
 A spread-law $\beta$ will be called \emph{finitary} if it satisfies the following condition: \begin{quote} $\forall s[\beta(s)=0\rightarrow \exists m\forall n[\beta(s\ast\langle n \rangle)=0\rightarrow n\le m]]$. \end{quote}
 
If the spread-law $\beta$ is finitary, the corresponding spread $\mathcal{F}_\beta$ will be called a \emph{finitary spread} or a \emph{fan}. \end{definition}

When I am creating an element $\alpha$ of a {\it fan} $\mathcal{F}_\beta$, then, at each stage $n$,  having completed \begin{quote} $\alpha(0), \alpha(1), \ldots, \alpha(n-1)$ \end{quote} I only have {\it finitely many choices} for the next value, $\alpha(n)$. 

\smallskip An important example of a fan is {\it Cantor space} $2^\omega:=\mathcal{C}:=\{\alpha\mid\forall n[\alpha(n)\le 1]\}$. 

\begin{definition} For all $\mathcal{X}\subseteq \omega^\omega$, for all $B\subseteq \omega$,  we define: \emph{$B$ is a bar in $\mathcal{X}$, $Bar_\mathcal{X}(B)$,}  if and only if $\forall \alpha \in \mathcal{X}\exists n[\overline \alpha n \in B]$. \end{definition}

\begin{theorem}[Fan Theorem]\label{T:fantheorem} Let $\beta$ be a finitary spread-law. \\If $B\subseteq \omega$ is a bar in $\mathcal{F}_\beta$, some finite $B'\subseteq B$ is  bar in $\mathcal{F}_\beta$. \end{theorem}

\begin{proof} Assume $\beta$ is a finitary spread-law and let $B\subseteq \omega$ be a bar in  $\mathcal{F}_\beta$. 

How may I have convinced myself that $B$ is indeed a bar in $\mathcal{F}_\beta$? 

(Under what circumstances shall we say that this conclusion is justified? Some agreement here is, intuitionistically,  the only way to make sense of the statement). 

Let us define, for each $s$ such that $\beta(s)=0$, \\
\emph{$B$ bars $s$ in $\mathcal{F}_\beta$} if and only if $Bar_{\mathcal{F}_\beta\cap s}(B)$, i.e. $\forall \alpha \in \mathcal{F}_\beta[ s\sqsubset \alpha\rightarrow \exists n[\overline \alpha n \in B]]$.  

Now  observe the following:

(i) For each $s$, if $\beta(s) =0$ and $s\in B$, then $B$ bars $s$ in $\mathcal{F}_\beta$.

(ii) For each $s$, if $\beta(s)=0$ and, for every $n$ such that $\beta(s\ast\langle n \rangle)=0$, $B$ bars $s\ast\langle n \rangle$ in $\mathcal{F}_\beta$, then $B$ bars $s$ in $\mathcal{F}_\beta$. 

(iii) For each $s$, if $\beta(s)=\beta(s\ast\langle n \rangle)=0$ and $B$ bars $s$ in $\mathcal{F}_\beta$, then $B$ bars $s\ast\langle n \rangle$ in $\mathcal{F}_\beta$. 

\smallskip 
Note that one may {\it prove} a statement of the form `$B$ bars $s$ in $\mathcal{F}_\beta$' by starting from observations of the form (i) and using observations of the form (ii) and (iii) as reasoning steps.

 Let us now agree to consider the statement `$B$ bars $s$ in $\mathcal{F}_\beta$' as  established or {\it true} if and only if we are able to provide such a {\it canonical} proof. 
 
 \smallskip
  This agreement marks an important point in the development of our intuitionistic mathematics. We are introducing an {\it axiomatic assumption}.
  
\smallskip  If we do so, we may argue as follows. 
 
 Assume $Bar_{\mathcal{F}_\beta}(B)$, i.e. $B$ bars $\langle \; \rangle$ in $\mathcal{F}_\beta$. 
 
 Find a canonical proof of this statement.
 
 Now replace in this canonical proof every statement `$B$ bars $s$ in $\mathcal{F}_\beta$' by `$B$ \emph{finitely} bars $s$ in $\mathcal{F}_\beta$' where the latter means: \begin{quote} some finite $B'\subseteq B$ bars $s$ in $\mathcal{F}_\beta$. \end{quote}
 
 Under this replacement our canonical proof changes into another valid proof. In order to see this, we have to use the fact that a finite union of finite sets of integers is itself a finite set of integers.
 
 The conclusion of the new proof will be: `\emph{some finite $B'\subseteq B$ bars $\langle\;\rangle$ in $\mathcal{F}_\beta$}' and that is what we wanted to establish. 
\end{proof}
\subsection{The Uniform-Continuity Theorem}
The following result is the first application of the Fan Theorem.
\begin{theorem}\label{T:uct} Every pointwise continuous function from $[0,1]$ to $\mathcal{R}$ is uniformly continuous on $[0,1]$.\end{theorem}

\begin{proof} Let $f$ be a pointwise continuous function from $[0,1]$ to $\mathcal{R}$. 

We first define $\rho$ such that, for each $s$ in $2^{<\omega}$, $\rho(s)=\bigl(\rho'(s), \rho''(s)\bigr)$ is a pair of rationals.   We define $\rho$ by induction on the length of the argument.

We define $\rho(\langle\;\rangle)= (0,1)$ and, for each $s$ in $2^{<\omega}$, \\$\rho(s\ast \langle 0 \rangle))= \bigl(\rho'(s), \frac{1}{3}\rho'(s) + \frac{2}{3}\rho''(s)\bigr)$ and $\rho(s\ast\langle 1 \rangle)=\bigl(\frac{2}{3}\rho'(s) +\frac{1}{3}\rho''(s), \rho''(s)\bigr)$.

\smallskip
We intend to prove, for each $m$, \begin{quote} $\exists n \forall x \in [0,1]\forall y \in [0,1][|x-y|<\frac{1}{2^n}\rightarrow |f(x)-f(y)|<\frac{1}{2^m}]$. \end{quote}

Let $m$ be given.

We define, for all rationals $p,q$ such that $0\le p < q\le 1$, \\$[p,q]$ is \emph{fine} if and only if  $\exists n \forall x \in [p,q]\forall y \in [p,q][|x-y|<\frac{1}{2^n}\rightarrow |f(x)-f(y)|<\frac{1}{2^m}]$.

We want to prove: $[0,1]$ is fine.

Let $B$ be the set of all $s$ in $2^{<\omega}$ such that
$[\rho'(s), \rho''(s)]$ is fine.

We first prove: $B$ is a bar in Cantor space $2^\omega$. 

Let $\alpha$ in $2^\omega$ be given. Find a real $x$ such that, for all $n$, $\rho'(\overline \alpha n)\le x \le \rho''(\overline \alpha n)$. As $f$ is continuous at $x$, find $l$ such that $\forall y \in [0,1][|x-y|<\frac{1}{2^l} \rightarrow |f(x)-f(y)|<\frac{1}{2^{m+1}}]$. Note: $\rho''\bigl(\overline \alpha(2l)\bigr)-\rho'\bigl(\overline \alpha(2l)\bigr)= (\frac{2}{3})^{2l}<(\frac{1}{2})^l$. Conclude: $\overline \alpha(2l)\in B$.

We thus see that $B$ is a bar in $2^\omega$.

\smallskip

One easily verifies:  for all $s$ in $2^{<\omega}$, \\$s\in B$ if and only if both $s\ast\langle 0\rangle \in B$ and $s\ast\langle 1 \rangle \in B$.

Now find a canonical proof of: `$B$ bars $\langle\;\rangle$ in $2^\omega$' and replace, in this proof, every statement: `$B$  bars $s$ in $2^\omega$' by `$s \in B$'. 

The result will be a valid proof, and the conclusion of the proof is: `$\langle \;\rangle \in B$', i.e. `$[0,1]$ is fine'.
\end{proof}

 \section{Measure and Integration}\label{S:measure}
 
 Brouwer worked on the theory of measure and integration, following the lead of H. Lebesgue, see \cite[Chapter VI]{heyting56}. Bishop chose for an approach inspired by P.J. Daniell see \cite[Chapter 6]{bishopbridges85}.  We here return to Brouwer's approach.
 
 \subsection{A note on real numbers}\label{SS:reals} A real number  is an infinite sequence \\$x=x(0), x(1), \ldots$ of pairs $x(n)=\bigl(x'(n), x''(n)\bigr)$ of  rationals such that \begin{enumerate} \item $x$ is \emph{shrinking}, i.e. for all $n$, $x'(n)\le x'(n+1)\le x''(n+1)\le x''(n)$, and \item $x$ is \emph{dwindling}, i.e. for every $m$, there exists $n$ such that $x''(n)-x'(n)<\frac{1}{2^m}$. \end{enumerate}
 
 $\mathcal{R}$ denotes the set of the real numbers. For all $x,y$ in $\mathcal{R}$, one defines \begin{enumerate} \item $x<_\mathcal{R} y$ if and only if, for some $n$, $x''(n)<y'(n)$, and \item $x\le_\mathcal{R} y$ if and only if, for all $n$, $x'(n)\le y''(n)$, and \item $x=_\mathcal{R}y$ ($x$ \emph{really-coincides with} $y$, $x$ is \emph{(really) equal} to $y$), \\if and only if $x\le_\mathcal{R} y$ and $y\le_\mathcal{R} x$. \end{enumerate}
 
 If confusion seems unlikely, we omit the subscript `$\mathcal{R}$'.
 
 One may prove \emph{Cantor's Intersection Theorem}: \begin{quote} Given an infinite sequence $(x_0, y_0), (x_1, y_1), \ldots$ of pairs of reals that is \emph{shrinking}, i.e. for all $n$, $x_n\le x_{n+1}\le y_{n+1}\le y_n$, and \emph{dwindling}, i.e. for all $m$ there exists $n$ such that $y_n-x_n < \frac{1}{2^m}$, then there exists a real $z$ such that, for all $n$, $x_n\le z \le y_n$, and, for each real $t$, if, for all $n$, $x_n\le_\mathcal{R} t \le_\mathcal{R} y_n$, then $t=_\mathcal{R} z$. 
 \end{quote}
 
 $[0,1]:= \{x\in \mathcal{R}\mid 0\le x \le 1\}$. 
 
 \smallskip We will treat rationals and also pairs of rationals as if they were natural numbers. This approach may be made precise by suitable coding functions, see \cite[Section 8]{veldman2011b}. 
 \subsection{Measurable open sets} \begin{definition}\label{D:opensetsmeasurable} Let $(q_0, r_0), (q_1, r_1), \ldots$ be an enumeration of all pairs $(q,r)$ of rationals such that $q\le  r$.

 For all $n$, for all $a$ in $\omega^n$, we define $\mathcal{H}_a :=\{x\in \mathcal{R}\mid \exists j < n[ q_{a(j)}< x<r_{a(j)}]\}.$
 
 For all $\alpha$ in $\omega^\omega$, we define $\mathcal{H}_\alpha :=\{x\in \mathcal{R}\mid \exists j[q_{\alpha(j)}<x<r_{\alpha(j)}]\}=\bigcup_n \mathcal{H}_{\overline\alpha n}$. 
 
 \smallskip $\mathcal{X}\subseteq [0,1]$ is called \emph{open} if and only if there exists $\alpha$ such that $\mathcal{X}=\mathcal{H}_\alpha\cap [0,1]$.
 
 \end{definition}
 \begin{definition} For each $n$, for each $b$ in $\omega^n$, 
 $b$ is \emph{neatly increasing} if and only if \\$\forall j<n[q_{b(j)}< r_{b(j)}]$ and $\forall j<n-1[r_{b(j)}\le q_{b(j+1)}]$. \end{definition}
 
 \begin{remark} For each $a$, there exists exactly one $b$ such that $b$ is neatly increasing and $\mathcal{H}_a=\mathcal{H}_b$.\end{remark}
 
 The proof of this fact is left to the reader.
 
 \begin{definition} For each $a$, we define  $\mu(a) := \sum_{j<length(b)} r_{b(j)}-_\mathbb{Q} q_{b(j)}$, where $b$ is neatly increasing and satisfies $\mathcal{H}_a =\mathcal{H}_b$.

  \medskip  $\alpha$ is called \emph{measurable} if and only if $\mu(\alpha):=\lim_{n\rightarrow \infty} \mu( \overline \alpha n)$ exists.   \end{definition}
  
  \begin{definition}\label{D:intersectionsegments} For all rationals $q,r,s,t$ such that $q\le r$ and $s\le t$ we define a pair of rationals called  $(q,r)\cap(s,t)$ as follows.
  
  If $r<s$ or $t<q$, we define  $(q,r)\cap(s,t)= (0,0)$, and,\\ if $s\le r$ and $q\le t$ we define: $(q,r)\cap (s,t)=(\max (q,s), \min(r,t))$. \end{definition}
  
  \begin{definition}\label{D:localizingmeasure}
  
  Let $\alpha$ be measurable and let rationals $q,r$ be given such that $q< r$. Find $\beta$ such that, for each $j$, $(q_{\beta(j)}, r_{\beta(j)})=(q_{\alpha(j)}, r_{\alpha(j)})\cap(q,r)$.   Note that $\beta$ is measurable. We define: $\mu\upharpoonright[q,r](\alpha):= \mu(\beta)$.

   For each $n$, we define: \emph{$\alpha$ covers $(q,r)$ for more than $1-\frac{1}{n}$} if and only if  \\$\mu\upharpoonright[q,r](\alpha)> (1-\frac{1}{n})(r-q)$.
   
   We also define: \emph{$\alpha$ never covers $(q,r)$} if and only if $\mu\upharpoonright[q,r](\alpha) <r-q$.\end{definition}
   
   \begin{lemma}\label{L:measurecover} Let $\alpha$ be measurable and let rationals $p,q,r$ be given such that $q<p<r$. \begin{enumerate}[\upshape (i)]\item  For each $n$, either $\alpha$ covers $(q,r)$ for more than $1-\frac{1}{n}$, or $\alpha$ never covers $(q,r)$. \item $\mu\upharpoonright[q,r](\alpha)=\mu\upharpoonright[q,p](\alpha)+\mu\upharpoonright[p,r](\alpha)$.
   
   \item For all $n$,
   if $\alpha$ covers both $(q,p)$ and $(p,r)$ for more than $1-\frac{1}{n}$, then $\alpha$ covers $(q,r)$ for more than $1-\frac{1}{n}$. \item
   If $\alpha$ never covers $(q,r)$, then {\it either} $\alpha$ never covers $(q,p)$ {\it or} $\alpha$ never covers $(p,r)$.\end{enumerate}\end{lemma}
 \begin{proof} The proof of these statements is left to the reader. \end{proof}  
   
 \begin{lemma}\label{L:measurable} Let $\alpha, \beta$ be given such that  both $\alpha, \beta$ are measurable. \begin{enumerate}[\upshape (i)]  \item If $\mu(\alpha)<\mu(\beta)$ and $\mathcal{H}_\alpha\subseteq \mathcal{H}_\beta$, then one may find an element of $\mathcal{H}_\beta\setminus \mathcal{H}_\alpha$. 
 \item if $\mathcal{H}_\beta\subseteq \mathcal{H}_\alpha$, then $\mu(\beta)\le\mu(\alpha)$, and,  if $\mathcal{H}_\beta= \mathcal{H}_\alpha$, then $\mu(\beta)=\mu(\alpha)$. \end{enumerate}\end{lemma}

 \begin{proof} (i) Find $n$ such that $\mu(\overline \beta n)>\mu(\alpha)$. Find $b$ such that $b$ is neatly increasing and $\mathcal{H}_b = \mathcal{H}_{\overline \beta n}$. Find $j<length(b)$ such that $\alpha$ never covers $ (q_{b(j)}, r_{b(j)})$. \\Using Lemma \ref{L:measurecover},  define a real $x$ such that $x(0)= (q_{b(j)}, r_{b(j)})$ and, for each $n$, \\$x'(n)<x'(n+1)<x''(n+1)<x''(n)$ and $\alpha$ never covers $x(n)$. \\Note: $x\in \mathcal{H}_\beta\setminus \mathcal{H}_\alpha$.
 
 \smallskip (ii) This easily follows from (i).
  \end{proof}
   
   Note that, in the proof of Lemma \ref{L:measurable}, we did not use the Fan Theorem.
  \subsection{On the complement of a measurable open set}

 \begin{definition}\label{D:varphibin} We define a function $B$ associating to every  $a$ in $2^{<\omega}$  a pair of rationals $B(a)= \bigl(B'(a)
  , B''(a)\bigr)$. 
  
  $B(\langle\;\rangle)=(0,1)$ and for all $a$ in $2^{<\omega}$, we consider $M(a):=\frac{B'(a)+B''(a)}{2}$ and then define $B(a\ast\langle 0 \rangle)= \bigl(B'(a) , M(a)\bigr)$ and $B(a\ast\langle 1 \rangle)= \bigl(M(a), B''(a)\bigr)$.
  
 \smallskip We also define a function $\varphi_{bin}$ from Cantor space $2^\omega$ to $[0,1]$, by the following. 
 
 For each $\gamma$ in $2^\omega$, for each $n$, $(\varphi_{bin}|\gamma)(n) = B(\overline \gamma n)$. \end{definition}
 
  Note that, for all $\gamma$ in $2^\omega$, $\varphi_{bin}|\gamma =_\mathcal{R} \sum_{n=0}^\infty \gamma(n)\cdot 2^{-n-1}$.  
 
 \smallskip Note that, constructively, $\varphi_{bin}$ is {\it not} a surjective mapping of Cantor space $\mathcal{C}$ onto $[0,1]$. The set $\varphi_{bin}|2^\omega$ consists of all $x$ in $[0,1]$ that admit of a {\it binary expansion}, i.e. for all $m$, for all $i<2^m$, one may decide: $x\le \frac{i}{2^m}$ or $\frac{i}{2^m} \le x$.

 The following Lemma shows that, within the complement of a `small'  measurable subset of $[0,1]$, one may construct `large',   `compact' sets. 
 \begin{lemma}\label{L:complopen}  Let $n>0$ and $ \alpha$ be given such that $\alpha$ is measurable and $\mu(\alpha)<\frac{1}{2^{n+2}}$. There exists a 
 fan-law $\beta$ in $2^\omega$ such that \begin{enumerate} \item for all $a$, if $\beta(a)=0$, then $a \in 2^{<\omega}$, and $\alpha$ never covers $B(a)$,  \item  for all $a$ in $2^{<\omega}$, for all $i<2$, if $\beta(a)= 0$ and $\beta(a\ast\langle i \rangle)=1$, \\then $\alpha$ covers $B(a\ast \langle i \rangle)$ for more than $\frac{1}{2}$, and \item there exists a measurable $\alpha^+$ such that  $\mu(\alpha^+)<\frac{1}{2^n}$ and, \\for all $x$ in $[0,1]$, if $x\notin \mathcal{H}_{\alpha^+}$, then  $x \notin \mathcal{H}_\alpha$ and $\exists \gamma \in \mathcal{F}_\beta[\varphi_{bin}|\gamma = x]$. \end{enumerate}\end{lemma}
 
 \begin{proof} Let $n>0$ and $ \alpha$ be given such that $\alpha$ is measurable and $\mu(\alpha)<\frac{1}{2^{n+2}}$.
 
 We define the promised fan-law $\beta$ as follows.

 For each $a$, if $a\notin 2^{<\omega}$, then $\beta(a)= 1$.
 
 For $a$ in $2^{<\omega}$,  $\beta(a)$ is defined by induction on $length(a)$. We will take care that, for each $a$ in $2^{<\omega}$, if $\beta(a) =0$, then $\alpha$ never covers $B(a)$.

  Define $\beta(\langle\;\rangle)=0$ and note: $\alpha$ never covers $B(\langle\;\rangle)$.
  
   Now assume: $a\in 2^{<\omega}$, and $\beta(a)$ has been defined.
   
   If $\beta(a)=1$, define $\beta(a\ast \langle 0\rangle)= \beta(a\ast\langle 1 \rangle)=1$.
   
   If $\beta(a) = 0$, we may assume: $\alpha$ never covers $B(a)$. 
   
   Find $i<2$ such that $\alpha$ never covers $B(a\ast\langle i \rangle)$ and define $\beta(a\ast\langle i \rangle)=0$.
   
   Then consider $a\ast\langle 1-i\rangle$ and note: \\{\it either} $\alpha$ never covers $B(a\ast\langle 1-i\rangle)$ {\it or} $\alpha$ covers $B(a\ast\langle 1-i\rangle)$ for more than $\frac{1}{2}$. 
   
   Define $\beta(a\ast\langle 1-i\rangle)$ such that, if $\beta(a\ast \langle 1-i\rangle)=0$, then $\alpha$ never covers $B(a\ast\langle 1 -i\rangle)$, and, if $\beta(a\ast\langle 1 -i\rangle) =1$, then $\alpha$ covers $B(a\ast\langle 1 - i\rangle)$ for more than $\frac{1}{2}$. 
   
   Note that we have some freedom in carrying out this step as the conditions `$a$ never covers $B(a\ast\langle 1-i\rangle)$' and `$a$ covers $B(a\ast\langle 1-i\rangle)$ for more than $\frac{1}{2}$' do not exclude each other.
   
   Define $C:=\{a\ast\langle i\rangle \mid a\in 2^{<\omega}, i<2 \mid \beta(a)= 0 \;\wedge\; \beta(a\ast\langle i \rangle)=1\}$.
   
     Note that, for all $a,b$ in $C$, if  $a\neq b$ then $B(a)\cap B(b)=(0,0)$.
   
   Define $\alpha^\dag$ such that, for each $n$, {\it if} there exists    $a$ in $C$ such that  \\$(q_n, r_n)=\bigl(B'(a), B''(a)\bigr)$, then $\alpha^\dag(n)=n$ and, {\it if not}, then $q_{\alpha^\dag(n)} = r_{\alpha^\dag(n)}=0$.

   Note: $\mathcal{H}_{\alpha^\dag} =\bigcup_{a\in C}\bigl(B'(a), B''(a)\bigr)$.
   
   We now prove that $\alpha^\dag$ is measurable. \\ Let $k$ be given. Find $p$ such that $\mu(\alpha)-\mu(\overline \alpha p) < \frac{1}{2^{k+2}}$. \\Find $m$ such that $\mu(\overline{\alpha^\dag}m)>\mu(\overline \alpha p)-\frac{1}{2^{k+2}}$. \\Note that, for all $n>m$, $\mu(\overline{\alpha^\dag}n) - \mu(\overline{\alpha^\dag} m)< 2\cdot\frac{1}{2^{k+2}} + 2\cdot\frac{1}{2^{k+2}} =\frac{1}{2^k}$. 
   \\We thus see that $\mu(\alpha^\dag)=\lim_{n\rightarrow\infty}\mu(\overline{\alpha^\dag}n)$ exists, i.e. $\alpha^\dag$ is measurable.

   \smallskip Note: $\mu(\alpha)<\frac{1}{2^{n+2}}$ and for each $a$ in $C$,  $\alpha$ covers $B(a)$ for more than $\frac{1}{2}$. \\Conclude: $\mu(\alpha^\dag)\le 2\mu(\alpha)< \frac{1}{2^{n+1}}$.

   \smallskip Finally, note that, for each $\gamma$ in $\mathcal{F}_\beta$, for each $n$, $\alpha$ does not cover \\$(\varphi_{bin}|\gamma)(n)= B(\overline \gamma n)$, and, therefore: $\varphi_{bin}|\gamma \notin \mathcal{H}_\alpha$.

   Now define $\alpha^+$ such that,  for each $a$, \\ $\alpha^+(2a)= \alpha^\dag(a)$, and  $\alpha(2a+1)=(q_a-\frac{1}{2^{n+a+3}}, q_a+\frac{1}{2^{n+a+3}})$. 
   
    Note that $\alpha^+$ is measurable and $\mu(\alpha^+)\le \mu(\alpha^\dag) + \frac{1}{2^{n+2}}<\frac{1}{2^n}$.

   Assume $x\in [0,1]$ and  $x \notin \mathcal{H}_{\alpha^+}$. Then $\forall q \in \mathbb{Q}[q \;\#\; x]$. We thus may find $\gamma$ in $2^\omega$ such that $\varphi_{bin}|\gamma = x$.  Note: $x\notin\mathcal{H}_{\alpha^+}$  and thus, for each $n$, $\overline \gamma n\notin C$. Conclude: $\forall n[\beta(\overline \gamma n)=0]$ and: $\gamma \in \mathcal{F}_\beta$.
   
   We thus see: $\forall x \in [0,1][x \notin \mathcal{H}_{\alpha^+}\rightarrow \exists \gamma \in \mathcal{F}_\beta[\varphi_{bin}|\gamma = x]]$.  
\end{proof}
 \subsection{Almost-full subsets of $[0,1]$ and  almost-full functions from $[0,1]$ to $[-1,1]$}
 
 \begin{definition} $\mathcal{X}\subseteq [0,1]$ will be called \emph{almost-full} if and only if, for each $n$, there exists a measurable $\alpha$ such that $\mu(\alpha)<\frac{1}{2^n}$ and  $\forall x \in [0,1][x\notin \mathcal{H}_\alpha \rightarrow x  \in \mathcal{X}]$. 
 
 \smallskip A partial function $f$ from $[0,1]$ to $[-1,1]$ will be called \emph{almost-full} if and only if its domain $Dom(f) = \{x\in [0,1]\mid \; f(x)\; is \; defined \}$ is almost-full. \end{definition}

 \begin{definition}\label{D:measurablefunction} Let $f$ be an almost-full function from $[0,1]$ to $[-1,1]$. 
 
 $f$ is called \emph{measurable} if and only if, for each $n$, one may find $m$ and rationals $u_0, u_1, \ldots, u_{2^m-1}$ such that for each $i<2^m$, $-1\le u_i\le 1$, and a measurable $\alpha$ such that $\mu(\alpha)<\frac{1}{2^{n+2}}$ and, for each $x$ in $Dom(f)$, for each $i<2^m$, if $x\notin \mathcal{H}_\alpha$ and $\frac{i}{2^m}< x <\frac{i+1}{2^m}$, then $|f(x)-u_i|<\frac{1}{2^{n+1}}$.
 
 \smallskip The number $\sum_{i<2^{m}} u_i\cdot\frac{1}{2^{m}}$ will be called an \emph{estimate of the integral of $f$ of accuracy $\frac{1}{2^{n}}$.}  \end{definition}
  
\begin{theorem}\label{T:integral} Let $f$ be an almost-full  and measurable function from $[0,1]$ to $[-1,1]$. \begin{enumerate}[\upshape (i)] \item For all rationals $q,r$, if $q,r$ are estimates of the integral of $f$ of accuracy  $\frac{1}{2^m}$, $\frac{1}{2^n}$, respectively, then $|q-r| \le \frac{1}{2^m} + \frac{1}{2^n}$. 

\item  There exists a real $x$ such that, for all $n$, for all estimates $q$ of the integral of $f$ of accuracy $\frac{1}{2^n}$, $|x-q|\le \frac{1}{2^n}$. \end{enumerate}\end{theorem}

\begin{proof} The proof of this Theorem is left to the reader. \end{proof}
The number intended in Theorem \ref{T:integral}(ii) is unique up to the relation of real coincidence and will be called: $\int_0^1f$, the {\it integral of $f$ on $[0,1]$}. 
 
 For the next result, see \cite[page 7]{rootselaar54} and \cite[Section 6.2.2, Theorem 1]{heyting56}.
 
 \begin{theorem}[van Rootselaar]$\;$\\ Every almost-full function from $[0,1]$ to $[-1,1]$ is measurable. \end{theorem}
 
 \begin{proof} Let $f$ be an almost-full function from $[0,1]$ to $[-1,1]$. 
 
 Let $n$ be given. 
 
 Find a measurable $\alpha$ such that $\mu(\alpha)<\frac{1}{2^{n+4}}$ and \\$\forall x \in [0,1][x\notin\mathcal{H}_\alpha\rightarrow x \in Dom(f)]$. 
 
 Using Lemma \ref{L:complopen}, find a fan-law $\beta$ in $2^\omega$ such that \begin{enumerate} \item for all $a$, if $\beta(a)=0$, then $a \in 2^{<\omega}$, and $\alpha$ never covers $B(a)$,  \item  for all $a$ in $2^{<\omega}$, for all $i<2$, if $\beta(a)= 0$ and $\beta(a\ast\langle i \rangle)=1$, \\then $\alpha$ covers $B(a\ast \langle i \rangle)$ for more than $\frac{1}{2}$, and \item there exists a measurable $\alpha^+$ such that  $\mu(\alpha^+)<\frac{1}{2^{n+2}}$ and, \\for all $x$ in $[0,1]$, if $x\notin \mathcal{H}_{\alpha^+}$, then  $x \notin \mathcal{H}_\alpha$ and $\exists \gamma \in \mathcal{F}_\beta[\varphi_{bin}|\gamma = x]$. \end{enumerate}\
 
\smallskip  Conclude: $\forall \gamma \in \mathcal{F}_\beta[\varphi_{bin}|\gamma \in Dom(f)]$. 

\smallskip Therefore: $\forall \gamma \in \mathcal{F}_\beta  \exists u \in \mathbb{Q}[|u-f(\varphi_{bin}|\gamma)|<\frac{1}{2^{n+1}}]$. 

\smallskip Using the extension of Brouwer's Continuity Principle to spreads, Theorem \ref{T:bcpext}, conclude: $\forall \gamma \in \mathcal{F}_\beta  \exists m\exists u \in \mathbb{Q}\forall \delta \in \mathcal{F}_\beta[ \overline \gamma m \sqsubset \delta\rightarrow [|u-f(\varphi_{bin}|\delta)|<\frac{1}{2^{n+1}}]$.  

\smallskip Using the Fan Theorem, Theorem \ref{T:fantheorem}, find $m$ such that\\$\forall \gamma \in \mathcal{F}_\beta  \exists u \in \mathbb{Q}\forall \delta \in \mathcal{F}_\beta[ \overline \gamma m \sqsubset \delta\rightarrow |u-f(\varphi_{bin}|\delta)|<\frac{1}{2^{n+1}}]$, i.e.
\\ $\forall a \in 2^{<\omega}_m[\beta(a)=0\rightarrow   \exists u \in \mathbb{Q}\forall \delta \in \mathcal{F}_\beta[ a \sqsubset \delta\rightarrow |u-f(\varphi_{bin}|\delta)|<\frac{1}{2^{n+1}}]]$.  

\smallskip Now find $u_0, u_1, \ldots, u_{2^m-1}$ such that, for all $i<2^m-1$, \\for all $a\in B_m$, if $B(a) = (\frac{i}{2^m}, \frac{i+1}{2^m})$, then, \begin{enumerate} \item if $\beta(a)=0$, then $\forall \delta \in \mathcal{F}_\beta[ a\sqsubset \delta \rightarrow |u_i-f(\varphi_{bin}|\delta)|<\frac{1}{2^{n+1}}]$, and, \item if $\beta(a)=1$, then $u_i=0$. \end{enumerate}  

Note that, for all $x$ in $Dom(f)$, for all $i<2^m$, \\if $x \notin \mathcal{H}_{\alpha^+}$ and $\frac{i}{2^m}<x<\frac{i+1}{2^m}$, then $|f(x)-u_i|<\frac{1}{2^{n+1}}$. 

\smallskip We thus see that, for each $m$, we can make an estimate of the integral of $f$ of accuracy $\frac{1}{2^n}$, and conclude:   $f$ is measurable. 
\end{proof}
\begin{definition} $\mathcal{Y}\subseteq [0,1]$ is \emph{measurable} if and only if its characteristic function \begin{quote} $\chi_\mathcal{Y}:=\{(x,i)\in [0,1]\times\{0,1\}\mid (x \in \mathcal{Y} \wedge i=1)\vee (x \notin \mathcal{Y} \wedge i=0)\}$ \end{quote} is measurable. If $\mathcal{Y}\subseteq [0,1]$ is measurable, we define  $\mu(\mathcal{Y}) := \int_0^1\chi_\mathcal{Y}$. \\ $\mu(\mathcal{Y})$ is called the \emph{measure} of $\mathcal{Y}$. \end{definition}

\begin{corollary}\label{Cor:almostfull} $\mathcal{Y}\subseteq [0,1]$ is measurable if and only if $\mathcal{Y}\cup ([0,1]\setminus\mathcal{Y})$ is almost-full. \end{corollary}
\subsection{Many integrable sets}\label{SS:manyintegrable}

We want to prove an important result overlooked by Brouwer and Heyting: \begin{quote}for every almost-full and therefore measurable function $f$  from $[0,1]$ to $[-1,1]$, for all but countably many $y$ in $[-1,1]$, \\the set $\{x\in [0,1] \mid x\in Dom(f)\;\wedge\;f(x) <y\}$ is measurable. \end{quote}We avoid the theory of \textit{profiles} developed by E. Bishop for the purpose of proving this theorem, see \cite[Chapter 6, Section 4 and Theorem 4.11]{bishopbridges85}. Our approach is the one followed in \cite{driessen}.

\begin{lemma}\label{L:findingapproximateoriginal} Let $f$ be a measurable function from $[0,1]$ to $[-1,1]$. Let rationals $q,r$ be given such that $-1\le q < r\le 1$. Let $n$ be given.

One may find rationals $s,t$ such that $q\le s<t\le r$ and a measurable $\alpha$ such that $\mu(\alpha)<\frac{1}{2^n}$ and $\forall x\in Dom(f)[x\notin \mathcal{H}_\alpha\rightarrow \bigl(f(x)<s \;\vee\; t<f(x)\bigr)]$. \end{lemma}

\begin{proof}Let $f$ be a measurable function from $[0,1]$ to $[-1,1]$. Let rationals $q,r$ be given such that $-1\le q < r\le 1$. Let $n$ be given.

\smallskip Find $l$ such that $\frac{1}{2^l}<\frac{r-q}{2^{n+1}}$. Note: $r-q\le 2$ and $l> n$.

Using Definition \ref{D:measurablefunction}, find $m$ and rationals $u_0, u_1, \ldots, u_{2^m-1}$ such that, for each $i<2^m$, $-1\le u_i\le 1$, and  a measurable $\alpha$ such that $\mu(\alpha)<\frac{1}{2^{l+2}}$ and, for each $x$ in $Dom(f)$, for each $i<2^m$, if $x\notin \mathcal{H}_\alpha$ and $\frac{i}{2^m}< x < \frac{i+1}{2^m}$, then $|f(x)-u_i|<\frac{1}{2^{l+2}}$.

Define $j_0:=\mu j[r<q+\frac{j+1}{2^l}]$. 

Note: $2^{n+1}\cdot\frac{1}{2^l}< r-q$, so $2^{n+1}\le j_0$. 

For each $j< j_0$,  consider the set $A_j:=\{i<2^m\mid q+\frac{j}{2^l}\le u_i < q+\frac{j+1}{2^l}\}$.

Note: for all $j<k< j_0$, $A_j\cap A_k = \emptyset$.

Find $j_1\le j_0$ such that, for all $j\le j_0$, $Card(A_{j_1})\le Card(A_j)$.

Note: $Card(A_{j_1})\le \frac{2^m}{j_0}\le\frac{2^m}{2^{n+1}}=2^{m-n-1}$.

Define $s:= q+\frac{j_1}{2^l}+\frac{1}{2^{l+2}}$ and $t:=q+\frac{j_1+1}{2^l}-\frac{1}{2^{l+2}}$ and note: $q<s<t<r$.

Note: for each $i<2^m$, if $i\notin A_{j_1}$, then {\it either} $u_i<q+\frac{j_1}{2^l}$ and,\\ for all $x$ in $Dom(f)\cap (\frac{i}{2^m}, \frac{i+1}{2^m})$, if $x\notin \mathcal{H}_\alpha$, then $f(x)< q+\frac{j_1}{2^l}+\frac{1}{2^{l+2}} = s$,
\\{\it or} $q+\frac{j_1+1}{2^l}\le u_i$ and, \\
for all $x$ in $Dom(f)\cap (\frac{i}{2^m}, \frac{i+1}{2^m})$, if $x \notin \mathcal{H}_\alpha$, then  $t= q+\frac{j_1+1}{2^l}-\frac{1}{2^{l+2}}< f(x) $.
 
\smallskip One now may  define $\alpha^+$ such that $\alpha^+$ is measurable and $\mathcal{H}_\alpha\subseteq \mathcal{H}_{\alpha^+}$ and, for each $i$ in $A_{j_1}$, $(\frac{i}{2^m}, \frac{i+1}{2^m})\subseteq \mathcal{H}_{\alpha^+}$ and, for each $i\le 2^m$, $\frac{i}{2^m}\in \mathcal{H}_{\alpha^+}$ and \\$\mu(\alpha^+)<\mu(\alpha) + \sum_{i\in A_{j_1}} \frac{1}{2^m}<\frac{1}{2^{l+2}} + \frac{1}{2^{n+1}}< \frac{1}{2^n}$.

 Note: $\alpha^+$ is measurable and $\mu(\alpha^+)<\frac{1}{2^n}$, and, for all $x$ in $Dom(f)$, if $x\notin \mathcal{H}_{\alpha^+}$, then either $f(x)<s$ or $t<f(x)$. 

 \end{proof}
 
 \begin{theorem}\label{T:findingoriginal} Let $f$ be a measurable function from $[0,1]$ to $[-1,1]$. Let rationals $u,v$ be given such that $-1\le u<v\le 1$.
 
 There exists $y$ in $(u,v)$ such that, for almost all $x$ in $[0,1]$, $x\in Dom(f)$ and $f(x) <y$ or $y<f(x)$. 
 \end{theorem} 
 
 \begin{proof} Applying Lemma \ref{L:findingapproximateoriginal}, we find an infinite sequence $(u_0, v_0), (u_1, v_1), \ldots$ of pairs of rationals, and an infinite sequence $\alpha_0, \alpha_1, \ldots$ of measurable elements of $\omega^\omega$ such that \begin{enumerate} \item $(u_0, v_0)= (u, v)$, \item for each $n$, $u_n < u_{n+1} < v_{n+1} < v_n$,  and $v_{n+1} - u_{n+1}<\frac{1}{2^n}$, \item for each $n$, $\mu(\alpha_n) < \frac{1}{2^n}$ and $\forall x \in Dom(f)\setminus \mathcal{H}_{\alpha_n}[f(x) < u_n \;\vee \;v_n < f(x)]$. \end{enumerate}
 
 Find $y$ such that, for each $n$, $u_n \le y \le v_n$ and note that $y$ satisfies the requirements. 
 \end{proof}
 \begin{corollary}\label{Cor:denseoriginal} Let $f$ be a measurable function from $[0,1]$ to $[-1.1]$.\\ The set $\{y\in [-1,1]\mid \{x\in [0,1]\mid f(x) <y\}\;is\;measurable\}$ is dense in $[-1,1]$. \end{corollary} \begin{proof} Use Theorem \ref{T:findingoriginal}  and Corollary \ref{Cor:almostfull}. \end{proof}
 \begin{lemma}\label{L:monotonecontinuouspreparation} Let $h$ be a partial function from $[-1,1]$ to $[0, 1]$ such that $ Dom(h)$ is dense in $[-1,1]$ and $h$ is non-decreasing, i.e. \\$\forall x \in Dom(h)\forall y \in Dom(h)[x<y\rightarrow h(x) \le h(y)]$. 
 
 Let $x,y$ in $Dom(h)$ be given such that $x<y$ and $h(x) < h(y)$. 
 
 There exists $z$ in $[x, y]$ such that, for all $u$ in $Dom(h)$, if $u\;\#\; z$, then  \\$\exists t\in Dom(h)\exists w\in Dom(h)[t<u<w \;\wedge\; h(w)-h(t) < \frac{2}{3}\bigl(h(y)-h(x)\bigr)]$. \end{lemma}
 
 \begin{proof} Define $\varepsilon:= h(y)-h(x)$.
 
 A point $u$ in $[x,y]$ will be called {\it neat} if and only if \begin{quote}$\exists t\in Dom(h)\exists w\in Dom(h)[t<u<w \;\wedge\; h(w)-h(t) < \frac{2}{3}\varepsilon]$.\end{quote}
 
 We  define an infinite sequence $(x_0,  y_0), (x_1, y_1), \ldots$ of pairs of elements of $Dom(h)$, such that, for all $n$, $x\le x_n\le x_{n+1}\le y_{n+1}\le y_n\le y$ and $y_n-x_n\le (\frac{2}{3})^n(y-x)$ and every $u$ in $[x,x_n)\cup(y_n, y]$ is neat.  
 
 We first define: $(x_0, y_0):= (x,y)$, 
 
 Now let $n$ be given such that $(x_n, y_n)$ has been defined already such that $x_n< y_n$.
 
 Determine  $a$ in $Dom(h)$ such that $\frac{2}{3}x_n + \frac{1}{3}y_n < a < \frac{1}{3}x_n + \frac{2}{3} y_n$.
 
 Note: $a-x_n<\frac{2}{3}(y_n-x_n)$ and also: $y_n-a<\frac{2}{3}(y_n-x_n)$.
 
 Note: $h(x)\le h(a) \le h(y)$. 
 
  Either $\frac{1}{2}\varepsilon<h(a)-h(x)$ or  $h(a) -h(x) <\frac{2}{3}\varepsilon$, and also: either $\frac{1}{2}\varepsilon<h(y)-h(a)$ or  $h(y) -h(a) <\frac{2}{3}\varepsilon$, but not both $\frac{1}{2}\varepsilon<h(a)-h(x)$ and $\frac{1}{2}\varepsilon<h(y)-h(a)$.

Note that, if $h(a)-h(x)<\frac{2}{3}\varepsilon$, then every $u$ in $[x,a)$ will be neat, and, if $h(y)-h(a)<\frac{2}{3}\varepsilon$, then every $u$ in $(a, y]$ will be neat.
 
  We  define $(x_{n+1}, y_{n+1})$  in such a way that {\it either} $(x_{n+1}, y_{n+1})=(x_n, a)$ and $h(y)-h(a) < \frac{2}{3}\varepsilon$,  {\it or} $(x_{n+1}, y_{n+1})=(a, y_n)$ and $h(a)-h(x) < \frac{2}{3}\varepsilon$.
  
  Then every $u$ in $[x,x_{n+1})\cup(y_{n+1}, y]$ will be neat.
  
  \smallskip Clearly, the infinite sequence $(x_0, y_0), (x_1, y_1), \ldots$ satisfies our requirements.
  
  Using Cantor's Intersection Theorem, find $z$ such that, for all $n$, $x_n\le z \le y_n$ and note: for all $u$ in $[0,1]$, if $u\;\#\;z$, then, for some $n$, $u<x_n$ or $y_n<u$ and: $u$ is neat.\end{proof}
  
  \begin{definition} Let $h$ be a partial function from $[-1,1]$ to $[0,1]$ such that $Dom(h)$ is dense in $[-1,1]$ and $h$ is non-decreasing. \\Let $u$ in $[-1,1]$ be given. $u$ is a \emph{point of continuity for $h$} if and only if \\$\forall n \exists t \in Dom(h)\exists w\in Dom(h)[t<u<w \;\wedge\; h(w)-h(t) <\frac{1}{2^n}]$. \end{definition}
  
  \begin{remark}\label{R:extension} 
  If $u$ is a point of continuity for $h$, there exists $y$ in $[0,1]$ such that $\forall n\exists m \forall t\in Dom(h)[u-\frac{1}{2^m}<t<u+\frac{1}{2^m} \rightarrow  |h(t)-y|<\frac{1}{2^n}]$. $y$ is unique up to the relation of real coincidence and will be called \emph{the value of $h$ at $u$}. One may extend the  partial function $h$ to the partial function $g\subseteq [-1,1]\times[0,1]$ consisting of all pairs $(u,y)$ such that either $(u,y)\in h$ or $u$ is a point of continuity for $h$ and $y$ is the value of $h$ at $u$. \end{remark}
  \begin{theorem}\label{T:manypointsofcontinuity} Let $h$ be a partial function from $[-1,1]$ to $[0,1]$ such that $Dom(h)$ is dense in $[-1,1]$ and $h$ is non-decreasing. All but countably many elements of $[-1,1]$ are points of continuity for $h$. \end{theorem}
 \begin{proof} Find an infinite sequence $x_0, x_1, \ldots$ of elements of $Dom(h)$ such that \\$\forall p \in \mathbb{Q} \forall q\in \mathbb{Q}[0\le p <q\le  1\rightarrow \exists n[p\le x_n\le q]]$. 
 
 Define, for all $m,n$ such that $x_m<x_n$ and $h(x_m)<h(x_n)$, for all $z$ in $[0,1]$, \begin{quote} $z$ \emph{resolves} $(x_m, x_n)$ \end{quote}if and only if $x_m\le z \le x_n$ and, for all $u$ in $[x_m, x_n]$, if $u\;\#\;z$, then \\$\exists t \in Dom(h)\exists w \in Dom(h)[t<u<w \;\wedge\; h(w)-h(t) < \frac{2}{3}\bigl(h(x_n)-h(x_m)\bigr)]$.

 Using Lemma \ref{L:monotonecontinuouspreparation},  define an infinite sequence $z_0, z_1, \ldots$ of elements of $[0,1]$ such that $z_0=0$ and $z_1=1$ and, for each $n$, \\{\it if} $\bigl(x_{n(0)})''\bigl(n(2)\bigr)<\bigl(x_{n(1)})'\bigl(n(2)\bigr)$ and  $(h\bigl(x_{n(0)}))''\bigl((n(3)\bigr) <(h\bigl(x_{n(1)}))'\bigl(n(3)\bigr)$, \\{\it then} $z_{n+2}$ resolves $(x_{n(0)}, x_{n(1)})$.  
 
 Let $u$ be given such that, for each $n$, $u\;\#\; z_n$. We prove that $u$ is a point of continuity for $h$, by showing, inductively:
  \begin{quote} For each $n$, there exist $t,w$ in $Dom(h)$ such that $t<u<w$ and $h(w)-h(t)<(\frac{2}{3})^n\cdot 3$ \end{quote}
 
 The case $n=0$ obviously holds: find $t,w$ in $Dom(h)$ such that $t<u<w$ and note: $h(w) -h(t)\le|h(w)|+|h(t)| \le 2<3$. 
 
 Now let $n, t,w$ be given such that $t,w$ in $Dom(h)$ and $t<u<w$ and \\$h(w) - h(t)<(\frac{2}{3})^n\cdot 3$. Find $m,p$ such that $t<x_m<u<x_p<w$ and note: $h(x_p) - h(x_m)<(\frac{2}{3})^n\cdot 3$. Now distinguish two cases. 
 
 \textit{Case (1). \it} $h(x_p) - h(x_m)<(\frac{2}{3})^{n+1}\cdot 3$ and  we are done, or
 
 {\it Case (2).} $0<h(x_p)-h(x_m)$. \\In the latter case,  find $q,r$ such that $x_m''(q)<x_p'(q)$ and $\bigl(h(x_m)\bigr)''(r) < \bigl(h(x_p)\bigr)'(r)$. \\Then find $s$ such that  $s(0) = m, \;s(1)= p, \;s(2) =q$ and $s(3)=r$. \\As $y_{s+2}$ resolves $(x_m, x_p)$ and $u\;\#\; y_{s+2}$,  find $t,w$ in $Dom(h)$ such that $t<u<w$ and $h(w)-h(t)<\frac{2}{3}\bigl(h(x_p)-h(x_m)\bigr)< (\frac{2}{3})^{n+1}\cdot 3$. 
 
 We thus see that $u$ is indeed a point of continuity for $h$.
 \end{proof}
 
 \begin{remark}\label{R:co-enumerable} Using Remark  \ref{R:extension} and Theorem \ref{T:manypointsofcontinuity}, observe that a non-decreasing partial function $h$ from $[-1,1]$ to $[0, 1]$ such that $Dom(h)$ is dense in $[-1,1]$ may be extended to a non-decreasing partial function $g$ from $[-1,1]$ to $[0,1]$ such that $Dom(g)$ is \emph{co-enumerable}, i.e. there exists an infinite sequence $z_0, z_1, \ldots$ of elements of $[-1,1]$ such that, for every $u$ in $[0,1]$, if $\forall n[ u \;\#\;z_n]$, then $u\in Dom(g)$. \end{remark}
 
 \begin{theorem} Let $f$ be a measurable function $f$ from $[0,1]$ to $[-1,1]$.  The domain of the partial function \begin{quote}$D(f):=\{(y, m)\in [-1,1]\times [0,1]\mid m= \mu(\{x\in [0,1]\mid f(x)<y\})\}$ \end{quote}is a co-enumerable subset of $[-1,1]$. \end{theorem}
 
 \begin{proof} Note that $D(f)$ is non-decreasing and that, by Corollary \ref{Cor:denseoriginal}, $Dom\bigl(D(f)\bigr)$ is a dense subset of $[-1,1]$. Now use Remark \ref{R:co-enumerable}.\end{proof}
 
 \section{The Bar Theorem}\label{S:barth}
 Brouwer, when first proving the Fan Theorem, obtained the Fan Theorem as a Corollary of a more general result, see \cite{brouwer27} and \cite{brouwer54}.  \begin{theorem}[Bar Theorem]\label{T:bartheorem}$\*$\\ If $B\subseteq \omega$ be a bar in $\omega^\omega$, there exists 
 a stump $S$ such that $ S\cap B$ is a bar in $\omega^\omega$.  \end{theorem}
 \begin{proof} 
 Let us define, for each $B\subseteq \omega$, for each $s$, \\$B$ {\it bars $s$} if and only if $Bar_{\omega^\omega\cap s}(B)$, i.e. $\forall \alpha[s\sqsubset \alpha\rightarrow \exists n[\overline \alpha n \in B]]$. 
 
  Observe the following:
  
  (i) For all $s$, if $s\in B$, then $B$ bars $s$.
  
  (ii) If, for all $n$, $B$ bars $s\ast\langle n \rangle$, then $B$ bars $s$ in $\omega^\omega$. 
  
  (iii) For all $s$, if $B$ bars $s$, then, for all $n$,  $B$ bars $s\ast\langle n \rangle$.
 
 \smallskip Now let $B\subseteq \omega$ be given such that $Bar_{\omega^\omega}(B)$, i.e. $B$ bars $\langle\;\rangle$. 
 
 Under what circumstances should we say that we are  entitled to affirm this statement?
 
 Note that one may {\it prove} a statement of the form `$B$ bars $s$' by starting from observations of the form (i) and using observations of the form (ii) and (iii) as reasoning steps.

 Let us now agree to consider the statement `$B$ bars $s$' as  established or {\it true} if and only if we are able to provide such a {\it canonical} proof.  
 
 Note that such a canonical proof is no longer a finite `tree', like in the case of the Fan Theorem, but an infinitary one. The structure of a canonical proof is comparable to the structure of a stump. 
 
 The above agreement marks an important point in the development of our \\intuitionistic mathematics. We are introducing an {\it axiomatic assumption}.
 
 \smallskip
 After shaking hands, we
  argue as follows. 
 
 Take a canonical proof of `$B$ bars $\langle \;\rangle$'.
 
 In this canonical proof, replace every statement: `$B$ bars $s$' by the statement: `there exists a stump $S$ such that 
$(s\ast S )\cap B$ bars $s$'.

We now verify that the new `proof' is  a valid proof.

\smallskip (i) If $s\in B$ we can take $S=\{\langle\;\rangle\}$. 

(ii) Let $s$ be given such that, for each $n$, there exists a stump $S$ such that $(s\ast\langle n \rangle\ast S)\cap B$ bars $s\ast\langle n \rangle$. Using the Second Axiom of Countable Choice\footnote{Stumps, as decidable subsets of $\omega$, may be identified with their characteristic functions.}, Axiom \ref{ax:secondchoice}, we   build, step by step, a stump $S$ such that, for each $n$, $(s\ast\langle n \rangle\ast (S\upharpoonright n))\cap B$ bars $s\ast\langle n \rangle$. Note that $(s\ast S)\cap B$ bars $s$.

(iii) Let $s, n$ be given. Let $S$ be a stump such that $(s\ast\langle n\rangle \ast S))\cap B$ bars $s$. Now distinguish two cases.

\textit{Case (1)}. $S\upharpoonright \langle n \rangle =\emptyset$. Conclude: $\exists t \sqsubseteq s[t \in B]$. Define $T:=\{\langle\;\rangle\}$ and note: $(s\ast\langle n \rangle\ast T)\cap B$ bars $s\ast\langle n \rangle$.

\textit{Case (2)}. $\langle \; \rangle \in S\upharpoonright \langle n \rangle$. Then $(s\ast\langle n \rangle \ast S\upharpoonright\langle n \rangle)\cap B$ bars $s\ast\langle n \rangle$. 

\smallskip We thus see that our new `proof' is a valid proof indeed.

We may affirm its conclusion: \\`there exists a stump $S$ such that $S\cap B$ bars $\langle\;\rangle$, i.e. $Bar_{\omega^\omega}(S\cap B)$.'
 \end{proof}
 
 \subsection{An application} \begin{definition}\label{D:inductivewell-orderings} For all $A,B\subseteq \mathbb{Q}$, we define: \\$A<B$ if and only if $\forall q \in A\forall r \in B[q<r]$.
 
 \smallskip We define a collection $\mathcal{WO}$ of subsets of $\mathbb{Q}$ by the following inductive definition. The elements of $\mathcal{WO}$ are called the \emph{inductively  well-ordered} subsets of $\mathbb{Q}$.
 
 \begin{enumerate}[\upshape(i)]\item $\emptyset \in \mathcal{WO}$, and, for each $q$ in $\mathbb{Q}$, $\{q\}\in \mathcal{WO}$. 
 \\These are the \emph{basic elements} of $\mathcal{WO}$. 
 \item For every sequence $A_0, A_1, \ldots$ of elements of $\mathcal{WO}$ such that, for each $n$, \\$A_n<A_{n+1}$, also $\bigcup_n A_n \in \mathcal{WO}$. \\This is the \emph{construction step} of the set $\mathcal{WO}$.
 
 \item Every element of $\mathcal{WO}$ is obtained from basic elements of $\mathcal{WO}$ by applying the construction step repeatedly. \end{enumerate}
 
 \smallskip
We let $q_0, q_1, \ldots$ be some canonical  enumeration without repetitions of $\mathbb{Q}$.

$A\subseteq \mathbb{Q}$ is a \emph{decidable} subset of $\mathbb{Q}$ if and only if $\exists \alpha\forall n[q_n \in A \leftrightarrow \alpha(n)\neq 0]$.

\smallskip Let $A\subseteq \mathbb{Q}$ and $\alpha$ be given. $\alpha$ is called an \emph{enumeration of} $A$ if and only if \\$\forall n [q_n\in A \leftrightarrow  \exists m[ \alpha(m)=n+1]]$.  

$A\subseteq \mathbb{Q}$ is \emph{ enumerable} if and only if there exists an enumeration of $A$. 

\medskip $A\subseteq \mathbb{Q}$ is \emph{well-founded} if and only if $\forall \gamma[\forall n[q_{\gamma(n)} \in A]\rightarrow \exists n[q_{\gamma(n)}\le q_{\gamma(n+1)}]]$. 
 \end{definition}

 \begin{lemma}\label{L:decidablewf}  Let $\alpha$ be an enumeration of $A\subseteq \mathbb{Q}$. \\$A$ is well-founded if and only if $\forall \gamma\exists n[ \forall i\le n+1[\alpha\circ\gamma(i)>0]\rightarrow q_{\alpha\circ\gamma(n)}\le q_{\alpha\circ\gamma(n+1)}]$. \end{lemma} 
 \begin{proof} Let $\alpha$ be an enumeration of $A\subseteq \mathbb{Q}$. 
 
 \smallskip First assume $A$ is well-founded. Let $\gamma$ be given. Distinguish two cases. \\ \textit{Case (a)}. $\alpha\circ\gamma(0)=0$. Then $\forall i\le 2[\alpha\circ\gamma(i)>0]\rightarrow q_{\alpha\circ \gamma(0)}\le q_{\alpha\circ\gamma(1)}$.
 \\ \textit{Cases (b)}. $\alpha\circ\gamma(0)>0$. Now define $\beta$ such that, for each $n$, if $\alpha(n)>0$, then $\beta(n) =\alpha(n)$, and, if $\alpha(n)=0$, then $\beta(n)=\alpha\circ\gamma(0)$. Note: $\forall n[\beta(n)>0 \;\wedge\; q_{\beta(n)-1} \in A]$. Find $n$ such that $q_{\beta(n)-1} \le q_{\beta(n+1)-1}$. Note that, if $\forall i\le n+1[\alpha(i)>0]$, then $\alpha(n)=\beta(n)$ and $\alpha(n+1)=\beta(n+1)$ and   $q_{\alpha(n)-1} \le q_{\alpha(n+1)-1}$.  \\Conclude: $\forall \gamma\exists n[ \forall i\le n+1[\alpha\circ\gamma(i)>0]\rightarrow q_{\alpha\circ\gamma(n)}\le q_{\alpha\circ\gamma(n+1)}]$.

 \smallskip  Now assume $\forall \gamma\exists n[ \forall i\le n+1[\alpha\circ\gamma(i)>0]\rightarrow q_{\alpha\circ\gamma(n)}\le q_{\alpha\circ\gamma(n+1)}]$.  
 
 Let $\gamma$ be given such that $\forall n[q_{\gamma(n)}\in A]$. Find $\delta$ such that, for each $n$, $\delta(n) >0$ and $q_{\alpha\circ\delta(n)-1}=q_{\gamma(n)}$. Find $n$ such that $q_{\alpha\circ\delta(n)-1}\le q_{\alpha\circ\delta(n+1)-1}$ and conclude: $q_{\gamma(n)}\le q_{\gamma(n+1)}$.
 \\ Conclude: $\forall \gamma [\forall n [q_{\gamma(n)}\in A]\rightarrow \exists n[q_{\gamma(n)}\le q_{\gamma(n+1)}]$, i.e. $A$ is well-founded. \end{proof}
 \begin{lemma}\label{L:wo}$\;$\begin{enumerate}[\upshape(i)]\item Every $A$ in $\mathcal{WO}$ is enumerable and well-founded.\item For all $A$ in $\mathcal{WO}$,  for all $a,b$ in $ A$, the set $\{c\in A\mid a< c \le b\}$ belongs to $\mathcal{WO}$. \item For every enumerable subset $A$ of $\mathbb{Q}$, if, for all $a$ in $A$, $A_{\le a}:=\{b\in A\mid b\le a\}$ belongs to $\mathcal{WO}$, then $A \in \mathcal{WO}$. \end{enumerate} \end{lemma}
 
 \begin{proof}  The proof of (i) and (ii) is by straightforward induction on $\mathcal{WO}$ and left to the reader.
 
 \smallskip
 
 (iii)  Let $\alpha$ be an enumeration of $A$.
 Define $\beta$ such that, for each $n$, \\\textit{if} $\alpha(n)>0$ and $\forall i<n[\beta(i)>0 \rightarrow q_{\beta(i)-1}<q_{\alpha(n)-1}]$, then $\beta(n)=\alpha(n)$, and, \\\textit{if not}, then $\beta(n) = 0$.
  \\Define an infinite sequence $A_0, A_1, A_2, \ldots$ of subsets of $A$, such that,  
 for each $n$, if $\beta(n)>0$,  then $A_n:=\{r \in A \mid r\le q_{\beta(n)-1}\; \wedge\;\forall i<n[\beta(i)>0\rightarrow r> q_{\beta(i)-1}]\}$, and, if $\beta(n)=0$, then $A_{n}=\emptyset$.
 
 Using (iii), note: for all $n$, $A_n \in \mathcal{WO}$.

 Then note: for all $n$, $A_n <A_{n+1}$ and $A:=\bigcup_n A_n$ and conclude: $A\in \mathcal{WO}$. 
 \end{proof}
 
 The next result may be compared to results in \cite[\S 5]{howard}.
 \begin{theorem}\label{T:wo}\footnote{In \cite[Section 4, Theorem 6]{veldman2008} one finds an intuitionistic version of a more difficult but related result: 
  F. Hausdorff's Theorem  on scattered subsets of $\mathbb{Q}$. The proof is wrong however and the result is doubtful.}   Let $\alpha$ be an enumeration of $A\subseteq \mathbb{Q}$.  \\If $\forall \gamma\exists n[\forall i\le n+1[\alpha\circ\gamma(i)>0]\rightarrow q_{\alpha\circ\gamma(n)}\le q_{\alpha\circ\gamma(n+1)}]$, then $A\in \mathcal{WO}$. \end{theorem}
 
 \begin{proof} For every $\alpha$,  define $B_\alpha:=\bigcup_n \{s\in \omega^{n+1} \mid \exists i \le n [\alpha\circ s (i)=0]\} \cup \\\bigcup_n\{s\in \omega^{n+2}\mid \alpha\circ s(n)>0\;\wedge\;\alpha\circ s(n+1) >0   \;\wedge\;  q_{\alpha\circ s(n)-1}\le q_{\alpha\circ s(n+1)-1}\}$. 
  
  \smallskip Note: if $\forall \gamma\exists n[\forall i\le n+1[\alpha\circ\gamma(i)>0]\rightarrow q_{\alpha\circ\gamma(n)}\le q_{\alpha\circ\gamma(n+1)}]$, then $B_\alpha$ is a bar in $\omega^\omega$, and, by Theorem \ref{T:bartheorem}, there exists a stump $S$ such that $S\cap B_\alpha$ is a bar in $\omega^\omega$.
  
 \smallskip We will say say that a stump $S$ has the property $(\ast)$ if and only if \begin{quote} 
 
 $(\ast)$ for every $\alpha$, if $S\cap B_\alpha$ is a bar in $\omega^\omega$,  then $\{q_{\alpha(n)}\mid n \in \omega\}\in \mathcal{WO}$. \end{quote}
 
 We now prove that every stump has the property $(\ast)$,  by induction on the set $\mathbf{Stp}$ of stumps.

 \smallskip Let a  stump $S$ be given such that every immediate substump of $S$ has the property $(\ast)$. 
 
 Let $\alpha$ be given such that $S\cap B_\alpha$ is a bar in $\omega^\omega$. 
 
 We want to prove: $A\in \mathcal{WO}$. 
 
 According to Lemma \ref{L:wo}, it suffices to prove: \\for all  $q$ in $\mathbb{Q}$, if $q \in A$, then $A_{\le q}=\{r\in A\mid r\le q\}\in \mathcal{WO}$. 

 \smallskip Let $q$ in $A$ be given. Find $n$ such that $\alpha(n)>0$ and $q_{\alpha(n)-1}=q$. 
 
   Define $\beta$ such that, for all $m$, if $\alpha(m)>0$ and $q_{\alpha(m)-1}<q_n$, then $\beta(m)=\alpha(m)$, and, if not, then $\beta(m)= 0$. Note that $\beta$ enumerates $\{q\in A\mid q<q_n\}$. 
   
   Let $\delta$ be given. 
   Find $p$ such that  $\langle n \rangle \ast \overline \delta p \in S\cap B_\alpha$. Note: $\alpha(n)>0$ and:  $p>0$. Now distinguish two cases.  
   \\{\it Case (a)}. $\alpha\circ \delta (p-1) =0$. Then also $\beta\circ\delta(p-1)=0$ and $\overline \delta p \in B_\beta$.
   \\{\it Case (b)}. $\alpha\circ\delta(p-1)>0$. We distinguish two subcases. 
   \\{\it Case (bi)}. $p=1$ and  $q_{\alpha(n)-1}\le q_{\alpha\circ \delta(0)-1}$. Then $\beta\circ\delta(0) =0$ and  $\overline \delta 1= \overline \delta p \in B_\beta$.
   \\{\it Case (bii)}. $p>1$ and  $\alpha\circ \delta(p-2)>0$ and  $\alpha \circ \delta (p-1)>0$ and \\$q_{\alpha\circ\delta(p-2)-1}\le q_{\alpha\circ\delta(p-1)-1}$. Now {\it either} $\beta \circ \delta(p-2)=\alpha\circ\delta(p-2)$ and $\beta \circ \delta(p-1)=\alpha\circ\delta(p-1)$ and $\overline \delta p \in B_\beta$, {\it or} $\exists i<p[\beta \circ \delta(i) = 0]$ and again: $\overline \delta p \in B_\beta$.

   Conclude: $\forall \delta \exists p[\langle n \rangle \ast \overline \delta p \in S \;\wedge\; \overline \delta p \in B_\beta]$.

 We thus see: $(S\upharpoonright\langle n \rangle)\cap B_\beta$ is a bar in $\omega^\omega$.

 \smallskip As $S\upharpoonright\langle n \rangle$ has the property $(\ast)$, conclude: $A_{<q_n} \in \mathcal{WO}$.
 
 But then also $A_{\le q_n}= A_{< q_n}\cup\{q_n\} \in \mathcal{WO}$.

 \smallskip We thus see: for all $n$, if $q_n\in A$, then $A_{\le q_n} \in \mathcal{WO}$, and conclude: \\$A=\bigcup_n A_n \in \mathcal{WO}$.
 
 \smallskip   Using Theorem \ref{T:bartheorem}, we conclude: for every decidable subset $A$ of $\mathbb{Q}$,\\ if  $\forall \gamma\exists n[q_{\gamma(n)}\le q_{\gamma(n+1)}\;\vee\;\gamma(n)\notin A]$, then $A\in \mathcal{WO}$.
 \end{proof}
 
 \subsection{Bar Induction} \begin{theorem}[Principle of Bar Induction]\label{T:barinduction} Let $B,C\subseteq \omega$ be given such that $Bar_{\omega^\omega}(B)$ and $B\subseteq C$ and $\forall s[s\in C\leftrightarrow \forall n[s\ast\langle n \rangle \in C]$. Then $\langle\;\rangle \in C$. \end{theorem}
 \begin{proof} Assume $Bar_{\omega^\omega}(B)$, i.e. $B$ bars $\langle\;\rangle$.
 
 Find a canonical proof of: `$B$ bars $\langle\;\rangle$'.
 
 Assume also: $B\subseteq C$ and $\forall s[s\in C\leftrightarrow \forall n[s\ast\langle n \rangle \in C]$. 
 
 In the canonical proof, replace every statement: `$B$ bars $s$' by the statement `$s \in C$'. 
 
 Note that the result is another valid proof, with conclusion: $`\langle\;\rangle \in C$'. \end{proof}
 
 We now may give a second proof of Theorem \ref{T:wo}. 
 
 \begin{proof} Let $\alpha$ be given such that  \\$\forall \gamma\exists n[\forall i\le n+1[\alpha\circ\gamma(i)>0]\rightarrow q_{\alpha\circ\gamma(n)}\le q_{\alpha\circ\gamma(n+1)}]$. 
 
 \smallskip
Define $B=B_\alpha:=\bigcup_n \{s\in \omega^{n+1} \mid \exists i \le n [\alpha\circ s (i)=0]\} \cup \\\bigcup_n\{s\in \omega^{n+2}\mid \alpha\circ s(n)>0\;\wedge\;\alpha\circ s(n+1) >0   \;\wedge\;  q_{\alpha\circ s(n)-1}\le q_{\alpha\circ s(n+1)-1}\}$. 

\smallskip
Let $C$ be the set of all $s$ such that {\it either} $\exists t\sqsubseteq s[t \in B]$ {\it or}  $s=\langle \;\rangle$ and $A\in \mathcal{WO}$, {\it or} $n:=length(s)>0$ and   $\alpha\circ s(n-1)>0$ and  $A_{<q_{\alpha\circ s(n-1)-1}}\in \mathcal{WO}$.  

\smallskip Note: $Bar_{\omega^\omega}(B)$ and $B\subseteq C$.

Also note: for all $s$, if $s\in C$, then $\forall n[s\ast\langle n \rangle \in C]$.

Finally, let $s$ be given such that $\forall n[s\ast\langle n \rangle \in C]$. We distinguish two cases.

{\it Case (1)}. $s=\langle \;\rangle$. For all $n$, we may consider $\langle n \rangle$, and conclude: \\ if $q_n \in A$, then $A_{<q_n} \in \mathcal{WO}$.

By Lemma \ref{L:wo}, $A\in \mathcal{WO}$ and $\langle\;\rangle \in C$.

{\it Case (2)}. $p:=length(s) >0$. \\For all $n$, such that $q_n<q_{s(p-1)}$, we may consider $s\ast\langle n \rangle$ and conclude:  $A_{<q_n} \in \mathcal{WO}$.

By Lemma \ref{L:wo}, $A_{<q_{s(p-1)}}\in \mathcal{WO}$ and $s \in C$.

\smallskip Using Theorem \ref{T:barinduction}, we conclude:  $\langle\;\rangle \in C$ and $A\in \mathcal{WO}$.

\end{proof} 

\section{The Almost-Fan Theorem}\label{S:aft}

This Section has seven Subsections. In Subsection \ref{SS:almostfinite} we introduce the notion of an almost-finite subset of $\omega$. In Subsection \ref{SS:afspreads} we introduce  almost-finitary spreads and we show that, like the Fan Theorem, the Almost-Fan Theorem follows from the Bar Theorem. We show that the Almost-Fan Theorem implies the Fan theorem. In Subsection \ref{SS:openinduction} we formulate the Principle of Open Induction on $[0,1]$ and show that it follows from the Almost-Fan Theorem. In Subsection \ref{SS:dedekind} we see that the Principle of Open Induction on $[0,1]$ implies a version of Dedekind's Theorem. In Subsection \ref{SS:ramsey} we prove that the Almost-Fan Theorem also implies an intuitionistic version of the Infinite Ramsey Theorem.  In Subsection \ref{SS:bw} we use this Ramsey Theorem together with Dedekind's Theorem in order to prove an intuitionistic version of the Bolzano-Weierstrass Theorem. In Subsection \ref{SS:phr} we show that the Infinite Ramsy Theorem implies the Paris-Harrington-Ramsey Theorem. 

\smallskip The  results of this Section may be seen as intuitionistic comments on results in \cite[Chapter III]{Simpson}.
\subsection{Almost-finite subsets of $\omega$}\label{SS:almostfinite} One may formulate many notions of finiteness, even for decidable subsets of $\omega$, see \cite{veldman1995}, \cite{veldman1999} and \cite[Section 3]{veldman2005}.  We need three of them. \begin{definition}\label{D:finite} $B\subseteq \omega$ is a  \emph{decidable subset of $\omega$} if and only if \\$\exists \alpha\forall n[n\in B\leftrightarrow \alpha(n)\neq 0]$.

A decidable subset $B$ of $\omega$ is \begin{enumerate} \item\emph{finite} if and only if $\exists n\forall  m>n[m\notin B]$,  \item \emph{bounded-in-number} if and only if $\exists k\forall s \in [\omega]^{k+1}\exists i\le k[s(i)\notin B]$, and \item \emph{almost-finite} if and only if $\forall \zeta \in [\omega]^\omega\exists i[\zeta(i)\notin B]$. \end{enumerate} \end{definition}

\begin{theorem} \begin{enumerate}[\upshape (i)] \item For every decidable $B\subseteq\omega$, if $B$ is finite, then $B$ is bounded-in-number, but not conversely.
 \item For every decidable $B\subseteq \omega$, if $B$ is bounded-in-number, then $B$ is almost-finite, but not conversely. \end{enumerate} \end{theorem} \begin{proof} (i) Let $B\subseteq \omega$ be decidable and finite. Find  $n$ be such that $\forall m>n[m \notin B]$. Clearly, for all $s$ in $[\omega]^{n+2}$, $s(n+1)>n$ and $s(n+1)\notin B$. We thus see that $B$ is bounded-in-number. 
 
 As to the converse, we give a counterexample in Brouwer's style.\\ Let $d:\omega\rightarrow \{0,1,\ldots,9\}$ be the decimal expansion of $\pi$. \\Define $B:=\{k_{99}\}:=\{n \mid n=\mu k \forall i<99[d(k+i)=9]\}$. \\$B$ is a decidable subset of $\omega$ and $B$ has at most one member, but we are unable to find  $n$ such that $\forall m>n[m\notin B]$.
 
 \smallskip (ii) Let $B\subseteq \omega$ be decidable and bounded-in-number. \\Find $k$ such that $\forall s \in [\omega]^{k+1}\exists i\le k[s(i) \notin B]$. \\Conclude: $\forall \zeta \in [\omega]^\omega \exists i\le k[\zeta(i)\notin B]$ and: $B$ is almost-finite.
 
 As to the converse, we give  a counterexample in Brouwer's style. \\Define $B:=\{n\mid k_{99} \le n \le 2\cdot k_{99}\}:=\\ \{ n\mid \mu k \forall i<99[d(k+i)=9\}]\le n \le 2\cdot \mu k \forall i<99[d(k+i)=9]\}$. 
 
 We now prove that $B$ is almost-finite. \\ Let $\zeta$ in $[\omega]^\omega$ be given. \\\textit{Either} $\zeta(0)< k_{99}$ and $\zeta(0)\notin B$, \textit{or} $\zeta(0)\ge k_{99}$ and $\zeta(2\cdot k_{99} +1)\notin B$. 
 
 We are unable, however,  to find $n$ such that $B$ has at most $n$ members.   \end{proof}

 We now prove that the union of two almost-finite subsets of $\omega$ is almost-finite, and that a decidable subset of $\omega$ that is the union of almost-finitely many almost-finite subsets of $\omega$,   is itself almost-finite.

\begin{lemma}\label{L:almost-finite} \begin{enumerate}[\upshape (i)] \item For all decidable subsets $B,C$ of $\omega$, \\if $B,C$ are almost-finite, then $B\cup C$ is almost-finite. \item For every decidable subset $B$ of $\omega$, if there  exists an infinite sequence \\$B_0, B_1, B_2, \ldots$ of decidable and almost-finite subsets of $\omega$ such that\\$B=\bigcup_n B_n$  and $\forall \zeta\in [\omega]^\omega\exists i[B_{\zeta(i)}=\emptyset]$, then $B$ itself is almost-finite. \end{enumerate}\end{lemma} 

\begin{proof} (i) Let $\zeta$ in $[\omega]^\omega$ be given. Find $\eta$ in $[\omega]^\omega$ such that, for each $n$, \\$\eta(n)=\mu p[\forall i<n[p>\eta(i)]\;\wedge\;\zeta(p) \notin B]$. Note: $\forall p[\zeta\circ\eta(p)\notin B]$. \\Find $p$ such that $\zeta\circ\eta(p)\notin C$. Define $m:=\eta(p)$ and note: $\zeta(m) \notin B\cup C$. \\We thus see: $\forall \zeta \in[\omega]^\omega\exists m[\zeta(m)\notin B\cup C]$.

\smallskip (ii) Let $\zeta$ in $[\omega]^\omega$ be given. We want to prove: $QED:=\exists p[\zeta(p)\notin \bigcup_n B_n]$.\footnote{QED: `quod \emph{est} demonstrandum', `what we (still) have to prove' rather than `quod \emph{erat} demonstrandum', what we \emph{did} have to prove'.}

Find $\alpha$ such that, for all $p$, if $\zeta(p)\in B=\bigcup_n B_n$, then $\alpha(p)=\mu n[\zeta(p)\in B_n]$. 

 We claim: $\forall p \exists q>p[\alpha(q)>p \;\vee\; QED]$.
 We prove this claim as follows. 
 
 Let $p$ be given. 
 
 Using (i), observe that $\bigcup_{i\le p}B_i$ is almost-finite. 
 
 Find $q>p$ such that $\zeta(q) \notin \bigcup_{i \le p}B_i$.
 
 If $\zeta(q) \in \bigcup_n B_n$, then $\alpha(q)>p$.
 
 If $\zeta(q) \notin \bigcup_n B_n$, then $QED$.  
 
 \smallskip Now find $\eta$ in $[\omega]^\omega$ such that $\eta(0)=0$ and $\forall p[\alpha\circ\eta(p+1)> \alpha\circ\eta(p) \;\vee\; QED]$.
 
 Find $p$ such that $B_{\alpha\circ\eta(p)}=\emptyset$ and conclude: either $QED$, or $\zeta\circ \eta(p)\notin \bigcup_nB_n$ and again: $QED$.  \end{proof}
 
 \subsection{Almost-finitary spreads}\label{SS:afspreads}
\begin{definition}\label{D:almostfan}
 A spread-law $\beta$ will be called \emph{almost-finitary} if it satisfies the following condition: \begin{quote} $\forall s[\beta(s)=0\rightarrow \forall \zeta\in[\omega]^\omega\exists n[\beta(s\ast\langle \zeta(n) \rangle)\neq 0]]$. \end{quote}
 
If the spread-law $\beta$ is almost-finitary, the corresponding spread $\mathcal{F}_\beta$ will be called an \emph{ almost-finitary spread} or an \emph{ almost-fan}. \end{definition}

If $\beta$ is an almost-finitary spread-law and $s$ is admitted by $\beta$, there are only almost-finitely many immediate successors $s\ast\langle n \rangle$ of $s$ that are admitted by $\beta$. 
\begin{definition} $P\subseteq \omega$ is \emph{almost-full} if and only if $\forall \zeta\in[\omega]^\omega \exists n[\zeta(n)\in P]$. \end{definition}
We need the following Lemma.

\begin{lemma}\label{L:almostfull} $\;$ \begin{enumerate}[\upshape (i)] \item For all $P,Q\subseteq \omega$, if $P$, $Q$ are almost-full, then $P\cap Q$ is almost-full. \item For all $n>0$, for all $P_0, P_1, \ldots, P_{n-1}\subseteq \omega$, \\if, for all $i<n$, $P_i$ is almost-full,  then $\bigcap_{i<n}P_i$ is almost-full. \end{enumerate} \end{lemma}
\begin{proof} (i) Assume $P,Q\subseteq \omega$ are almost-full. \\Let $\zeta$ in $[\omega]^\omega$ be given. Using the First Axiom of Countable Choice, find $\eta$ in $[\omega]^\omega$ such that, for each $n$, $\eta(n)\in P$. \\Find $p$ such that $\zeta\circ\eta(p)\in Q$. Define $m:=\eta(p)$ and note: $\zeta(m) \in P\cap Q$. \\We thus see: $\forall \zeta \in[\omega]^\omega\exists m[\zeta(m)\notin P\cap Q]$, i.e. $P\cap Q$ is almost-full. 

\smallskip (ii) Use (i) and induction. \end{proof}

The following theorem appears in \cite{veldman2001} and \cite[Section 7.2]{veldman2011c}.
\begin{theorem}[Almost-Fan Theorem]
\label{T:aft} Let $\beta$ be an almost-finitary spread-law and let $B\subseteq \omega$ be  a bar in $\mathcal{F}_\beta$. \begin{enumerate}[\upshape (i)] \item  There exists $B'\subseteq B$ such that $B'$ is a bar in $\mathcal{F}_\beta$ and $B'$ is a decidable and almost-finite subset of $\omega$. 
\item $\forall \zeta \in [\omega]^\omega\exists n[\beta(\zeta(n)) \neq 0 \;\vee \;\exists t\sqsubseteq \zeta(n)[t \in B]]$, and, therefore: \\$\forall \zeta \in [\omega]^\omega[\forall n[\beta\bigl(\zeta(n)\bigr)=0]\rightarrow \exists n \exists t[t\sqsubseteq \zeta(n)\;\wedge\; t \in B]]$.  \end{enumerate} \end{theorem}

\begin{proof} Let $\beta$ be an almost-finitary-spread-law
 and let $B\subseteq\omega$ be a bar in $\mathcal{F}_\beta$.  
 
 Define $B^+:= B \cup \{s\mid \beta(s)\neq 0\}$.
   
   We claim that $B^+$ is a bar in $\omega^\omega$. \\In order to see this, we let $\rho:\omega^\omega \rightarrow \omega^\omega$ be a retraction\footnote{See the proof of Theorem \ref{T:bcpext}.} of $\omega^\omega$ onto $\mathcal{F}_\beta$, i.e.  \\$\forall \alpha[\rho|\alpha \in \mathcal{F}_\beta]$ and even $\forall \alpha \forall n[\overline{(\rho|\alpha)}n \neq \overline \alpha n\rightarrow \beta(\overline\alpha n)\neq 0]$, and $\forall \alpha \in \mathcal{F}_\beta[\rho|\alpha =\alpha]$.
   
   Given $\alpha$, find $n$ such that $\overline{(\rho|\alpha)}n \in B$ and distinguish two cases. \textit{Either} \\$\overline{(\rho|\alpha)}n =\overline \alpha n$ and $\overline \alpha n \in B^+$, \textit{or} $\overline{(\rho|\alpha)}n \neq\overline \alpha n$ and $\beta(\overline \alpha n)\neq 0$ and, again, $\overline \alpha n \in B^+$.

   We thus see that, indeed, $B^+$ is a bar in $\omega^\omega$.

\smallskip (i) Let $C$ be the set of all $s$ such that either $\beta(s)\neq 0$ or $\beta(s)=0$ and there exists $B'\subseteq B$ such that $Bar_{\mathcal{F}_\beta\cap s}(B')$ and $B'$ is a decidable and almost-finite subset of $\omega$.

\smallskip Note the following:

(i)1. $B^+\subseteq C$.

(i)2. For all $s$, if $s\in C$, then for all $n$, $s\ast\langle n \rangle \in C$.

(i)3. Let $s$ be given such that $\forall n[s\ast\langle n  \rangle \in C]$. \\Find, using axiom \ref{ax:secondchoice},  an infinite sequence $B_0, B_1, \ldots$ of subsets of $B$ such that, for each $n$, if $\beta(s\ast\langle n \rangle)\neq 0$, then $B_n=\emptyset$ and, if $\beta(s\ast\langle n \rangle)=0$, then $B_n$ is a decidable and almost-finite subset of $\omega$ and $B_n$ is a bar in $\mathcal{F}_\beta\cap s\ast\langle n\rangle$.  Define $E:=\bigcup_n B_n$ and note: $E$ is a decidable subset of $\omega$ and $\forall \zeta \in [\omega]^\omega\exists n[B_{\zeta(n)}=\emptyset]$ and $E$ is a bar in $\mathcal{F}_\beta\cap s$ and $E\subseteq B$. \\  According to 
 Lemma \ref{L:almost-finite}, $E$ is almost-finite.  One may conclude: $s \in C$.   

\smallskip Using Theorem \ref{T:barinduction},  conclude: $\langle\;\rangle \in C$, i.e. there exists $B'\subseteq B$ such that $B'$ is  bar in $\mathcal{F}_\beta$ and $B'$ is a decidable and almost-finite subset of $\omega$.

\medskip (ii) Let $D$ be the set of all $s$ such that either: $\beta(s)\neq 0$ or:\\ $\beta(s)=0$ and $\forall \zeta \in [\omega]^\omega]\exists n[\neg\bigl(s\sqsubseteq \zeta(n)\bigr)\;\vee\;\beta\bigl(\zeta(n)\bigr)\neq 0\;\vee\;\exists t\sqsubseteq \zeta(n)[t\in B]]$.  

 \smallskip Note the following:

(ii)1. $B^+\subseteq D$.

(ii)2. For all $s$, if $s\in D$, then for all $i$, $s\ast\langle i \rangle \in D$.

(ii)3. For all $s$, if $\forall i[s\ast\langle i \rangle \in D]$, then $s \in D$. \\We prove (ii)3  as follows. \\Let $s$ be given such that $\forall i[s\ast\langle i\rangle \in D]$. \\We want to prove: $s \in D$ and may assume: $\beta(s)=0$. 
  \\Let $\zeta$ in $[\omega]^\omega$ be given. \\Define $QED:=\exists n[\neg\bigl(s\sqsubseteq \zeta(n)\bigr)\;\vee\;\beta\bigl(\zeta(n)\bigr)\neq 0\;\vee\;\exists t\sqsubseteq \zeta(n)[t\in B]]$. \\Define $\zeta^\ast$ in $[\omega]^\omega]$ such that, for each $n$, \\{\it if} $\forall i\le n +1 [s\sqsubseteq \zeta(i) \;\wedge\;\beta\bigl(\zeta(i)\bigr)=0]$, then $\zeta^\ast(n)=\zeta(n+1)$ and, \\{\it if not}, then $\zeta^\ast(n)$ is the least $k$ such that $\forall i<n[k>\zeta(i)]$ and $s\sqsubseteq k$ and $\beta(k)=0$.
 \\Note that, for each $n$, $s\sqsubset \zeta^\ast (n)$ and $\beta\bigl(\zeta^\ast(n)\bigr)=0$. \\Define $QED^\ast:= \exists n\exists t\sqsubseteq \zeta^\ast(n)[t\in B]$.
  \\ Note that, by assumption, for each $i$, \\the set $P_i:=\{u\mid\neg( s\ast \langle i \rangle \sqsubseteq u) \;\vee\;\beta(u)\neq 0\;\vee\; \exists t\sqsubseteq u[t\in B] \}$ is almost-full. 
 \\Using Lemma \ref{L:almostfull}, we conclude that, for each $n$, the set $\bigcap_{i\le n}P_i$ is almost-full. \\Using the First Axiom of Countable Choice, we  determine $\eta$ in $[\omega]^\omega$  such that, \\for all $n$, $\zeta^\ast\circ\eta(n)\in \bigcap_{i\le n} P_i$. 
 \\Now find $k:=length(s)$ and define  $\gamma$ such that, for each $n$,  $\gamma(n)=\bigl(\zeta^\ast\circ\eta(n)\bigr)(k)$.  
 \\Note that, for each $n$, if  $\gamma(n) \le n$, then $s\ast \langle \gamma(n)\rangle \sqsubseteq \zeta^\ast\circ\eta (n)$, and, \\as $\zeta^\ast\circ\eta(n) \in P_{\gamma(n)}$, one may conclude $\exists t\sqsubseteq \zeta^\ast\circ\eta(n)[t\in B]$ and $QED^\ast$.
\\Define $\gamma^\ast$ such that, for each $m$, \\{\it if } $\forall n\le m[\gamma(n)> n]$, then $\gamma^\ast (m)=\gamma(m)$, and, {\it if not}, then $\gamma^\ast(m)=m$. \\\ Note that, for each $m$, if $\gamma^\ast(m)\neq \gamma(m)$, then $\exists n\le m[\gamma(n)\le n]$ and $QED^\ast$.  Moreover, $\forall p\exists q[\gamma^\ast(q)>p]$. 
\\Find $\theta$ in $[\omega]^\omega$ such that $\theta(0)=0$ and $\forall n[\gamma^\ast\circ\theta(n+1)>\gamma^\ast\circ \theta(n)]$. 
\\Use the fact that $\beta$ is an almost-fan-law and find $n$ such that $ \beta\bigl(s\ast\langle \gamma^\ast\circ\theta(n)\rangle\bigr)\neq 0$.  
\\Conclude: $\gamma^\ast\circ \theta(n)\neq \gamma\circ\theta(n)$ and:   $QED^\ast$.\\Find $n,t$ such that $t\sqsubseteq \zeta^\ast(n)$ and $t\in B$. There are two cases.\\{\it Case} (ii)3a. $\zeta^\ast(n)=\zeta(n+1)$. Then $QED$. \\ {\it Case} (ii)3b. $\zeta^\ast(n)\neq \zeta(n+1)$. \\Then $\exists i\le n+1[\neg \bigl(s\sqsubseteq \zeta(i)\bigr) \;\vee\; \beta\bigl(\zeta(i)\bigr)\neq 0]$ and again: $QED$.

We thus see: for all $s$, if $\forall i[s\ast\langle i \rangle \in D]$, then $s\in D$. \\This concludes our proof of (ii)3. 

\smallskip Using Theorem \ref{T:barinduction},  conclude: $\langle\;\rangle \in D$, i.e. \\$\forall \zeta \in [\omega]^\omega]\exists n[ \beta\bigl(\zeta(n)\bigr)\neq 0 \;\vee\; \exists t\sqsubseteq \zeta(n)[t\in B]]$.

\end{proof}

 The Fan Theorem may be derived from the Almost-Fan Theorem, as follows. 

\begin{corollary}\label{C:fan3}  Let $\beta$ be a finitary spread-law.
 
 If $B\subseteq \omega$ is a a decidable subset of $\omega$ and a  bar in $\mathcal{F}_\beta$, some finite $B'\subseteq B$ is  bar in $\mathcal{F}_\beta$. \end{corollary} 
 
 \begin{proof}  Let $\beta$ be a finitary spread-law and let $B\subseteq \omega$ be a decidable subset of $\omega$ and a bar in $\mathcal{F}_\beta$.
 
 Applying the Almost-Fan Theorem, find $B'\subseteq B$ such that $B'$ is a bar in $\mathcal{F}_\beta$ and $B'$ is a decidable and almost-finite subset of $\omega$. 
 
 Note that $\mathcal{F}_\beta$ is a fan. Therefore, for all $n$, for all $s$ in $\omega^n$, one may decide: either $\forall \alpha \in \mathcal{F}_\beta \exists i<n[s(i)\sqsubset \alpha]$, or $\exists \alpha \in \mathcal{F}_\beta\forall i <n[s(i)\perp\alpha]$.

  Now define $\eta$  such that, for all $n$, \\\textit{if}  $\exists \alpha \in \mathcal{F}_\beta\forall i<n[\alpha\perp \eta(i)]$, then $\eta(n) = \mu t \in B'\forall i<n[t\neq \eta(i)]$, and, \\\textit{if } $\forall \alpha \in \mathcal{F}_\beta\exists i<n[\ \eta(i)\sqsubset\alpha]$, then $\eta(n)=\mu p[\beta(p)\neq 0 \;\wedge\;\forall i<n[\eta(i)<p]]$.  
  
  Note $\eta\in [\omega]^\omega$. Define $k:=\mu n[\eta(n)\notin B']$ and conclude: $\{\eta(0), \eta(1), \ldots\eta(k-1)\}$ is a finite subset of $B$ and a bar in $\mathcal{F}_\beta$.  \end{proof}

\subsection{Open Induction in $[0,1]$}\label{SS:openinduction}

\begin{definition} For every $n>0$, for every finite sequence \\$\bigl((a_0, b_0), (a_1, b_1), \ldots, (a_{n-1}, b_{n-1})\bigr)$ of pairs of rationals, \\for every rational $c\ge 0$, we define the relation \begin{quote} $\bigl((a_0, b_0), (a_1, b_1), \ldots, (a_{n-1}, b_{n-1})\bigr)$ \emph{covers} $[0,c)$ \end{quote} as follows, by induction.

(i) $n=1$ and \emph{either:}  $a_0<0<b_0$ and $c=b_0$, \emph{or:}  $c=0$.

(ii) $n>1$ and there exists $d$ such that $\bigl((a_0, b_0), (a_1, b_1), \ldots, a_{n-2}, b_{n-2})$ covers $[0,d)$ and \emph{either:} $a_{n-1}<d<b_{n-1}$ and $c=b_{n-1}$ \emph{or:} $c=d$.

\smallskip If $A:=\{(a_0, b_0), (a_1, b_1), \ldots, (a_{n-1}, b_{n-1})\}$ is a finite \emph{set} of pairs of rationals, then, for every rational $c\ge 0$, \emph{$A$ covers $[0,c)$} if and only if, \\for some permutation $\pi$ of  $\{0,1,\ldots, n-1\}$, \\the finite sequence $\bigl((a_{\pi(0)}, b_{\pi(0}), (a_{\pi(1)}, b_{\pi(1)}, \ldots, (a_{\pi(n-1)}, b_{\pi(n-1)})\bigr)$ covers $[0,c)$. \end{definition}

Note that one may \emph{decide}, for all $n>0$, for every finite set $A$ of pairs of rationals, for every rational $c\ge 0$, if $A$ covers $[0,c)$ or not.
\begin{definition}\label{D:progressive}

$\mathcal{X}\subseteq [0,1]$ is called \emph{progressive in $[0,1]$} if and only if, for every $x$ in $[0,1]$, if $[0,x)\subseteq \mathcal{X}$, then $x\in \mathcal{X}$. \end{definition}

The following principle was used by \'E. Borel in the proof of what is now called the Heine-Borel Theorem, see \cite{borel} and \cite[Section 4.3]{veldman2011c}.  Its first proof in an intuitionistic context is due to Th. Coquand, see \cite{coquand} and \cite[Theorem 4.1]{veldman2011c}. 
\begin{theorem}[Principle of Open Induction in ${[} 0, 1{]} $] \label{T:openinduction}\;\\ If $\mathcal{H}\subseteq \mathcal{R}$ is open and progressive in $[0,1]$, then $[0,1]\subseteq \mathcal{H}$. \end{theorem}

\begin{proof} We first construct $\rho$. For each $s$, $\rho(s)=\bigl(\rho'(s), \rho''(s)\bigr)$ will be a pair of rationals such that $0\le \rho'(s)\le\rho''(s)\le 1$.

Define $\rho(\langle\;\rangle):= (0,1)$ and, for each $s$ in $2^{<\omega}$, \\$\rho(s\ast \langle 0 \rangle)):= \bigl(\rho'(s), \frac{1}{3}\rho'(s) + \frac{2}{3}\rho''(s)\bigr)$ and, \\for each $n>0$,  $\rho(s\ast\langle n \rangle):=\bigl(\frac{2}{3}\rho'(s) +\frac{1}{3}\rho''(s), \rho''(s)\bigr)$.

Let $\mathcal{H}\subseteq \mathcal{R}$ be given such that $\mathcal{H}$ is open and progressive in $[0,1]$.

Using Definition \ref{D:opensetsmeasurable}, find $\alpha$ such that $\forall x \in \mathcal{R}[x\in \mathcal{H}\leftrightarrow \exists n[q_{\alpha(n)}<x<r_{\alpha(n)}]]$. 

Now define $\beta$ such that $\beta(\langle\;\rangle)=0$ and, for all $s$, \begin{enumerate}[\upshape (1)] \item if $\beta(s)=0$, then $\beta(s\ast\langle 0 \rangle)=0$ and, \item for all $ n$, $\beta(s\ast\langle n+1 \rangle) =0$ if and only if $n$ is  the least $p$ such that \\$\{(q_{\alpha(i)}, r_{\alpha(i)})\mid i<p\}$ covers $[0, \rho'(s\ast\langle n \rangle)$. \end{enumerate}

One may prove: for each $s$, if $\beta(s)=0$, then $[0, \rho'(s)]\subseteq \mathcal{H}$. The proof is by induction on $length(s)$. 

Note that $\beta$ is a spread-law, and that, for each $s$, if $\beta(s)=0$, there are at most two numbers $n$ such that $\beta(s\ast\langle n \rangle)=0$, so $\beta$ is an almost-finitary spread-law.

We let $B$ be the set of all $s$ such that $\beta(s)=0$ and,\\ for some $n<length(s)$, $q_{\alpha(n)} < \rho'(s) <\rho''(s) < r_{\alpha(n)}$.

The following argument shows that $B$ is a bar in $\mathcal{F}_\beta$. 

Let $\gamma$ in $\mathcal{F}_\beta$ be given.

Note that, for each $n$, $[0, \rho'(\overline \gamma n)]\subseteq \mathcal{H}$.

 Let $x$ be the real number  such that, for each $n$, $x(n)=\rho(\overline \gamma n )$.  

Note:  $[0,x)\subseteq \mathcal{H}$, and, therefore, $x\in \mathcal{H}$.

Find $n$ such that $q_{\alpha(n)}<x<r_{\alpha(n)}$. 

Find $m>n$ such that $q_{\alpha(n)}<\rho'(\overline\gamma m)<\rho''(\overline\gamma m)<r_{\alpha(n)}$ and note: $\overline \gamma m \in  B$.

We thus see: $Bar_{\mathcal{F}_\beta}(B)$.

\smallskip We now  apply  Theorem \ref{T:aft}(i) and find $B'\subseteq B$ such that $B'$ is a bar in $\mathcal{F}_\beta$ and  a decidable and almost-finite subset of $\omega$.

Define $\zeta$ in $\omega^\omega$  as follows, by induction. 

Let $\zeta(0)$ the least element $s$ of $B'$ such that $\rho'(s) < 0 <\rho''(s)$. 

Note: $\{\rho\bigl(\zeta(0)\}$ covers $[0, \rho''\bigl(\zeta(0)\bigr)$, so $[0, \rho''\bigl(\zeta(0)\bigr)\subseteq \mathcal{H}$.

Now let $n$ be given such that we defined $\zeta(0), \zeta(1), \ldots \zeta(n)$ and \\$\{\rho\bigl(\zeta(0)\bigr), \rho\bigl(\zeta(1)\bigr), \ldots, \rho\bigl( \zeta(n)\bigr)\}$ covers $[0, \rho''\bigl(\zeta(n)\bigr))$, so $[0, \rho''(\zeta(n))\bigr)\subseteq \mathcal{H}$. 

Consider $t:=\rho''\bigl(\zeta(n) \bigr)$ and note: $[0,t)\subseteq \mathcal{H}$, and: $t\in \mathcal{H}$. 

If $t<_\mathbb{Q} 1_\mathbb{Q}$, let $\zeta(n+1)$ be the least $s$ in $B'$ such that $\rho'(s)< t< \rho''(s)$ and $\forall i<length(s)[\overline s i\notin B]$. Note: $\forall j<n+1[\zeta(j)\neq \zeta(n+1)]$. 

If $t\ge_\mathbb{Q} 1$, let $\zeta(n+1)$ be the least $s$ such that $s\notin B'$ and $\forall i\le n[s\neq \zeta(i)$. 

Note: $\forall i\forall j[i<j\rightarrow \zeta(i)\neq\ \zeta(j)]$.

Find $\eta$ in $[\omega]^\omega$ such that $\zeta\circ \eta \in [\omega]^\omega$,

Use the fact that $B'$ is almost-finite and note: $\exists n[\zeta\circ\eta(n)\notin B']$.

Let $p$ be the least $n$ such that $\zeta(p+1)\notin B$ and note:\\
$\rho''\bigl(\zeta(p)\bigr)\ge 1$, and: $[0,1]\subseteq \mathcal{H}$.
\end{proof}

\subsection{Dedekind's Theorem}\label{SS:dedekind} The following Theorem shows that a nondecreasing sequence of reals that positively fails to converge grows beyond all bounds. This is a counterpart in intuitionistic analysis to `{\it Dedekind's Theorem}':  \begin{quote}`{\it  A non-decreasing infinite sequence of reals that is bounded from above must have a limit}'. \end{quote}R.~Dedekind's aim in writing \cite{dedekind} was  to give a rigorous proof of this statement.\footnote{If one reads the classical formulation of Dedekind's Theorem constructively, one obtains a false statement: consider the sequence $q_0, q_1, \ldots$ of rationals scuh that, for all $n<k_{99}$, $q_n=0$ and for all $n\ge k_{99}$, $q_n=1$.}

\begin{theorem}[Dedekind's Theorem]\label{T:dedekind}For all $\gamma$, \\if $\forall n[q_{\gamma(n)}\le q_{ \gamma(n+1)}]$ and $\forall \zeta \in [\omega]^\omega\exists n [ q_{\gamma\circ\zeta(n)} + \frac{1}{2^n}< q_{ \gamma\circ\zeta(n+1)}]$, then $\exists n[ 1 <q_{ \gamma(n)}]$.\end{theorem}
 
 Note that the condition: `$\forall \zeta \in [\omega]^\omega\exists n [ q_{\gamma\circ\zeta(n)} + \frac{1}{2^n}< q_{ \gamma\circ\zeta(n+1)}]$' says: the sequence $q_{\gamma(0)}, q_{\gamma(1)}, \ldots$ positively fails to be convergent.
 
 Also note that, once one sees how to obtain the conclusion `$\exists n[ 1 <q_{ \gamma(n)}]$' , one will also see how to obtain the conclusion:  `$\forall M\in \mathbb{Q}\exists n[ M<r_{\gamma(n)}]$', i.e. the sequence $q_{\gamma(0)}, q_{\gamma(1)}, \ldots$  grows beyond all bounds.
 
\begin{proof} Define $\mathcal{H}:=\{x\in \mathcal{R}\mid \exists n[x<q_{\gamma(n)}]\}$. 

Note that $\mathcal{H}$ is an open subset of $\mathcal{R}$.

We now prove that $\mathcal{H}$ is progressive in $[0,1]$. 
\\Let $x$ in $[0,1]$ be given such that $[0,x)\subseteq \mathcal{H}$.\\In particular, for each $n$, $x-\frac{1}{2^n} \in \mathcal{H}$. \\ Find $\zeta$ in $[\omega]^\omega$ such that, for each $n$, $x-\frac{1}{2^n} < q_{\gamma\circ\zeta(n)}$, i.e. $x<q_{\gamma\circ\zeta(n)}+\frac{1}{2^n}$. \\ Find $n$ such that $q_{\gamma\circ\zeta(n)} +\frac{1}{2^n}<q_{\gamma\circ\zeta(n+1)}$ and conclude: $x<q_{\gamma\circ\zeta(n+1)}$, and: $x\in \mathcal{H}$.\\Conclude:  $\forall x \in [0,1][[0,x)\subseteq \mathcal{H}\rightarrow x \in \mathcal{H}]$, i.e. 
$\mathcal{H}$ is progressive in $[0,1]$.

\smallskip Using Theorem \ref{T:openinduction},  conclude: $1\in \mathcal{H}$, i.e.
$\exists n[1<q_{\gamma(n)}]$.  \end{proof}
\subsection{Ramsey's Theorem}\label{SS:ramsey}

The usual formulation of (the two-dimensional case of) the Infinite Ramsey Theorem is the following: 

\begin{quote} Given $R\subseteq [\omega]^2$, there exists an infinite subset $Z$ of $\omega$ such that\\ either $[Z]^2 \subseteq R$ or $[Z]^2\subseteq [\omega]^2\setminus R$. \end{quote}

 Given  any $X\subseteq \omega$, $[X]^2$  denotes the collection of the 2-element-subsets of $X$. 
\\We prefer to define $[X]^2$ as the collection of strictly increasing sequences of elements of $X$ of length 2.
\\We also use the set $[\omega]^\omega$ of the infinite strictly increasing sequences of natural numbers rather than the collection of the infinite subsets of $\omega$. \\We  reformulate Ramsey's theorem as follows: 

\begin{quote} Given $R\subseteq [\omega]^2$, there exists $\zeta$ in $[\omega]^\omega$ such that \\either  $\forall s \in [\omega]^2[\zeta\circ s \in R]$ or $\forall s \in [\omega]^2[\zeta\circ s \notin R]$.  \end{quote}

As finite sequences of natural numbers are coded by natural numbers, we may further simplify this to:

\begin{quote} Given $R\subseteq \omega$, there exists $\zeta$ in $[\omega]^\omega$ such that \\either  $\forall s \in [\omega]^2[\zeta\circ s \in R]$ or $\forall s \in [\omega]^2[\zeta\circ s \in \omega\setminus R]$.  \end{quote}

 If we consider this statement from a constructive point of view, we soon \\discover, thinking of the constructive interpretation of `or', that it can not be true. A counterexample in Brouwer's style is given by the set \begin{quote} $R:=\{s\in [\omega]^2 \mid s(0)<k_{99} \;\wedge\; s(1) < k_{99}\}$. \end{quote}
 (Note that,  for every $\zeta$ in $[\omega]^\omega$, if $\forall s \in [\omega]^2[\zeta\circ s \in R]$, then $\forall n[n< k_{99}]$, and, \\if $\forall s \in [\omega]^2[\zeta\circ s \notin R]$, then $\exists n\le \zeta(0)[n=k_{99}]$.) 
 
 \smallskip One might hope however, that the following holds: 
 
 \begin{quote} Given $R\subseteq \omega$, a contradiction follows from: \\$\neg\exists \zeta\in[\omega]^\omega\forall s \in [\omega]^2[\zeta\circ s \in R]$ and $\neg\exists \zeta\in[\omega]^\omega\forall s \in [\omega]^2[\zeta\circ s \in \omega\setminus R]$. 
 \end{quote}
 
 One might be even more hopeful about: \begin{quote} Given $R\subseteq \omega$, a contradiction follows from: \\$\forall \zeta\in[\omega]^\omega\exists s \in [\omega]^2[\zeta\circ s \in \omega\setminus R]$ and $\forall \zeta\in[\omega]^\omega\exists s \in [\omega]^2[\zeta\circ s \in  R]$. 
 \end{quote}
 
 Indeed, this may be proven intuitionistically. But one may do better and show: 
 
 \begin{quote} Given $R, T\subseteq \omega$, \\ if $\forall \zeta\in[\omega]^\omega\exists s \in [\omega]^2[\zeta\circ s \in \omega\setminus R]$ and $\forall \zeta\in[\omega]^\omega\exists s \in [\omega]^2[\zeta\circ s \in  T]$, \\then $\forall \zeta\in[\omega]^\omega\exists s \in [\omega]^2[\zeta\circ s \in \omega\setminus R\cap T]$.
 \end{quote}
 
 We will prove Ramsey's Theorem and its extension to higher dimensions in the above form. The proof uses the Almost-Fan Theorem. This theorem enables one to use a version of the `\textit{Erd\"os-Rado compactness argument}'.
\begin{definition}\label{D:almostfull} For each $\alpha$, $D_\alpha:=\{n\mid\alpha(n)\neq 0\}$. \\$D_\alpha$ is called \emph{the subset of $\omega$ decided by $\alpha$}. 

For each $k>0$, for each $\alpha$, $D_\alpha$ is  \emph{$k$-almost-full}\footnote{See also Definition \ref{D:almostfull}.} if and only if \\$\forall \zeta \in [\omega]^\omega\exists s \in [\omega]^k[\zeta\circ s \in D_\alpha]$. 

\smallskip For all $n,k$ such that $k\le n$, for all $s$ in $[\omega]^n$, for all $\alpha$, \begin{enumerate} \item $s$ is \emph{$(\alpha, k)$-monochromatic} if and only if \\$\forall t \in [n]^k\forall u \in [n]^k[\alpha(s\circ t)=\alpha(s\circ u)]$, and \item $s$ is \emph{$(\alpha, k+1)$-pre-monochromatic} if and only if \\$\forall t \in [n]^k\forall p \forall q[(t\ast\langle p \rangle \in [n]^{k+1} \;\wedge\; t\ast\langle q \rangle \in [n]^{k+1})\rightarrow \alpha(t\ast\langle p \rangle)=\alpha(t\ast\langle q \rangle)
]$.\end{enumerate} \end{definition}

Note that, if $D_\alpha$ is $k$-almost-full, then $\neg\exists \zeta\in[\omega]^\omega\forall s \in [\omega]^k[\zeta\circ s \in \omega\setminus D_\alpha]$. 

\begin{theorem}[Ramsey's Theorem, the infinite case]\label{T:ramsey}$\;$\\
For all $k>0$,  for all $\alpha, \beta$, \\if $D_\alpha$, $D_\beta$ are $k$-almost-full, then $D_\alpha\cap D_\beta$ is $k$-almost-full. \end{theorem}

\begin{proof} We use induction, and start with the case $k=1$.

Let $\alpha, \beta$ be given such that $D_\alpha, D_\beta$ are $1$-almost-full.

Let $\zeta$ in $[\omega]^\omega$ be given. Find $\eta$ in $[\omega]^\omega$ such that $\forall n \in \omega[\langle\zeta\circ\eta(n)\rangle\in D_\alpha]$. Find $p$ such that $\langle \zeta\circ\eta(p)\rangle \in D_\beta$. Define $q:=\eta(p)$ and note: $\langle \zeta(q)\rangle \in D_\alpha\cap D_\beta$.

We thus see:  $\forall \zeta \in [\omega]^\omega\exists q[\langle \zeta(q)\rangle \in D_\alpha\cap D_\beta$, i.e. $D_\alpha\cap D_\beta$ is $1$-almost-full. 

\smallskip Now assume $k>0$ is given such that   for all $\alpha, \beta$, \\if $D_\alpha$, $D_\beta$ are $k$-almost-full, then $D_\alpha\cap D_\beta$ is $k$-almost-full.

 Let $\alpha, \beta$ be given such that $D_\alpha$, $D_\beta$ are $k+1$-almost-full.
 
 We want to prove: $D_\alpha\cap D_\beta$ is $k+1$-almost-full.
 
 Let $\zeta$ in $[\omega]^\omega$ be given.
 
 We want to prove: $QED:=\exists s\in [\omega]^{k+1}[\zeta \circ s \in D_\alpha\cap D_\beta]$.\footnote{QED: `quod \emph{est} demonstrandum', `what we have to prove' rather than `quod \emph{erat} demonstrandum', what we \emph{did} have to prove'.}

 \smallskip
 We define: $\alpha^\dag = \alpha\circ\zeta$ and $\beta^\dag=\beta\circ \zeta$.

 We  define $\delta$, as follows, by induction.  $\delta(\langle\;\rangle)=0$, and, for all $s$, for all $n$, $\delta(s\ast\langle n \rangle)=0$ if and only if $\delta(s)=0$ and $s$ is the largest $t<n$ such that $\delta(t)=0$ and $t\ast\langle n \rangle$ is both ($k+1,\alpha^\dag)$-pre-monochromatic and  $(k+1,\beta^\dag)$-pre-monochromatic.

The set $D_\delta$ has the property that, for all $s,n$, if $s\ast\langle n \rangle \in D_\delta$, then $s \in D_\delta$, and for this reason, is called a {\it tree}. $D_\delta $ should be called the \emph{$(k+1,\alpha^\dag, \beta^\dag)$-Erd\"os-Rado-tree}.

Note: $D_\delta\subseteq [\omega]^{<\omega}$. 

Note: for each $n$, there exists exactly one $s$ such that $s\ast\langle n \rangle \in D_\delta$. 

Note: for each $n$, for each $s$ in $D_\delta\cap \omega^n$,  there are at most $4^{\binom{n}{k}}$ numbers $i$ such that $s\ast\langle i \rangle\in D_\delta$. This is because, for each $n$,  the set $[n]^k$ has $\binom{n}{k}$ elements, and  for each  $t$ in $[n]^k$,   for each $i$, $(s\circ t)\ast\langle i \rangle$ belongs to one of the four sets $D_{\alpha^\dag} \cap D_{\beta^\dag}$, $D_{\alpha^\dag} \setminus D_{\beta^\dag}$, $D_{\beta^\dag} \setminus D_{\alpha^\dag}$ and $\omega\setminus(D_{\alpha^\dag} \cup D_{\beta^\dag})$.

	  Define $\varepsilon$ such that $\forall s[\varepsilon(s) = 0\leftrightarrow \exists t\exists n[t \in D_\delta\;\wedge \;s = t\ast\overline{\underline 0}n]]$. 
	  
	  Note that $\varepsilon$ is an almost-finitary spread-law and that the set $\mathcal{F}_\varepsilon$ is  an almost-finitary spread.
	 
	 Define $B:=\bigcup_n\{s\in \omega^n\mid\varepsilon(s)=0 \;\wedge\;(\exists t\in [n]^{k+1}[s\circ t \in D_{\alpha^\dag}\cap D_{\beta^\dag}]\;\vee\; s \notin D_\delta)\}$.

	    We  now prove that $B$ is a bar in the almost-finitary spread $\mathcal{F}_\varepsilon$.  
	   
	   \smallskip
	  Assume: $\gamma\in\mathcal{F}_\varepsilon$.

	 Define $\gamma^\ast$ such that, for each $n$, if $\overline {\gamma}(n+1) \in D_\delta$ 
	  then $\overline{\gamma^\ast}(n+1) = \overline\gamma(n+1)$, and, if not, then $\overline{\gamma^\ast}(n+1)\in[\omega]^{n+1}\setminus D_\delta$.
	 Note: $\gamma^\ast \in [\omega]^\omega$ and $\forall n[\overline \gamma n \neq \overline{\gamma^*}n\rightarrow\overline \gamma n \notin D_\delta]$.
	 
	Recall: $k>0$.
	For each $t$ in $[\omega]^k$, we let $t^+$ be the element of $[\omega]^{k+1}$ satifying $t\sqsubset t^+$ and $t^+(k)=t(k-1)+1$.
	 
	   Define $\alpha^\ast$ and $\beta^\ast$  such that, for each $t$ in $[\omega]^k$,\\
	    $\alpha^\ast(t) = \alpha^\dag\bigl(\gamma^\ast\circ t^+)$ and $\beta^\ast(t) = \beta^\dag\bigl(\gamma^\ast\circ t^+)$.
	    
	We now prove: $\forall \eta \in [\omega]^\omega \exists t\in [\omega]^k [\eta \circ t \in D_{\alpha^\ast} \;\vee \;\exists n[\overline{\gamma^\ast}n\notin D_\delta]].$
	\\Assume: $\eta\in[\omega]^\omega$. Find $s$ in $[\omega]^{k+1}$ such that $\gamma^\ast \circ \eta \circ s \in D_{\alpha^\dag}$.
	\\Define: $n:=(\eta \circ s)(k) +1$. \\Define $t:= \overline s k$ and $i:= s(k)$. \\Note: $\gamma^\ast\circ \eta \circ (t\ast\langle i \rangle) \in D_{\alpha^\dag}$, i.e. $\gamma^\ast\circ \bigl((\eta \circ t)\ast\langle \eta (i)\rangle\bigr) \in D_{\alpha^\dag}$. Conclude: \\{\it either}: $\gamma^\ast\circ(\eta\circ  t)^+ \in D_{\alpha^\dag}$.
 and therefore: 	$\eta\circ t \in D_{\alpha^\ast}$, \\{\it or}: $(\eta\circ t)^+ \neq  \eta\circ s$ and $\overline {\gamma^\ast}n$ is not $(k+1, \alpha^\dag)$-pre-monochromatic,\\ and, therefore: $\overline {\gamma^\ast} n \notin D_\delta$.
	
	One may  prove by a similar argument:

	$\forall \eta \in [\omega]^\omega\exists t\in [\omega]^k [\eta \circ t \in D_{\beta^\ast}\;\vee \;\exists n[\delta\bigl(\overline{\gamma^\ast}n\bigr)=0]]$.

	Define $\alpha^{\ast\ast},\beta^{\ast\ast}$  such that $\forall s[\alpha^{\ast\ast}(s)\neq 0 \leftrightarrow\bigl(\alpha^\ast(s)\neq 0\;\vee\;\exists n\le s[\delta(\overline{\gamma^\ast}n)=0]\bigr)]$ and $\forall s[\beta^{\ast\ast}(s)\neq 0 \leftrightarrow\bigl(\beta^\ast(s)\neq 0\;\vee\;\exists n\le s[\delta(\overline{\gamma^\ast}n)=0]\bigr)]$. 

 Conclude:
	$\forall \eta \in [\omega]^\omega \exists t \in [\omega]^k[\eta \circ t \in D_{\alpha^{\ast\ast}}] \;\wedge\; \forall \eta \in [\omega]^\omega \exists t \in [\omega]^k[\eta \circ t \in D_{\beta^{\ast\ast}}] $.

	Using the induction hypothesis,  we  conclude:

	$\forall \eta \in [\omega]^\omega \exists t \in [\omega]^k[\eta \circ t \in D_{\alpha^{\ast\ast}}\cap D_{\beta^{\ast\ast}}]$.

	Find $t$ in $[\omega]^k$ such that $t \in D_{\alpha^{\ast\ast}} \cap D_{\beta^{\ast\ast}}$.  
	
	\textit{Either:} $t \in D_{\alpha^\ast} \cap D_ {\beta^\ast}$ \textit{or:} $\exists n[\delta(\overline{\gamma^\ast}n) = 0]$, that is,
	 \\ \textit{either:} $\gamma^\ast\circ(t\ast\langle j \rangle) = (\gamma^\ast \circ t)\ast \langle \gamma^\ast( j)\rangle \in D_{\alpha^\dag} \cap D_{\beta^\dag}
	 $, where $j = t(k-1) + 1$, \textit{or:} $\exists n[\delta(\overline{\gamma^\ast} n) = 0]$. In both cases, we find $n$ such that $\overline {\gamma^\ast} n \in B$.
	  
	   \textit{Either:} $\overline \gamma n = \overline{\gamma^\ast} n$ \textit{or:} $\overline \gamma n \notin D_\delta$. In both cases: $\overline \gamma n \in B$.

	We thus see: $\forall \gamma\in\mathcal{F}_\varepsilon\exists n[\overline\gamma n \in B]$, i.e. $Bar_{\mathcal{F}_\varepsilon}(B)$.

	\smallskip
	We now use Theorem \ref{T:aft}(ii). \\Find $\eta$ in $[\omega]^\omega$ such that $D_\delta=\{\eta(n)\mid n\in \omega\}$. \\Then find $n,m$ such that $\overline{\eta(n)}m \in B$.  \\Conclude: $\exists t \in [\omega]^{k+1}[s\in D_{\alpha^\dag} \cap D_{\beta^\dag}]$, i.e. $\exists t \in [\omega]^{k+1}[ \zeta\circ t \in D_\alpha\cap D_\beta]$. 
	
	We thus see: $\forall \zeta \in [\omega]^{\omega}\exists t \in [\omega]^{k+1}[ \zeta\circ t \in D_\alpha\cap D_\beta]$, i.e. \\$D_\alpha\cap D_\beta$ is $(k+1)$-almost-full. \end{proof} 
	
	\subsection{The Bolzano-Weierstrass Theorem}\label{SS:bw} The Bolzano-Weierstrass Theorem:\begin{quote}`{\it An infinite sequence of reals bounded both from above and from below must have a convergent subsequence}'. \end{quote} is a strengthening of Dedekind's Theorem, see Subsection \ref{SS:dedekind}.  As Dedekind's Theorem, in its usual formulation, already fails to be true constructively, the case of the usual formulation of the Bolzano-Weierstrass Theorem is also hopeless.\begin{theorem}[Bolzano-Weierstrass-Theorem]\label{T:bolzanow} For all $\gamma$, \\if  $\forall \zeta \in [\omega]^\omega\exists n [ \frac{1}{2^n}<|q_{\gamma\circ\zeta(n} -q_{\gamma\circ\zeta(n+1)}|]$, then $\exists n[ 1<|q_{\gamma(n)}|]$.\end{theorem} 
	Note that the condition: `$\forall \zeta \in [\omega]^\omega\exists n [ \frac{1}{2^n}<|q_{\gamma\circ\zeta(n} -q_{\gamma\circ\zeta(n+1)}|]$' says: every subsequence of the sequence $q_{\gamma(0)}, q_{\gamma(1)}, \ldots$ positively fails to be convergent.
	
	Also note that, once one sees how to obtain the conclusion `$\exists n[ 1 <|q_{ \gamma(n)}|]$' , one will also see how to obtain the conclusion:  `$\forall M\in \mathbb{Q}\exists n[ M<|r_{\gamma(n)}|]$', i.e. the sequence $q_{\gamma(0)}, q_{\gamma(1)}, \ldots$  grows beyond all bounds.
	
	Our proof uses both Dedekind's Theorem and Ramsey's Theorem. 
	\begin{proof}Let $\gamma$ be given such that $\forall \zeta \in [\omega]^\omega\exists n [ \frac{1}{2^n}<|q_{\gamma\circ\zeta(n+1)} -q_{\gamma\circ\zeta(n)}|]$.

	\smallskip

	We first prove:\begin{quote} $(\ast)$ $\forall \zeta \in [\omega]^\omega\exists n[q_{\gamma\circ\zeta(n+1)}< q_{\gamma\circ\zeta(n)} \;\vee \;  1<q_{\gamma(n)}]$. \end{quote}
	
	Let $\zeta$ in $[\omega]^\omega$ be given.  Define $\delta=\gamma\circ\zeta$.
	
	Define $\delta^\ast$ such that $\delta^\ast(0)=\delta(0)$ and, for each $n$, {\it if} $\forall i \le n[q_{\delta(i)}\le q_{ \delta(i+1)}]$, then $\delta^\ast(n+1)=\delta(n+1)$ and, {\it if not}, then $q_{\delta^\ast(n+1)}= q_{\delta^\ast(n)}+1$.
	
	Note: $\forall n[q_{\delta^\ast(n)}\le q_{\delta^\ast(n+1)}]$.
	
	We now prove: $\forall \eta \in [\omega]^\omega \exists n[q_{\delta^\ast\circ\eta(n)}+\frac{1}{2^n} < q_{\delta^\ast\circ \eta(n+1)}]$. 
	
	Let $\eta$ in $[\omega]^\omega$ be given. Find $n$ such that $\frac{1}{2^n}<|q_{\gamma\circ\zeta\circ\eta(n+1)}-q_{\gamma\circ \zeta \circ \eta(n)}|$. {\it Either} $q_{\delta^\ast\circ \eta(n)}=q_{\gamma\circ\zeta\circ\eta(n)}$ and $q_{\delta^\ast\circ \eta(n+1)}=q_{\gamma\circ\zeta\circ\eta(n+1)}$ and $q_{\delta^\ast\circ\eta(n)}+\frac{1}{2^n} < q_{\delta^\ast\circ \eta(n+1)}$, {\it or} $\exists m<\zeta\circ\eta(n+1)[q_{\delta(m+1)}< q_{\delta(m)}]$. In the latter case, for all sufficiently large $i$, $q_{\delta^\ast\circ \eta(i)} + 1\le q_{\delta^\ast\circ\eta(i+1)}$. 
	
	\smallskip Using Theorem \ref{T:dedekind}, we find $n$ such that $1<q_{\delta^\ast(n)}$. \\{\it If} $q_{\delta^\ast(n)} = q_{\delta(n)}$, we conclude: $\exists m[1<q_{\gamma(m)}]$, and {\it if not}, \\we conclude: $\exists i\le n[q_{\delta(i+1)}<q_{\delta(i)}]$, i.e.  $\exists i\le n[q_{\gamma\circ\zeta(i+1)}<q_{\gamma\circ\zeta(i)}]$.
	
	 This concludes our proof of $(\ast)$. 
	 
	 \smallskip One may also prove: 
	 \begin{quote} $(\ast\ast)$ $\forall \zeta \in [\omega]^\omega\exists n[q_{\gamma\circ\zeta(n)}< q_{\gamma\circ\zeta(n+1)} \;\vee \;  q_{\gamma(n)}<-1]$. \end{quote}
	 
	 (Find $\gamma^\ast$ such that $\forall n[q_{\gamma^\ast(n)}=-q_{\gamma(n)}]$ and use $(\ast)$, but now for $\gamma^\ast$ rather than for $\gamma$ itself.)
	 
	 \smallskip Using both $(\ast)$ and $(\ast\ast)$ and Theorem \ref{T:ramsey}, conclude: \begin{quote} $\exists n[q_{\gamma(n)}<-1\;\vee\; 1<q_{\gamma(n)}]$, i.e. $\exists n[1<|q_{\gamma(n)}|]$. \end{quote}

	 \end{proof}

\subsection{The Paris-Harrington-Ramsey  Theorem}\label{SS:phr}
F.P.  Ramsey proved the Infinite Ramsey Theorem, in \cite{ramsey}, in order to make his reader  gain experience before attacking the Finite Ramsey Theorem. Later, it turned out that one may prove the Finite Ramsey Theorem from the Infinite Ramsey Theorem  by a so-called `{\it compactness argument}', see \cite{debruijnerdos}.  Paris and Harrington then saw that one may prove also certain strengthenings of the Finite Ramsey Theorem from the Infinite Ramsey Theorem, see \cite[Sections 1.5 and 6.3]{graham}.  One such statement turned out to be expressible in the language of first-order arithmetic but unprovable from Peano's axioms. We want to show that the `compactness argument' works also intuitionistically, thanks to the intuitionistic version of the Infinite Ramsey Theorem proven in Subsection \ref{SS:ramsey} and the Fan Theorem. 

We need some terminology in order to introduce the Finite Ramsey Theorems.

\begin{definition}  
     For all $r>0$, $r^{<\omega}:=\{c\mid \forall i<length(s)[c(i)<r]\}$.
 
 \smallskip For all positive integers $m,k$, $[m]^k:=\{s\in [\omega]^k\mid\forall i<k[s(i) < m]\}$. 
 
 \emph{One may consider elements of $[m]^k$ as $k$-element subsets of $m=\{0,1,\ldots, m-1\}$}. 
   
   \smallskip For all positive integers  $c,m,k,r$, \\$c:[m]^k\rightarrow r$ if and only if $c\in r^{<\omega}$  and $\forall s \in [m]^k[s <\mathit{length}(c)]$. 
   
  \emph{One may consider $c:[m]^k\rightarrow r$ as an $r$-colouring of the $k$-element subsets of $m$.}
   
   \smallskip For all positive integers $c,k,m,r,t,n$, if $c:[m]^k\rightarrow r$ and $n\le m$ and $t\in [m]^n$, then \emph{$t$ is $c,k$-monochromatic} if and only if $\exists j<r\forall u \in [n]^k[c(t\circ u)=j]$.

   \smallskip
   For all positive integers $k,r,n,M$, $M\rightarrow (n)^k_r$ if and only if,\\ for every $c:[M]^k\rightarrow r$, there exists    $t$ in $[M]^n$  such that  $t$ is $c,k$-monochromatic. 
   
   \emph{If $M\rightarrow (n)^k_r$, then, for every $r$-colouring $c$ of the $k$-element subsets of $M=\{0,1, \ldots, M-1\}$ there exists an $n$-element subset $A=\{t(0), t(1), \ldots, t(n-1)\}$ of $M$ such that all $k$-element subsets of $A$ obtain, from $c$, one and the same colour}. 
   
   \smallskip $t\in\bigcup_n[\omega]^n$ is \emph{relatively large} if and only if $n>0$ and $length(t)\ge t(0)$.
   
   \emph{The expression \emph{relatively large} is used in \cite{parisharrington}. A finite subset $A$ of $\omega$ is relatively large if the number of elements of $A$ is at least as big as the smallest member of $A$.}
   
   \smallskip
   For all positive integers $k,r,n,M$, $M\rightarrow_{\ast} (n)^k_r$ if and only if,\\ for every $c:[M]^k\rightarrow r$,  there exist   $t$  in $[M]^{<\omega}$ such that $length(t)\ge n$ and $t$ is relatively large and  $c,k$-monochromatic.

   \end{definition}

	 We want to call a collection $A$ of finite subsets of $\omega$ \emph{omnipresent} if and only if every infinite subset of $\omega$  has a subset in $A$.  With our terminology, the definition takes the following form.
	 
	 \begin{definition}\label{D:omnipresent} $A\subseteq [\omega]^{<\omega}$  is called \emph{ omnipresent} if and only if \\$\forall \zeta\in[\omega]^\omega\exists s \in [\omega]^{<\omega}[\zeta\circ s \in A]$. 
	 
	 \smallskip For every positive integer $r$, $r^\omega:=\{\chi\mid\forall i[\chi(i)<r]\}$. 
	 
	 \emph{Note that $r^\omega$ is a fan}. 
	 
	 \emph{One may consider an element $\chi$ of $r^\omega$ as an 
	 $r$-colouring of $\omega$}.  \end{definition}
	 \begin{theorem}\label{T:omnipresent} Let $A\subseteq [\omega]^{<\omega}$ be a decidable subset of $\omega$ and omnipresent. \begin{enumerate}[\upshape (i)] \item For all positive integers $k,r$, for all $\chi$ in $r^\omega$, the set \\ $B=B(A, k,r, \chi):=\{s\in A\mid s \;is\; (\chi, k)$-$monochromatic\}$ is omnipresent. \item For all positive integers $k,r$, the set $C:= C(A,k,r):=\\\{s \in [\omega]^{<\omega}\mid \forall \chi\in r^\omega\exists t\in [\omega]^{<\omega}[s\circ t \in A \;\wedge\;s\circ t\;is \;(\chi,k)$-$monochromatic ]\}$ \\is a bar in $[\omega]^\omega$. \end{enumerate} \end{theorem} 
	 
	 \begin{proof} (i) Let $A\subseteq [\omega]^{<\omega}$ be a decidable subset of $\omega$ and  omnipresent. \\Let $k>0, r>0$ and $\chi$ in $r^\omega$ be given. \\We are going to prove that $B:=\{s\in A\mid s \;is\; (\chi, k)$-$monochromatic\}$ \\is omnipresent.
	 
	 \smallskip Let $\zeta \in [\omega]^\omega$ be given.
	 
	  We want to prove $QED:= \exists s[\zeta\circ s \in A\;\wedge\;\zeta\circ s \;is\; (\chi,k)$-$monochromatic]$.
	  
	 Let $j<c$ and $\eta$ in $[\omega]^\omega$ be given. Find $s$ such that $\zeta\circ \eta \circ s  \in A$.\\ Note: either $\zeta\circ\eta\circ s \;is\; (\chi,k)$-$monochromatic]$ or $\exists t \in [\omega]^k[\chi(\zeta\circ \eta\circ s\circ t)\neq j]$.
	
	Define $\chi^\ast:\omega\rightarrow r+1$ such that, for  all $u$ in $[\omega]^k$, \\{\it if} $\exists s\le u[\zeta \circ s \in A\;\wedge\; \zeta\circ s \; is\; (\chi,k)$-$monochromatic]$, then $\chi^\ast(u)=r$ and, \\{\it if not}, then $\chi^\ast(u)=\chi(u)$. 
	
	Note: $\forall j<r\forall \eta \in [\omega]^\omega \exists u \in [\omega]^k[\chi^\ast (u)= r \;\vee\; \chi^\ast(u) \neq j]$.
	
	 Ramsey's Theorem, the infinite case, (Theorem \ref{T:ramsey}), implies that the intersection of $r$ decidable subsets of $\omega$ each of which is $k$-almost-full, is $k$-almost-full itself. 
	 
	 Conclude:  $\forall \eta \in [\omega]^\omega \exists u \in [\omega]^k\forall j<r[\chi^\ast (u)= r \;\vee\; \chi^\ast(u) \neq j]$. 
	
	Conclude: $\exists u \in [\omega]^k[\chi^\ast(u)=r]$ and $QED$. 
	\\We thus see that the set $B=B(A,k,r,\chi)$  is omnipresent.

	\smallskip (ii) Let positive integers $k,r$ and $\zeta$ in $[\omega]^\omega$ be given.
	
	By (i), for each $\chi$ in $r^\omega$, one may find $n$ such that \\$\exists t \in [n]^{<n+1}[\overline \zeta n \circ t \in A \;\wedge\; \overline \zeta n \circ t \;is\;(\chi, k)$-$monochromatic]$. 
	
	Conclude that, for each $\chi$ in $r^\omega$, one may find $m$ such that, for some $n<m$, \\$\exists t \in [n]^{<n+1}[\overline \zeta n \circ t \in A \;\wedge\; \overline \zeta n \circ t \;is\;(\overline \chi m, k)$-$monochromatic]$.
	
	Applying the Fan Theorem, one may find $m$ such that for each $\chi$ in $r^\omega$, for some $n<m$, $\exists t \in [n]^{<n+1}[\overline \zeta n \circ t \in A \;\wedge\; \overline \zeta n \circ t \;is\;(\overline \chi m, k)$-$monochromatic]$.
	
	Clearly, $\overline \zeta m \in C(A,k,r)$. 
	\\We thus see that the set $C=C(A,k,r)$ is a bar in $[\omega]^\omega$.  \end{proof}
	
	\begin{corollary}\label{C:finiteramsey}$\;$ \begin{enumerate}[\upshape (i)] \item The Finite Ramsey Theorem: $\forall k \forall r \forall n \exists M[M\rightarrow (n)^k_r]$. \item The Paris-Harrington-Ramsey Theorem: $\forall k \forall r \forall n \exists M[M\rightarrow_\ast (n)^k_r]$. \end{enumerate} \end{corollary}
	
	\begin{proof} (i) Define $A:=[\omega]^n$. Note that $A$ is omnipresent. Let positive integers $k,r$ be given. Using Theorem \ref{T:omnipresent}, conclude that the set \\$C=\{s\in \omega]^{<\omega}\mid \forall \chi \in r^\omega\exists t\in [\omega]^{<\omega}[s\circ t \in A\;\wedge\; s\circ t\; is\; (\chi$-$k)$-$monochromatic]\}$ is a bar in $[\omega]^\omega$. Consider $Id_\omega$, the element of $[\omega]^\omega$ such that $\forall n[Id_\omega(n)=n]$. Find $M$ such that $\overline{Id_\omega}M \in C$ and note: $M\rightarrow (n)^k_r$.
	
	(ii) Start with $A:=\bigcup_n \{s\in [\omega]^{n+1}\mid s(0) \le n\}$. Note that $A$ is omnipresent  and repeat the argument given for (i). \end{proof}
	
	\section{Notation and conventions}\label{S:notat}
\subsection{} $\omega$ denotes the set of the natural numbers $0,1,2,\ldots$.
 
 \smallskip We use $a, b, \ldots m,n, \ldots p,q,r,s, \ldots$ as variables over $\omega$.
 
 We assume a  bijective function $J:\omega\times\omega\rightarrow \omega$ has been defined with inverse functions $K,L:\omega\rightarrow\omega$ such that $\forall n[J\bigl(K(n), L(n)\bigr)=n]$.
 
 $\forall m\forall n[(m,n):=J(m,n)]$ and $\forall n[n':=K(n)]$ and $\forall n[n'':=L(n)]$. 
 
 \smallskip We assume a function $<\;>:\bigcup_k \omega^k \rightarrow \omega$ has been defined that is a bijection.
 
 If  $\langle m_0, m_1, \ldots,m_{k-1}\rangle =s$ then $length(s) =k$, and for each $i<k$, $s(i)=m_i$.

 $\omega^k:=\{s\mid length(s) = k\}$.

\smallskip $\omega^{<\omega}:=\bigcup_k \omega^k$.

\smallskip For all $k,l$, for all $s$ in $\omega^k$, for all $t$ in $\omega^l$, $s\ast t$ is the element $u$ of $\omega^{k+l}$ satisfying $\forall i<k[u(i)=s(i)]$ and $\forall j<l[u(k+j)=t(j)]$.

\smallskip
For every $s$ in $\omega$, for every  $A\subseteq\omega$, 
$s\ast A := \{s\ast t \mid t \in A\}.$

\smallskip For all $s,t$ such that $\forall n<length(t)[t(n)<length(s)]$, $s\circ t$ is the number $u$ satisfying $length(u)=length(t)$ and\\ $\forall n<length(u)[u(n)=s\bigl(t(n)\bigr)]$. 

\smallskip $\langle\;\rangle$ denotes the \textit{empty sequence}, that is, the unique $s$ such that \\$length(s)=0$.

\smallskip $s\sqsubseteq t\leftrightarrow \exists u[t=s\ast u]$.

\smallskip $s\sqsubset t\leftrightarrow (s\sqsubseteq t \;\wedge\; s\neq t)$. 

\smallskip $s\perp t\leftrightarrow \neg(s\sqsubseteq t \;\vee\; t\sqsubseteq s)$.

\smallskip
For all $s$, for all $n\le length(s)$, $\overline s n$ is the unique $u$ in $\omega^n$ such that $u\sqsubseteq s$.

\smallskip For all $s$, $n$, $s^n$ is the largest $u$ such that \\$\forall m<length(u)[\langle n \rangle\ast m<length(s) \;\wedge\; u(m)=s(\langle n \rangle \ast m)]$. 

\smallskip $2^{<\omega}:=\{s\mid\forall n<length(s)[s(i)<2]\}$. 

\smallskip 
$Bin_n:=\{s\in 2^{<\omega}\mid length(s)=n\}$. 

\smallskip
 $[\omega]^k:=\{s\in \omega^k\mid\forall n<k-1[s(n)<s(n+1)]\}$. 

\smallskip $[\omega]^{<\omega}:=\bigcup_k [\omega]^k$.

 \smallskip $\mu n[P(n)]=k\;\;\leftrightarrow \;\;\bigl(P(k)\;\wedge\;\forall n<k[\neg P(n)]\bigr)$. 

\subsection{}

$\omega^\omega=\mathcal{N}$ is the set of all functions from $\omega$ to $\omega$.

\smallskip We use $\alpha, \beta, \ldots, \zeta, \eta, \ldots, \varphi, \psi, \ldots$ as variables over $\omega^\omega$.

\smallskip $Id$ is the element of $\omega^\omega$ satisfying $\forall n[Id(n)=n]$.

\smallskip $\mathcal{C}=2^\omega$ is the set of all $\alpha$ in $\omega^\omega$ such that 
$\forall n[\alpha(n)<2]$.

\smallskip
For each $n$, $\underline n$ is the element $\beta$ of $\omega^\omega$ such that $\forall m[\beta(m)=n]$. 

\smallskip $\alpha\;\#\:\beta \leftrightarrow \alpha\perp\beta\leftrightarrow \exists n[\alpha(n)\neq \beta(n)]$.

\smallskip $\overline \alpha k:=\langle  \alpha(0), \alpha(1), \ldots \alpha(k-1)\rangle$.

\smallskip $s\sqsubset\alpha \leftrightarrow \overline \alpha\bigl(length(s)\bigr)= s$.

\smallskip For all $\mathcal{X}\subseteq \omega^\omega$, for all $s$, $\mathcal{X}\cap s:=\{\alpha\in\mathcal{X}\mid s\sqsubset \alpha\}$.

\smallskip $s\perp \alpha\leftrightarrow \alpha\perp s\leftrightarrow \neg(s\sqsubset \alpha)$. 

\smallskip $(\alpha\upharpoonright s)(t):=\alpha(s\ast t)$.

\smallskip $(\alpha\upharpoonright\langle n \rangle)(m) :=\alpha^n(m) :=\alpha(\langle n \rangle \ast m)$. 

\smallskip For every $X\subseteq \omega$, for every $s$, $X\upharpoonright s:=\{t\mid s\ast t \in X\}$.
 
\smallskip $[\omega]^\omega$ is the set of all $\alpha$ in $\omega^\omega$ such that $\forall n[\alpha(n)<\alpha(n+1)]$.

\subsection{}$\mathcal{F}\subseteq\omega^\omega$ is a spread if and only if \\$\exists \beta[\forall s[\beta(s)=0\leftrightarrow\exists n[\beta(s\ast\langle n \rangle )=0]]\;\wedge\;\forall\alpha[\alpha \in \mathcal{F}\leftrightarrow \forall n[\beta(\overline\alpha n)=0]]]$.

\smallskip Let $\mathcal{F}\subseteq \omega^\omega$ be a spread.

 $\varphi:\mathcal{F}\rightarrow\omega\;\;\leftrightarrow\;\; \forall \alpha\in \mathcal{F}\exists n[\varphi(\overline \alpha n)\neq 0]$. 

If $\varphi:\mathcal{F}\rightarrow\omega$ then, for each $\alpha$ in $\mathcal{F}$, $\varphi(\alpha)$ is the number $q$ such that $\exists n[\varphi(\overline \alpha n)=q+1\;\wedge\;\forall m<n[\varphi(\overline \alpha m)=0]]$.

\medskip
$\varphi:\mathcal{F}\rightarrow \omega^\omega\;\;\leftrightarrow \;\;\forall n[\varphi^n:\mathcal{F}\rightarrow \omega]$.

If $\varphi:\mathcal{F}\rightarrow \omega^\omega$, then, for each $\alpha$ in $\mathcal{F}$, $\varphi|\alpha$ is the element $\beta$ of $\omega^\omega$ such that $\forall n[\beta(n)=\varphi^n(\beta)]$. 

\smallskip If $\varphi:\mathcal{F}\rightarrow \omega^\omega$, then, for each $s$ such that $\exists \delta \in \mathcal{F}[a\sqsubset\delta]$, \\$\varphi|s$ is the greatest number $t\le s$ such that \\$\forall i<length(t)\exists n<length(s) [\varphi^i(\overline s n) =t(i)+1\;\wedge\;\forall j<n[\varphi^i(\overline s j)=0]]$. 

Note: if $\varphi:\mathcal{F}\rightarrow \omega^\omega$, then, for all $\alpha$ in $\mathcal{F}$, for all $ \beta$,\\ $\varphi|\alpha =\beta \;\;\leftrightarrow \;\;\forall n \exists m[\overline \beta n\sqsubseteq \varphi|\overline \alpha m]$. 

\smallskip
For all $\varphi:\omega^\omega\rightarrow\omega^\omega$, for all $\mathcal{X}\subseteq \omega^\omega$, $\varphi|\mathcal{X}:=\{\varphi|\alpha\mid\alpha\in \mathcal{X}\}$.

\smallskip
For each $V\subseteq\omega^\omega$, $\varphi:\mathcal{F}\hookrightarrow V$ if and only if $\varphi:\mathcal{F}\rightarrow \omega^\omega$ and \\$\forall \alpha \in \mathcal{F}\forall \beta\in\mathcal{F}[\alpha\;\#\;\beta\rightarrow \varphi|\alpha\;\#\;\varphi|\beta]$ and $\forall \alpha \in \mathcal{F}[\varphi|\alpha \in V]$.

\smallskip For each $\mathcal{X}\subseteq \omega^\omega$, $\mathcal{X}^\neg:=\{\alpha\mid\alpha\notin\mathcal{X}\}$. 

\smallskip $\mathcal{R}$, the set of the real numbers, may be defined as a subset of $\omega^\omega$. 

\smallskip For $x,y$ in $\mathcal{R}$, $x\;\#_\mathcal{R}\;y\leftrightarrow \exists n[|x-y|>_\mathcal{R} \frac{1}{2^n}$.


\begin{thebibliography}{80}
 
 
 \bibitem{bishopbridges85} E. Bishop and D. Bridges, \textit{Constructive Analysis}, Springer Verlag, Berlin, 1985.\bibitem{borel}\'E. Borel, Sur quelques points de la th\'eorie des fonctions, {\it Annales Scientifiques de l'\'Ecole Normale Sup\'erieure} (3)12 (1895),  pp. 9-55, also in \cite[pp. 239-287]{borelcw}.

\bibitem{borelcw} \textit{Oeuvres de \'Emile Borel}, Tome 1, \'Editions du Centre National de Recherche Scientifique, Paris, 1972. \bibitem{bridgesrichman87} D. Bridges and F. Richman, \textit{Varieties of Constructive Mathematics}, Cambridge Universiry press, Cambridge 1987.
\bibitem{bridgesvita06} D.S.~Bridges and L.S.  V\^i\c{t}\u{a}, {\it Techniques of Constructive Analysis}, Springer, New York, 2006.

\bibitem{brouwer18} L.E.J. Brouwer, Begr\"undung der mengenlehre unabh\"angig vom Satz vom ausgeschlossenem Dritten. erster Teil: Allgemeine Mengenlehre. {\it KNAW Verhandelingen} $1^e$ sectie 12 no. 5, also in \cite[150-190]{brouwer75}.
\bibitem{brouwer27} L.E.J. Brouwer, \"Uber Definitionsbereiche von Funktionen, \textit{Math. Annalen} 97(1927)60-75, also in \cite{brouwer75}, pp. 390-405.



\bibitem{brouwer54} L.E.J. Brouwer, Points and spaces, \textit{Can. J. Math.} 6(1954)1-17, also in \cite{brouwer75}, pp. 522-538.
\bibitem{brouwer75} L.E.J.~Brouwer, \textit{Collected Works, Vol. I: Philosophy and
 Foundations of Mathematics}, ed.~A.~Heyting, North Holland Publ. Co., Amsterdam,
 1975. \bibitem{debruijnerdos} N.G.~de Bruijn and P.~Erd\"os, A Color Problem for Infinite Graphs and a Problem in the Theory of Relations,  \textit{Nederl. Akad. Wetensch. Proc. Ser. A} 54(1951)371-373. 
 



\bibitem{coquand}T.~Coquand, A Note on the Open Induction Principle, 1997, www.cse.chalmers.se/~coquand/open.ps
 
\bibitem{dedekind} R.~Dedekind, \textit{Stetigkeit und Irrationalzahlen}, Braunschweig, 1872.


\bibitem{driessen} B.~Driessen, {\it Intuitionistic Probability Theory}, Master Thesis in Mathematics, Radboud University, Nijmegen, 2018.
 
 
 \bibitem{graham} R.L. Graham, B.L. Rothschild and  J.H. Spencer, \textit{Ramsey Theory}, John Wiley and Sons, New York, { \it Second  Edition}, 1990.
 
  \bibitem{heyting56} A. Heyting, Intuitionism, an Introduction, \textit{North Holland Publ. Co.} 1956. \bibitem{howard} W.A. Howard, G. Kreisel,
Transfinite induction and bar induction of types zero and one, and the role of continuity intuitionistic analysis
{\it Journal of Symbolic Logic}, 31 (1966)325-358.

\bibitem{kleene52} S.C. Kleene, Recursive functions and intuitionistic mathematics, \textit{Proceedings of the International Congress of mathematicians,(Cambridge, Mass., U.S.A., Aug. 30 - Sept. 6, 1950)}, 1952, vol. I, pp. 679-685.

\bibitem{kleenevesley65} S.C. Kleene, R.E. Vesley, The Foundations of Intuitionistic mathematics, especiaaly in relation to Recursive Functions, \textit{North-Holland Publ.Co.}, 1965.

 \bibitem{parisharrington} J. Paris, L. Harrington, A Mathematical Incompleteness in Peano Arithmetic, in: \textit{Handbook of Mathematical Logic, ed. J. Barwise, Studies in Logic and the Foundations of Mathematics, vol. 90}, Amsterdam(North Holland Publ. Co.) 1977, pp. 1133-1142.
 \bibitem{ramsey} F.P. Ramsey, On a problem in formal logic, \textit{Proc. London Math. Soc.} 30(1928)264-286. 
 
 \bibitem{rootselaar54} B. van Rootselaar, {\it Generalization of the Brouwer integral} Thesis. Amsterdam 1954.
 \bibitem{Simpson}S.G.~Simpson, \textit{Subsystems of Second Order Arithmetic},
 Perspectives in Mathematical Logic,  Springer Verlag, Berlin etc., 1999.
 \bibitem{veldman1981}W.~Veldman,  {\it Investigations in Intuitionistic Hierarchy
Theory}, Ph.D.~Thesis, Katholieke Universiteit Nijmegen, 1981. 
\bibitem{veldman1982} W. Veldman, \textit{On the continuity of functions in intuitionistic real analysis, some remarks on Brouwer's paper: `\"Uber Definitionsbereiche von Funktionen'}, Report 8210, Mathematisch Instituut, Katholieke Universiteit Nijmegen, 1982.  



\bibitem{veldman1995}
W.~Veldman,  Some intuitionistic variations on the notion of a
finite set of natural numbers, in: H.C.M.~de Swart, L.J.M.~Bergmans
(ed.), {\it Perspectives on Negation,essays in honour of Johan J.~de
Iongh on the occasion of his 80th birthday}, Tilburg University
Press, Tilburg, 1995, pp.~177-202.

\bibitem{veldman1999}
W.~Veldman,  On sets enclosed between a set and its double
complement, in: A. Cantini e.a.(ed.), {\em Logic and Foundations of 
Mathematics}, Proceedings Xth International Congress on Logic,
Methodology and Philosophy of Science, Florence 1995, Volume III, Kluwer
Academic Publishers, Dordrecht, 1999, pp. 143-154.

 \bibitem{veldman2001} W.~Veldman, Understanding and using Brouwer's Continuity
 Principle, in: U.~Berger, H.~Osswald, P.~Schuster (ed.), \textit{Reuniting the
 Antipodes, constructive and nonstandard views of the continuum, Proceedings of a
 Symposium held in San Servolo/Venice, 1999}, Kluwer, Dordrecht, 2001,
 pp. 285-302.
 \bibitem{veldman2001a} W.~Veldman, Bijna de waaierstelling {\it Almost the Fan Theorem}, \textit{Nieuw Archief voor Wiskunde}, vijfde serie, deel 2(2001), pp. 330-339.
 


\bibitem{veldman2005} W.~Veldman, Two simple sets that are not positively Borel, {\em Annals of Pure and Applied Logic} 135(2005)151-2009.


\bibitem{veldman2006b} W. ~Veldman, Brouwer's Real Thesis on Bars, in: G.~Heinzmann, G.~Ronzitti, eds., {\em Constructivism: Mathematics, Logic, Philosophy and Linguistics,  Philosophia Scientiae, Cahier Sp\'ecial 6}, 2006, pp. 21-39.
\bibitem{veldman2008a} W.~Veldman, The Borel hierarchy theorem from Brouwer's intuitionistic perspective, {\it The Journal of Symbolic Logic}, 73(2008)1-64.


\bibitem{veldman2008}W.~Veldman, Some Applications of Brouwer's Thesis on Bars, in: M.~van Atten, P.~Boldini, M.~Bourdeau, G.~Heinzmann, eds., \textit{One Hundred Years of Intuitionism (1907-2007), The Cerisy Conference}, Birkh\"auser, Basel etc., 2008, pp. 326-340.

\bibitem{veldman2009b} W.~Veldman, The fine structure of the intuitionistic Borel hierarchy, {\it The Review of Symbolic Logic}, 2(2009)30-101.

\bibitem{veldman2011b}W.~Veldman,  Brouwer's Fan Theorem as an axiom and as a contrast to Kleene's Alternative,   \textit{Archive for Mathematical Logic} 53(2014)621-693.


\bibitem{veldman2011c}
W. ~Veldman, The Principle of Open Induction on Cantor space  and the Approximate-Fan Theorem,  August 2014, arXiv 1408.2493.
\bibitem{veldman2019} W.~Veldman, Projective sets, intuitionistically, October 2018, arXiv:1104.3077.

\bibitem{veldman2011d} W. ~Veldman, The Fan Theorem, its strong negation and the determinacy of games, October 2020, arXiv: 1311:6988.                                                                                                         


	
	\bibitem{veldman20b} W. Veldman, Treading in Brouwer's footsteps, in: A. Rezu\c{s} (ed.), {\it Contemporary Logic and Computing}, [Series: {\it Landscapes in Logic}, Volume 1], College Publications, London, 2020, pp. 355-396.
\bibitem{veldman21}W. Veldman, Intuitionism: An Inspiration?, {\it Jahresber. Dtsch. Math. Ver.} 123(2021)221–284.
	\bibitem{veldmanbezem}W. Veldman and M. Bezem, 
	 Ramsey's Theorem and the Pigeonhole Principle in intuitionistic mathematics, 
 	{\it Journal of the London Mathematical Society}  47(1993)193--211.
 
 
 
 \end{thebibliography}
\end{document}